\documentclass[12pt,a4paper]{amsart}
\usepackage[margin=2.5cm,marginpar=2cm]{geometry}

\usepackage{hyperref}
\usepackage[nameinlink]{cleveref}

\usepackage{multirow}
\usepackage[normalem]{ulem}

\usepackage[abbrev]{amsrefs}
\usepackage{amsthm}
\usepackage{graphicx}
\usepackage{braket}
\usepackage[pointedenum]{paralist}
\usepackage[all,cmtip]{xy}
\usepackage{rotating}
\usepackage[rgb,table,x11names]{xcolor}

\usepackage{amssymb}
\usepackage{amsmath}
\usepackage{mathrsfs}
\usepackage{mathrsfs}

\usepackage[abbrev]{amsrefs}

\usepackage[mathlines,pagewise]{lineno}

\definecolor{refkey}{gray}{.3}
\definecolor{labelkey}{gray}{.3}

\newcommand{\rO}{\mathrm{O}}

\newcommand{\pr}{\mathrm{pr}}

\newcommand{\Rad}{\mathrm{Rad}}

\newcommand{\depth}{\mathrm{depth}}

\newcommand{\bfone}{\mathbf{1}}
\newcommand{\piSigma}{\pi_\Sigma}

\newcommand{\eva}{\mathrm{eva}}

\DeclareMathOperator{\Mp}{Mp}

\def\Stab{{\rm Stab}}
\def\ad{{\rm ad}}
\def\id{{\rm id}}
\def\idCx{\id_{\bC^\times}}
\def\tv{\bfone}

\makeatletter
\def\inn#1#2{\left\langle 
      \def\ta{#1}\def\tb{#2}
      \ifx\ta\@empty{\;} \else {\ta}\fi ,
      \ifx\tb\@empty{\;} \else {\tb}\fi
      \right\rangle} 
\makeatother

\makeatletter
\def\binn#1#2{\overline{\inn{#1}{#2}}} 
\makeatother

\def\circckG{{}^\circ \ckG}
\def\cdK{{}^\circ \dotK}

\def\cdrho{{}^\circ \dotrho}
\def\cSigma{{{}^c{\Sigma}}}

\def\cdSigma{{}^\circ \dotSigma}
\def\SC#1{{\bf SC{#1}}$_i$}

\def\innvga#1#2{\inn{#1}{#2}_{V_\Gamma}}
\def\innw#1#2{\inn{#1}{#2}_{W}}
\def\innv#1#2{\inn{#1}{#2}_{V}}
\def\innvp#1#2{\inn{#1}{#2}_{V'}}
\def\innGa#1#2{\inn{#1}{#2}_{\Gamma}}
\def\innGap#1#2{\inn{#1}{#2}_{-\Gamma'}}
\def\binnGa#1#2{\binn{#1}{#2}_{\Gamma}}
\def\binnGap#1#2{\binn{#1}{#2}_{-\Gamma'}}

\def\simrightarrow{\iso}
\def\surj{\twoheadrightarrow}

\newcommand\iso{\xrightarrow{
   \,\smash{\raisebox{-0.65ex}{\ensuremath{\scriptstyle\sim}}}\,}}

\def\abs#1{\left|{#1}\right|}

\def\Sp{{\rm Sp}}

\def\SO{{\rm SO}}
\def\det{{\rm det}}

\usepackage{xparse}
\def\usecsname#1{\csname #1\endcsname}
\def\useLetter#1{#1}
\def\usedbletter#1{#1#1}

\ExplSyntaxOn 

\def\mydefcirc#1#2#3{\expandafter\def\csname
  circ#3{#1}\endcsname{{}^\circ {#2{#1}}}}
\def\mydefvec#1#2#3{\expandafter\def\csname
  vec#3{#1}\endcsname{\vec{#2{#1}}}}
\def\mydefdot#1#2#3{\expandafter\def\csname
  dot#3{#1}\endcsname{\dot{#2{#1}}}}

\def\mydefbar#1#2#3{\expandafter\def\csname bar#3{#1}\endcsname{\bar{#2{#1}}}}
\def\mydefhat#1#2#3{\expandafter\def\csname hat#3{#1}\endcsname{\hat{#2{#1}}}}
\def\mydefwh#1#2#3{\expandafter\def\csname wh#3{#1}\endcsname{\widehat{#2{#1}}}}
\def\mydeft#1#2#3{\expandafter\def\csname t#3{#1}\endcsname{\tilde{#2{#1}}}}
\def\mydefu#1#2#3{\expandafter\def\csname u#3{#1}\endcsname{\underline{#2{#1}}}}
\def\mydefr#1#2#3{\expandafter\def\csname r#3{#1}\endcsname{\mathrm{#2{#1}}}}
\def\mydefb#1#2#3{\expandafter\def\csname b#3{#1}\endcsname{\mathbb{#2{#1}}}}
\def\mydefwt#1#2#3{\expandafter\def\csname wt#3{#1}\endcsname{\widetilde{#2{#1}}}}
\def\mydeff#1#2#3{\expandafter\def\csname f#3{#1}\endcsname{\mathfrak{#2{#1}}}}
\def\mydefbf#1#2#3{\expandafter\def\csname bf#3{#1}\endcsname{\mathbf{#2{#1}}}}
\def\mydefc#1#2#3{\expandafter\def\csname c#3{#1}\endcsname{\mathcal{#2{#1}}}}
\def\mydefsf#1#2#3{\expandafter\def\csname sf#3{#1}\endcsname{\mathsf{#2{#1}}}}
\def\mydefs#1#2#3{\expandafter\def\csname s#3{#1}\endcsname{\mathscr{#2{#1}}}}
\def\mydefcks#1#2#3{\expandafter\def\csname cks#3{#1}\endcsname{{\check{
        \csname s#2{#1}\endcsname}}}}
\def\mydefckc#1#2#3{\expandafter\def\csname ckc#3{#1}\endcsname{{\check{
      \csname c#2{#1}\endcsname}}}}
\def\mydefck#1#2#3{\expandafter\def\csname ck#3{#1}\endcsname{{\check{#2{#1}}}}}

\NewDocumentCommand{\doGreek}{m}
{
  \clist_map_inline:nn {alpha,beta,gamma,Gamma,delta,Delta,epsilon,varepsilon,zeta,eta,theta,vartheta,Theta,iota,kappa,lambda,Lambda,mu,nu,xi,Xi,pi,Pi,rho,sigma,varsigma,Sigma,tau,upsilon,Upsilon,phi,varphi,Phi,chi,psi,Psi,omega,Omega,tG} {#1{##1}{\usecsname}{\useLetter}} 
}

\NewDocumentCommand{\doAtZ}{m}
{
  \clist_map_inline:nn {A,B,C,D,E,F,G,H,I,J,K,L,M,N,O,P,Q,R,S,T,U,V,W,X,Y,Z} {#1{##1}{\useLetter}{\useLetter}} 
}
\NewDocumentCommand{\doatz}{m}
{
  \clist_map_inline:nn {a,b,c,d,e,f,g,h,i,j,k,l,m,n,o,p,q,r,s,t,u,v,w,x,y,z} {#1{##1}{\useLetter}{\usedbletter}} 
}

\NewDocumentCommand{\doallAtZ}{}
{
\clist_map_inline:nn {mydefsf,mydeft,mydefu,mydefwh,mydefhat,mydefr,mydefwt,mydeff,mydefb,mydefbf,mydefc,mydefs,mydefck,mydefcks,mydefckc,mydefbar,mydefvec,mydefcirc,mydefdot} {\doAtZ{\csname ##1\endcsname}}
}
\NewDocumentCommand{\doallatz}{}
{
\clist_map_inline:nn {mydefsf,mydeft,mydefu,mydefwh,mydefhat,mydefr,mydefwt,mydeff,mydefb,mydefbf,mydefc,mydefs,mydefck,mydefbar,mydefvec,mydefdot} {\doatz{\csname ##1\endcsname}}
}

\NewDocumentCommand{\doallGreek}{}
{
\clist_map_inline:nn {mydefck,mydefwt,mydeft,mydefwh,mydefbar,mydefu,mydefvec,mydefcirc,mydefdot} {\doGreek{\csname ##1\endcsname}}
}

\NewDocumentCommand{\doGroups}{m}
{
  \clist_map_inline:nn {GL,Sp,rO,rU,fgl,fsp,foo,fuu,fkk,fuu,ufkk,uK} {#1{##1}{\usecsname}{\useLetter}} 
}
\NewDocumentCommand{\doallGroups}{}
{
\clist_map_inline:nn {mydeft,mydefu,mydefwh,mydefhat,mydefwt,mydefck,mydefbar} {\doGroups{\csname ##1\endcsname}}
}

\doallAtZ
\doallatz
\doallGreek
\doallGroups
\ExplSyntaxOff

\def\hatgamma{\hat{\gamma}}

\def\wtG{\widetilde{G}}

\def\wtSp{\widetilde{\rSp}}

\def\teta{{\widetilde{\eta}}}
\def\tSigma{{\widetilde{\Sigma}}}
\def\tSigmap{{\widetilde{\Sigma}'}}
\def\trho{{\tilde{\rho}}}

\def\biota{{\,\bar{\iota}\,}}

\def\GL{\mathrm{GL}}
\def\wtSp{\widetilde{\mathrm{Sp}}}

\def\fgl{{\mathfrak{gl}}}
\def\fsl{{\mathfrak{sl}}}
\def\fso{{\mathfrak{so}}}

\def\Ker{{\rm Ker}\,}
\def\Lie{{\rm Lie}}
\def\Im{{\rm Im\,}}

\def\Ind{{\rm Ind}}

\DeclareMathOperator{\Ad}{Ad}

\DeclareMathOperator{\Hom}{Hom}
\DeclareMathOperator{\End}{End}
\DeclareMathOperator{\Irr}{Irr}

\def\Norm{\mathrm{Norm}}
\def\Gal{\mathrm{Gal}}

\DeclareMathOperator{\supp}{supp}
\DeclareMathOperator{\val}{val}

\def\lD{\cD}
\def\nD{\mathring{\lD}}
\def\nrho{\mathring{\rho}}

\def\tD{\widetilde{\lD}}

\def\tpiD{\tpi_{\tD}}

\def\etaSigma{\eta_{\Sigma}}
\def\tetaSigma{\teta_{\tSigma}}
\def\etaSigmap{\eta'_{\Sigma'}}
\def\tetaSigmap{\teta'_{\tSigma'}}

\def\tetaD{\teta_{\tD}}
\def\etanD{\eta_{\nD}}
\def\etaaanD{\eta_{\aanD}}
\def\etabbnD{\eta_{\bbnD}}

\newcommand{\trivial}[2][]{\if\relax\detokenize{#1}\relax
{\orange \vspace{0em} The following is trivial:  #2}
\else\ifx#1h\ifcsname showtrivial\endcsname
{\orange \vspace{0em} The following is trivial:  #2}
\fi\else{\red Wrong argument!}\fi\relax\fi
}

\def\floor#1{{\lfloor #1 \rfloor}}
\def\ceil#1{{\lceil #1 \rceil}}
\def\foo{{\mathfrak{o}}}

\def\Hom{{\mathrm{Hom}}}

\def\tgg{{\tilde{g}}}

\def\Stab{{\mathrm{Stab}}}
\def\Sp{{\mathrm{Sp}}}

\def\tpiSigma{\tpi_{\tSigma}}
\def\tpiSigmap{\tpi'_{\tSigma'}}

\def\wteta{{\widetilde{\eta}}}

\def\tpidotSigma{\tpi_{\dot{\tSigma}}}
\def\tpiSigmap{{\tpi'_{\tSigma'}}}

\def\half{{\frac{1}{2}}}
\def\quarter{{\frac{1}{4}}}

\def\fracmm{{\frac{1}{m}}}
\def\fracdmm{{\frac{1}{2m}}}

\def\bpsi{{\overline{\psi}}}

\def\Cent#1#2{{\mathrm{Z}_{#1}({#2})}}

\def\bomega{{\overline{\omega}}}
\def\bomegab{\overline{\omega}_{\bfbb}}
\def\bomegabz{\overline{\omega}_{\bfbbz}}

\def\bomegaww{{\bomega_w}}
\def\bomegaSK{{\bomega_{\SK}}}
\def\bomegadgS{{\bomega_{\dgS}}}
\def\bomegaiiS{{\bomega_{\iiS}}}
\def\bomegaaaS{{\bomega_{\aaS}}}
\def\bomegabbS{{\bomega_{\bbS}}}
\def\bomegadgSK{{\bomega_{\dgSK}}}

\def\bomegaaaSK{{\bomega_{\aaSK}}}
\def\bomegabbSK{{\bomega_{\bbSK}}}

\DeclareMathOperator{\cInd}{c-Ind}

\DeclareFontFamily{U} {MnSymbolC}{}
\DeclareFontShape{U}{MnSymbolC}{m}{n}{
  <-6> MnSymbolC5
  <6-7> MnSymbolC6
  <7-8> MnSymbolC7
  <8-9> MnSymbolC8
  <9-10> MnSymbolC9
  <10-12> MnSymbolC10
  <12-> MnSymbolC12}{}
\DeclareFontShape{U}{MnSymbolC}{b}{n}{
  <-6> MnSymbolC-Bold5
  <6-7> MnSymbolC-Bold6
  <7-8> MnSymbolC-Bold7
  <8-9> MnSymbolC-Bold8
  <9-10> MnSymbolC-Bold9
  <10-12> MnSymbolC-Bold10
  <12-> MnSymbolC-Bold12}{}

\DeclareFontFamily{U} {MnSymbolD}{}
\DeclareFontShape{U}{MnSymbolD}{m}{n}{
  <-6> MnSymbolD5
  <6-7> MnSymbolD6
  <7-8> MnSymbolD7
  <8-9> MnSymbolD8
  <9-10> MnSymbolD9
  <10-12> MnSymbolD10
  <12-> MnSymbolD12}{}
\DeclareFontShape{U}{MnSymbolD}{b}{n}{
  <-6> MnSymbolD-Bold5
  <6-7> MnSymbolD-Bold6
  <7-8> MnSymbolD-Bold7
  <8-9> MnSymbolD-Bold8
  <9-10> MnSymbolD-Bold9
  <10-12> MnSymbolD-Bold10
  <12-> MnSymbolD-Bold12}{}

\DeclareSymbolFont{MnSyC} {U} {MnSymbolC}{m}{n}
\DeclareSymbolFont{MnSyD} {U} {MnSymbolD}{m}{n}

\DeclareMathSymbol{\medstar}{\mathord}{MnSyC}{130}
\DeclareMathSymbol{\boxslash}{\mathord}{MnSyC}{114}
\DeclareMathSymbol{\boxbackslash}{\mathord}{MnSyC}{115}
\DeclareMathSymbol{\smblksquare}{\mathord}{MnSyC}{105}
\DeclareMathSymbol{\nequiv}{\mathord}{MnSyD}{121}

\DeclareMathSymbol{\otimes}{\mathrel}{MnSyC}{97}
\DeclareMathSymbol{\boxtimes}{\mathrel}{MnSyC}{117}

\def\triangle{{\mathrm{\Delta}}}

\def\mstar{{\medstar}}

\def\dd{{\mathrm{d}}}
\def\dalpha{{\dd\alpha}}
\def\dalphap{{\dalpha^\perp}}
\def\odalpha{\overline{\dalpha}}
\def\odalphaprime{\overline{\dalpha'}}

\def\Jump{{\mathrm{Jump}}}

\def\ckfgg{{\check{\fgg}}}
\def\ckfgl{{\check{\fgl}}}
\def\ckGamma{{\check{\Gamma}}}

\def\sspan{{\mathrm{Span}}}

\def\dbG{{\breve{G}}}
\def\dbK{{\breve{K}}}
\def\dbJ{{\breve{J}}}
\def\dbbfW{{\breve{\bfW}}}
\def\dbKp{{\dbK_+}}
\def\dbKzp{{\dbK_{0^+}}}
\def\dbpsi{{\breve{\psi}}}
\def\dbrho{{\breve{\rho}}}
\def\dbkappa{{\breve{\kappa}}}
\def\dbeta{{\breve{\eta}}}
\def\dbbomega{{\breve{\bomega}}}
\def\dbGamma{{\breve{\Gamma}}}
\def\dbsfGz{{\breve{\sfG}^0}}

\def\SK{S_K}

\def\TTidx#1#2{\,{}^{#1}\hspace{-0.1em}#2} 

\def\any{\smblksquare}
\def\mydefTT#1#2#3{
\expandafter\def\csname ii#3{#1}\endcsname{\TTidx{i}{#2{#1}}}
\expandafter\def\csname zz#3{#1}\endcsname{\TTidx{0}{#2{#1}}}
\expandafter\def\csname ll#3{#1}\endcsname{\TTidx{l}{#2{#1}}}
\expandafter\def\csname aa#3{#1}\endcsname{\TTidx{a}{#2{#1}}}
\expandafter\def\csname bb#3{#1}\endcsname{\TTidx{b}{#2{#1}}}
\expandafter\def\csname oo#3{#1}\endcsname{\TTidx{1}{#2{#1}}}
\expandafter\def\csname ss#3{#1}\endcsname{\TTidx{\boxslash}{#2{#1}}}
\expandafter\def\csname dg#3{#1}\endcsname{\TTidx{\boxbackslash}{#2{#1}}}
\expandafter\def\csname any#3{#1}\endcsname{\TTidx{\any}{#2{#1}}}
}
\def\usecsname#1{\csname #1\endcsname}
\def\useLetter#1{#1}
\def\usedbletter#1{#1#1}

\def\cTp{\cT'}
\def\VGa{V_\Gamma}

\def\sLp{\sL'}
\def\Sigmap{\Sigma'}

\def\OmegaK{\Omega_K}

\mydefTT{cTp}{\usecsname}{\useLetter}
\mydefTT{sB}{\usecsname}{\useLetter}
\mydefTT{sL}{\usecsname}{\useLetter}
\mydefTT{tD}{\usecsname}{\useLetter}
\mydefTT{alpha}{\usecsname}{\useLetter}
\mydefTT{Sigma}{\usecsname}{\useLetter}
\mydefTT{tSigma}{\usecsname}{\useLetter}
\mydefTT{Sigmap}{\usecsname}{\useLetter}
\mydefTT{Gamma}{\usecsname}{\useLetter}
\mydefTT{ckGamma}{\usecsname}{\useLetter}
\mydefTT{Gammap}{\usecsname}{\useLetter}
\mydefTT{Omega}{\usecsname}{\useLetter}
\mydefTT{OmegaK}{\usecsname}{\useLetter}
\mydefTT{chi}{\usecsname}{\useLetter}
\mydefTT{phi}{\usecsname}{\useLetter}
\mydefTT{eta}{\usecsname}{\useLetter}
\mydefTT{etanD}{\usecsname}{\useLetter}
\mydefTT{teta}{\usecsname}{\useLetter}
\mydefTT{kappa}{\usecsname}{\useLetter}
\mydefTT{dbkappa}{\usecsname}{\useLetter}
\mydefTT{bomega}{\usecsname}{\useLetter}
\mydefTT{biota}{\usecsname}{\useLetter}
\mydefTT{rho}{\usecsname}{\useLetter}
\mydefTT{nrho}{\usecsname}{\useLetter}
\mydefTT{nnrho}{\usecsname}{\useLetter}
\mydefTT{dbrho}{\usecsname}{\useLetter}
\mydefTT{bfbb}{\usecsname}{\useLetter}
\mydefTT{sL}{\usecsname}{\useLetter}
\mydefTT{sLp}{\usecsname}{\useLetter}
\mydefTT{r}{\useLetter}{\useLetter}
\mydefTT{A}{\useLetter \!}{\useLetter}
\mydefTT{lD}{\usecsname}{\useLetter}
\mydefTT{nD}{\usecsname}{\useLetter}
\mydefTT{tnD}{\usecsname}{\useLetter}
\mydefTT{J}{\useLetter}{\useLetter}
\mydefTT{M}{\useLetter}{\useLetter}
\mydefTT{G}{\useLetter}{\useLetter}
\mydefTT{K}{\useLetter}{\useLetter}
\mydefTT{W}{\useLetter}{\useLetter}
\mydefTT{S}{\useLetter}{\useLetter}
\mydefTT{V}{\useLetter}{\useLetter}
\mydefTT{X}{\useLetter}{\useLetter}
\mydefTT{Vp}{\usecsname}{\useLetter}
\mydefTT{bfW}{\usecsname}{\useLetter}
\mydefTT{dbK}{\usecsname}{\useLetter}
\mydefTT{dbeta}{\usecsname}{\useLetter}
\mydefTT{dbG}{\usecsname}{\useLetter}
\mydefTT{dbJ}{\usecsname}{\useLetter}
\mydefTT{SK}{\usecsname}{\useLetter}
\mydefTT{ckG}{\usecsname}{\useLetter}
\mydefTT{End}{\usecsname}{\useLetter}
\mydefTT{xi}{\usecsname}{\useLetter}
\mydefTT{d}{\useLetter}{\usedbletter}
\mydefTT{x}{\useLetter}{\usedbletter}
\mydefTT{v}{\useLetter}{\usedbletter}
\mydefTT{w}{\useLetter}{\usedbletter}
\mydefTT{h}{\useLetter}{\usedbletter}
\mydefTT{fgg}{\usecsname}{\useLetter}
\mydefTT{fggp}{\usecsname}{\useLetter}

\def\dgsSB{\sS(\dgsB_0)}
\def\bbsSB{\sS(\bbsB_0)}
\def\aasSB{\sS(\aasB_0)}

\def\ssdbkappa{\TTidx{\boxslash}{\dbkappa}}

\def\basB{\TTidx{ba}{\sB}}
\def\absB{\TTidx{ab}{\sB}}
\def\ssiota{\TTidx{\boxslash}{\iota}}

\def\babfbb{\TTidx{ba}{\bfbb}}
\def\abbfbb{\TTidx{ab}{\bfbb}}

\def\baW{{}^{ba}W}
\def\abW{{}^{ab}W}

\def\ggs{\fgg_{x,s}}
\def\ggsp{\fgg'_{x',s}}
\def\ggss{\fgg_{x,s:s^+}}
\def\ggssp{\fgg'_{x',s:s^+}}

\def\fooD{{\foo_D}}
\def\fppD{{\fpp_D}}
\def\fffD{{\fff_D}}
\def\fppF{{\fpp}}
\def\fffF{{\fff}}
\def\fooF{{\foo}}

\def\II{\mathcal{I}}

\def\TT{\mathfrak{T}}

\def\fglDV{{\fgl(V)}}

\def\EndDV{{\End_D(V)}}

\def\Sec#1{\S~#1}

\newcommand{\BTB}[2][]{\if\relax\detokenize{#1}\relax
\cB(#2)
\else 
\cB(#2,#1)
\fi
}
\newcommand{\rBTB}[2][]{\if\relax\detokenize{#1}\relax
\cB_{\rm red}(#2)
\else 
\cB_{\rm red}(#2,#1)
\fi
}

\newcommand{\remove}[1]{\relax}

\newcommand{\tr}{\mathrm{tr}}

\crefformat{equation}{(#2#1#3)}

\crefformat{enumi}{(#2#1#3)}
\crefformat{enumi}{(#2#1#3)}
\Crefformat{enumi}{part (#2#1#3)}
\crefmultiformat{enumi}{(#2#1#3)}{ and~(#2#1#3)}{, (#2#1#3)}{ and~(#2#1#3)}
\Crefmultiformat{enumi}{part~(#2#1#3)}{ and~(#2#1#3)}{, (#2#1#3)}{ and~(#2#1#3)}
\crefrangeformat{enumi}{(#3#1#4) to~(#5#2#6)}
\Crefrangeformat{enumi}{part~(#3#1#4) to~(#5#2#6)}
\crefformat{rmk}{remark~#2#1#3}

\newtheorem*{thmM}{Main Theorem}
\crefformat{thmM}{main theorem}
\Crefformat{thmM}{Main Theorem}
\newtheorem{thm}{Theorem}[section]
\newtheorem{lemma}[thm]{Lemma}

\newtheorem*{lemma*}{Lemma}
\newtheorem{prop}[thm]{Proposition}
\newtheorem{cor}[thm]{Corollary}
\newtheorem{claim}{Claim}
\newtheorem*{claim*}{Claim}
\theoremstyle{definition}
\newtheorem{definition}[thm]{Definition}
\newtheorem*{IndH}{Induction Hypothesis}
\theoremstyle{remark}
\newtheorem*{remark}{Remark}
\newtheorem*{remarks}{Remarks}
\newtheorem*{eg*}{Example}

\crefformat{del}{{\mbox{\red Deleted reference!}}}
\Crefformat{del}{{\mbox{\red Deleted reference!}}}

\def\tpi{\tilde{\pi}}
\def\bCx{\bC^\times}
\def\scG{\hatG_{\text{sc}}}
\def\osD{\overline{\sD}}
\def\DV{\sD_V}
\def\DVp{\sD_{V'}}
\def\bDV{\osD_V}
\def\bDVp{\osD_{V'}}

\def\bDTp{\osD_{\cT'}}
\def\dtheta{\vartheta}
\def\dthetap{\vartheta^+}
\def\dthetaVT{\dtheta_{V,\cT'}}
\def\aacTp{{}^a\cT'}
\def\dthetaaaVT{\dtheta_{\aaV,\aacTp}}
\def\thetaVVp{\theta_{V,V'}}

\def\rhol{\rho_\ell}
\def\rhols{\rho_{\ell^*}}
\def\rholp{\rho_{\ell'}}
\def\rholsp{\rho_{\ell'^*}}

\def\otimesD{\otimes_D}

\def\biotabz{\biota_{\bfbbz}}

\def\trD{{\tr_{D/F}}}
\def\trF{{\tr_F}}

\def\bfbbz{\bfbb_0}
\def\bfbbpp{\bfbb_+}

\def\spt{\xi}
\def\tspt{\tilde{\spt}}
\def\txi{\tilde{\xi}}

\def\pidotSigma{\pi_{\dot{\Sigma}}}

\def\dotGamma{{\dot{\Gamma}}}
\def\dotxi{{\dot{\xi}}}
\def\tdotxi{{\tilde{\dotxi}}}
\def\dotphi{{\dot{\phi}}}
\def\dotrho{{\dot{\rho}}}
\def\dotSigma{{\dot{\Sigma}}}
\def\dottSigma{{\dot{\tSigma}}}
\def\dottpiSigma{\tpi_{\dottSigma}}

\def\varpiD{\varpi_{D}}

\def\varpiF{\varpi_{F}}

\newcounter{myenumi}
\def\savemyenumi{\setcounter{myenumi}{\value{enumi}}}
\def\resumemyenumi{\setcounter{enumi}{\value{myenumi}}}
\newcounter{sdenumi}
\def\savesdenumi{\setcounter{sdenumi}{\value{enumi}}}
\def\resumesdenumi{\setcounter{enumi}{\value{sdenumi}}}

\author{Hung Yean Loke}

\address{Department of Mathematics,
National University of Singapore,
2 Science Drive 2, Singapore 117543}
\email{matlhy@nus.edu.sg}

\author[Jia-Jun Ma]{Jia-Jun Ma}
\address{School of Mathematical Sciences, Shanghai Jiao Tong University
800 Dongchuan RD, Shanghai, China 200240
}

\email{hoxide@sjtu.edu.cn}

\subjclass{22E46, 22E47}

\numberwithin{equation}{section}

\title[Corr. between Supercuspidal Reps.]{Local theta correspondences\\ 
  between
  supercuspidal representations\\
  \medskip
  Correspondances th\^eta locales\\ entre les repr\'esentations supercuspidales
}

\long\def\resp#1{}

\begin{document}

\newpage

\begin{abstract}
  By the works of Yu, Kim and Hakim-Murnaghan, we have a parameterization and
  construction of all supercuspidal representations of a reductive $p$-adic
  group in terms of supercuspidal data, when $p$ is sufficiently large.  In this
  paper, we will define a correspondence of supercuspidal data via
  moment maps and theta correspondences over finite fields. Then we will show
  that local theta correspondences between supercuspidal representations are
  completely described by this notion.  In Appendix B, we give a short proof of
  a result of Pan on ``depth preservation''.

  \vspace{2em}

  \noindent R\'esum\'e. Par les travaux de Yu, Kim et Hakim-Murnaghan, on a un
  param\'etrisatrage et une construction de toutes les repr\'esentations
  supercuspidales d'un groupe r\'eductif $p$-adique en termes de donn\'ees
  supercuspidales, quand $p$ est suffisamment grand.  Dans cet article, nous
  d\'efinirons une correspondance entre les donn\'ees supercuspidales par
  l'interm\'ediaire d'applications moments et de correspondances th\^eta sur des
  corps finis.  Ensuite, nous montrerons que les correspondances th\^eta locales
  entre les repr\'esentations supercuspidales sont compl\`etement d\'ecrites par
  cette notion.  Dans l'Appendice B, nous fournissons une courte d\'emonstration
  d'un r\'esultat de Pan sur ``la pr\'eservation de la profondeur''.
  
\end{abstract}

\maketitle

\smallskip
{\let\thefootnote\relax\footnotetext{\emph{Key words.} local theta correspondence, moment maps,
  supercuspidal representations.}}
{\let\thefootnote\relax\footnotetext{\emph{Mots cl\'es.} Correspondances th\^eta locales, application moment, repr\'esentations supercuspidales.
}}

\tableofcontents

\section{Introduction} \label{sec:introduction}

In this paper, we give an explicit description of the local theta 
correspondences between tamely ramified supercuspidal representations in terms
of the  supercuspidal data developed in \cite{Howe,Yu,Kim,HM}.

\subsection{Notation} \label{sec:notation}

Throughout this paper, we fix a non-Archimedean local field $F$ of
characteristic zero with ring of integers $\fooF$, and finite residual field
$\fffF$. Let ``$\val$'' denote the normalized valuation map such that
$\val(F) =\bZ$. Suppose $E$ is a finite extension of $F$ or the central simple
quaternion division algebra over $F$, let $\foo_E$ denote its ring of integers,
let $\fpp_E$ denote the maximal ideal in $\foo_E$ and let
$\fff_E := \foo_E/\fpp_E$ denote the residue field. We continue to let
``$\val$'' denote the natural extension of valuations to $E$.
When $E = F$, we sometimes omit the subscript.  We fix a non-trivial additive
character $\psi\colon F\rightarrow \bC^\times$ with conductor $\fppF$
(i.e. $\psi|_{\fppF}$ is trivial but $\psi|_{\fooF}$ is non-trivial). Let
$\bpsi$ denote the additive character on $\fffF$ induced by $\psi$.  For a
vector space $\fV$ with an endomorphism $\mstar$, we let
$\fV^{\mstar,\varepsilon}$ denote the $\varepsilon$-eigenspace of $\mstar$ in
$\fV$.

\subsection{The set of data} \label{sec:intro.data}
Let $(D,\tau)$ denote one of the division algebras over $F$ given in
    \Cref{sec:DP} with an $F$-linear
  involution $\tau$. Let $\epsilon \in \set{\pm 1}$ and $\epsilon' = -\epsilon$.
  Let $(V, \innv{}{})$ (respectively $(V', \inn{}{}_{V'})$) denote a right
  $D$-module equipped with an \mbox{$\epsilon$-}Hermitian form $\innv{}{}$
  (respectively \mbox{$\epsilon'$-}Hermitian form $\innvp{}{}$).  Then
  $W := V \otimes_D V'$ is naturally a symplectic space.  Let
$(G,G') = (\rU(V),\rU(V'))$ be an irreducible type~I reductive dual pair in the
symplectic group $\Sp:=\Sp(W)$. For any subset  $E$ of $\Sp$ let $\wtE$ be its
inverse image in the metaplectic $\bCx$-cover $\wtSp(W)$ of $\Sp(W)$. See
\Cref{sec:LT} for more details of the notation.

We assume that $p$ is large enough compared to the sizes of $G$ and $G'$ since
we need the hypotheses in \cite[\S{3.5}]{Kim} to hold. 
We will give a lower bound for $p$ in \Cref{cor:pbdd}. 
We will review the construction of supercuspidal representations for $\wtG$
following \cite{Yu,Kim} in \Cref{sec:SC}. 
Let $\Sigma := (x,\Gamma, \phi,\rho)$ be a supercuspidal datum as in
  \cite{Kim}. We briefly explain the entries
  in $\Sigma$:
  \begin{inparaenum}[(i)]
  \item  $\Gamma$ is a semisimple element
    in $\fgg$ and $G^0 := \Cent{G}{\Gamma}$;
  \item $x$ is a point in the building $\BTB{G^0}$ of $G^0$;
  \item $\phi$ and $\rho$ are certain representations of $G^0_x$. 
  \end{inparaenum}
  See \Cref{def:SC.D} for details. 
Then $\Sigma$ will determine
an open compact subgroup $K\subseteq G$ and an irreducible
$K$-module~$\etaSigma$ and, $\pi_\Sigma := \cInd_K^G \etaSigma$ is a
supercuspidal representation of $G$. By \cite{Kim}, under the assumption that
$p$ is large enough, this construction gives all supercuspidal representations
of $G$.  Let $\DV$ be the set of all supercuspidal data and let $\scG$ be the
equivalence classes of irreducible supercuspidal $G$-modules. In \cite{HM} an
equivalence relation $\sim$ on $\DV$ is defined so that
$\bDV:=\DV/ \sim \ \rightarrow \scG$ given by $[\Sigma] \mapsto [\piSigma]$ is a
bijection. In other words, $\bDV$ parametrizes $\scG$.  In fact, the equivalence
relation is just $G$-conjugacy in our situation (cf. \Cref{def:SC.eq}).

Now we consider the covering group $\wtG$.  It is well known that the cover
$\wtK\twoheadrightarrow K$ splits.  Given a certain splitting
$\xi\colon K \rightarrow \wtK$, we identify $\wtK$ with $K\times \bCx$.  We call
$\tSigma := (\Sigma, \xi) = (x,\Gamma,\phi,\rho,\xi)$ a supercuspidal datum
of~$\wtG$.  Define $\wteta_{\tSigma} := \etaSigma \boxtimes \id_{\bCx}$ which is
an irreducible $\wtK$-module.  Then $\tpiSigma := \cInd_{\wtK}^{\wtG} \wteta_{\tSigma}$ is an irreducible supercuspidal
representation of~$\wtG$.  We will see in \Cref{sec:SC.E.c} that under the
assumption that $p$ is large enough, the construction of~$\tpiSigma$ exhausts
all the irreducible supercuspidal genuine\footnote{Here genuine means
    $\bC^\times\subseteq \wtG$ acts by multiplication.} representations of~$\wtG$. The equivalence relation on the set of data of $\wtG$ could also be deduced from
that of $G$ easily (cf. \Cref{sec:SC.cover}).

\subsection{Statement of the main theorem}\label{sec:mainthm}
We retain the notation in \Cref{sec:intro.data}.  Fix a Witt tower $\cT'$
of $\epsilon'$-Hermitian spaces. The covering group $\wtG$ in the dual pair
$(G,G') = (\rU(V), \rU(V'))$ for all $V'\in \cT'$ are canonically
isomorphic to one another.  Let $\omega$ be the Weil representation of
$\wtSp(W)$ with respect to the character $\psi$ and let
  \begin{equation}\label{eq:RG}
    \sR(\wtG,\omega) := \set{\tpi\in \Irr_{\text{gen}}(\wtG)|
      \Hom_{\wtG}(\omega,
      \tpi)\neq 0}
  \end{equation}
  be the equivalence classes of irreducible smooth genuine $\wtG$-modules which
  could be realized as a quotient of~$\omega$.  Let
  $\thetaVVp\colon \sR(\wtG,\omega)\rightarrow \sR(\wtG', \omega)$ denote the
  theta correspondence map.  

Let $\tpi$ be an irreducible supercuspidal genuine $\wtG$-module. Note that the
$\tpi$-isotypic component $\omega[\tpi]$ of $\omega$ is naturally a $\wtG\times \wtG'$ module, 
say $\omega[\tpi] \cong \tpi\boxtimes \Theta_{V,V'}(\tpi)$ where
$\Theta_{V,V'}(\tpi)$ is a genuine $\wtG'$-module. 
Let 
  \[
  m_{\cT'}(\tpi) = \min\set{\dim_D(V'')|\Theta_{V,V''}(\tpi) \neq 0
    \text{ where } V''\in \cT'}
  \]
  which is called the \emph{first occurrence index} of $\pi$ with respect to the Witt
  tower $\cT'$.
 
It is well known that (cf. \cite[Section~3.IV.4 Th\'eor\`eme principal]{MVW}): 
\begin{enumerate}[(i)]
\item $\Theta_{V,V'}(\tpi)$ is either zero or irreducible.

\item \label{it:mT} $m_{\cT'}(\tpi) \leq 2\dim V + a_{\cT'}$ where $a_{\cT'} =
  \min\set{\dim_DV''|V''\in \cT'}$ is the dimension of the anisotropic kernel in
  $\cT'$ (cf. \cite{Li1989}).  

\item $ \Theta_{V,V'}(\tpi)\neq 0$ if and only if $\dim_D(V') \geq m_{\cT'}(\tpi)$ in which case
  $\thetaVVp(\tpi) = \Theta_{V,V'}(\tpi)$.
	
\item $\thetaVVp(\tpi)$ is supercuspidal if and only if
  $\dim(V') = m_{\cT'}(\tpi)$. In this case, we say that the first occurrence of
  $\tpi$ is at $V'$.
\end{enumerate}
The aim of this paper is to describe the first occurrences of theta lifts of
supercuspidal representations in terms of the supercuspidal data.

Let 
\begin{equation} \label{eqbDTp}
\bDTp = \bigsqcup_{V'\in \cT'} \bDVp.
\end{equation}
Using the moment maps and theta correspondences over finite fields, 
we will define theta lifts of equivalence classes of supercuspidal data in
\Cref{sec:LD}, i.e. we will define a map
\begin{equation}\label{eq:dtheta.0}
\xymatrix@R=0em{
\dthetaVT\colon\; \bDV \;\ar@{^(->}[r]& \bDTp.
}
\end{equation}

Fix a pair of data $(\Sigma, \Sigma')\in \DV\times \DVp$.
There is a canonical splitting
\[
\xymatrix{
\xi_{x,x'} \colon K \times K' \ar[r]& \wtK \times \wtK'
}
\]
constructed from the generalized lattice model (cf. \eqref{eq:SPT}). We always
set $\tSigma = (\Sigma, \xi_{x,x'}|_{K})$ and
$\tSigmap = (\Sigma',\xi_{x,x'}|_{K'})$.

\begin{thmM}\label[thmM]{thm:main} 
\begin{asparaenum}[(i)]
\item \label{it:Main.1} Suppose $\Sigma \in \DV$ and
  $[\Sigma'] := \dthetaVT([\Sigma]) \in \bDVp$ for certain $V'\in \cT'$.  Then
  $\thetaVVp(\tpiSigma) = \tpiSigmap$.
\item \label{it:Main.2}
Conversely, suppose $\thetaVVp(\tpi) = \tpi'$, such that $\tpi$ and $\tpi'$ are
supercuspidal representations. Then there exists $\Sigma\in \DV$ such that $\tpi = \tpiSigma$
and $\tpi' = \tpiSigmap$ where $[\Sigma'] =
\dthetaVT([\Sigma])$ and $\cTp$ is the Witt class of  $V'$.
\end{asparaenum}
\end{thmM}

\begin{remarks}
\begin{asparaenum}[1.]
\item If $\tpi$ is a depth zero supercuspidal representation, then
  $\vartheta_{V,\cT'}(\tpi)$ is essentially constructed in \cite{Pan02J}.
		
  \item After the completion of the first draft of this paper, we received a
    preprint \cite{Pan16} from Pan which describes the theta lifts of certain
    positive depth supercuspidal representations.

\item The main theorem generalizes our earlier results with Savin for epipelagic
  representations \cite{LMS}.
  
\item The construction of $\dthetaVT$ provides a criterion on the
    occurrence of supercuspidal representations by conditions on the isomorphism
    classes of the Hermitian spaces modulo the theta correspondences over finite
    fields.
    On the other hand, for some supercuspidal representations, theta
    correspondences over finite fields do not show up in the descriptions of
    their first occurrences. See \Cref{sec:eg.Gen} for details.

  \item In the proof of \Cref{thm:main}~\cref{it:Main.2}, we need a
    generalization of \cite[Proposition~6.3]{Pan02D} which is proved in
    \Cref{sec:IPD}. This also leads to a simpler proof of Pan's theorem on
    ``depth preservation'' \cite{Aubert, Pan02D}.

\item A similar result in terms of the parametrization developed by Bushnell-Kutzko
\cite{BK} and Stevens \cite{S08} should also be established. We hope to take on
this problem in a future project.
\end{asparaenum}
\end{remarks}

\subsection{Organization of the paper} \label{sec:org} 
In
\Cref{sec:LT} we recall some basic definitions and notations of local theta
correspondences and generalized lattice models.
In \Cref{sec:SC}, we review the definition of supercuspidal data and
the constructions of supercuspidal representations for both linear and covering
groups. 

In \Cref{sec:GoodFact} we define the block decompositions of supercuspidal data
in terms of valuations of eigenvalues. 
In \Cref{sec:LD}, we first review the correspondence for depth zero representations and
then define the lift of a single block supercuspidal datum using the moment
maps.  By taking direct sum, the lift in the general case is defined in the end.

We begin the proof of the main theorem with the single block case in
\Cref{sec:OB1} and \Cref{sec:OB2}.  The geometric structures of moment maps are
studied in \Cref{sec:OB1} and refined $K$-types are constructed in
\Cref{sec:OB2} using these structures. These two sections are the most technical
parts of the paper.

By induction on the number of blocks, we prove part \cref{it:Main.1} of the \Cref{thm:main}
in \Cref{sec:CKtype.gen}. Using \Cref{it:Main.1} and a similar
induction, part \cref{it:Main.2} of the \Cref{thm:main} is finally proved in
\Cref{sec:exhaust}.

In \Cref{sec:AppendixA}, we review the Heisenberg-Weil representations over a finite
field and the special isomorphisms of Yu. These are used freely in
\Cref{sec:OB1,sec:CKtype.gen,sec:exhaust}.  In \Cref{sec:IPD}, we first prove
the generalization of an identity of Pan needed in \Cref{sec:exhaust} and then
finish the paper by giving a quick proof of the ``depth preservation''.

\subsection*{Acknowledgment}
We would like to thank Gordan Savin and Jiu-Kang Yu for inspiring discussions.
We would like to thank Wee Teck Gan and Eitan Sayag for their valuable comments
and encouragement. We also would like to thank the referees for their careful reading and comments, in particular for their call to discuss supercuspidal representations of Metaplectic groups in greater detail. Hung Yean Loke is supported by a MOE-NUS AcRF Tier 1 grant
R-146-000-208-112.  Jia-Jun Ma is partially supported by HKRGC Grant CUHK 405213
during his postdoctoral fellowship in IMS of CUHK and by MOE Tier 1 Grant
R-146-000-189-112 and Shanghai Jiao Tong University Startup Fund WF220407110
during the revision of this paper.

\section{Preliminary: Local theta correspondence} \label{sec:LT} In this
section, we set up some notations and review some facts about the generalized
lattice model of the oscillator representation.

\subsection{Type~I dual pairs and moment maps} \label{sec:DP}
Let $(D,\tau)$ denote a division algebra $D$ over $F$ with an $F$-linear
involution $\tau$  in one of the
following situations:
\begin{asparaenum}[(a)] 
\item $D = F$ and $\tau$ is the identity map;
\item $D$ is a quadratic field extension of $F$ and $\tau$ is the nontrivial
  element in $\Gal(D/F)$;
\item $D$ is the central division quaternion algebra over $F$ and $\tau$ is the main involution.
\end{asparaenum}

\subsubsection{}\label{sec:CG}
Let $\epsilon \in \set{\pm 1}$.
Let $(V, \innv{}{})$ or simply $V$  denote a right $D$-module equipped with an
\mbox{$\epsilon$-}Hermitian form $\innv{}{}$.   
Let $\fglDV := \End_D(V)$ be the Lie algebra of $\GL_D(V)$.
For
$X \in \fgl(V)$, let $X^* \in \fgl(V)$ denote the adjoint of $X$ which is defined by
\[
\innv{Xv_1}{v_2} = \innv{v_1}{X^*v_2} \quad \forall v_1,v_2 \in V.
\]
Then the isometry group of $V$ and its Lie algebra are given by 
\begin{align*}
\rU(V) &= \Set{ g \in \fglDV | g g^* = \id } \quad
\text{and}\\ 
\fuu(V) &= \Set{ X \in \fglDV | X + X^* = 0 } = \fglDV^{*,-1}
\end{align*}
respectively.
We will always view $\rU(V)$ and $\fuu(V)$ as subsets of $\EndDV$.

Let $\trD\colon D\rightarrow F$ be the reduced trace on $D$. We set $\trF := \trD
\circ \tr \colon \fgl(V) \rightarrow F$. Clearly $\trF(X) = \trF(X^*)$. 
Let 
\[
\bB(X,Y) := \half\trF(XY).
\]
 It is the invariant non-degenerate bilinear form on
$\fglDV$ and $\fuu(V)$ which we fix throughout this paper.

\subsubsection{} \label{sec:dualpair} Let $\epsilon' = -\epsilon$ and
$(V', \innv{}{}')$ be a right $D$-module equipped with an
$\epsilon'$-Hermitian sesquilinear form $\innv{}{}'$.  We view $V'$ as a left
$D$-module by $a v = v a^\tau$ for all $a \in D$ and $v \in V'$.  Let
$W = V \otimes_D V'$. We always identify $W$ with $\Hom_D(V,V')$ by
$v\otimes v' \mapsto (v_1 \mapsto v'\innv{v}{v_1})$.  For any
$w \in \Hom_D(V,V')$, let $w^\mstar \in \Hom_D(V',V)$ denote its adjoint
which is defined by
\[
\inn{wv}{v'}_{V'} = \inn{v}{w^\mstar v'}_{V} \quad \forall v \in V, 
v' \in V'.
\]

The $F$-vector space $W$  will be equipped with a symplectic form  $\innw{}{}$ given by
\[
\innw{v_1 \otimes v_1'}{v_2 \otimes v_2'} = 
\trD(\innv{v_1}{v_2} \innvp{v_1'}{v_2'}^{\tau}).
\]

Let $G = \rU(V, \inn{}{}_V)$ and $G' = \rU(V',\inn{}{}_{V'})$. 
The pair $(G,G')$ is called an irreducible reductive dual pair of type~I in
$\Sp(W)$ following Howe. The above construction gives all such pairs when
  $F$ varies (cf. \cite[\S{5}]{Howe0} or \cite[Lecture~5]{LiMin} ).

Let $\fgg := \fuu(V)$ and $\fgg' :=\fuu(V')$ denote the Lie algebras of $G$ and
$G'$ respectively.  For $w \in W$, it is not hard to see that
$w^\mstar w \in \fgg \subseteq \End_D(V)$ and
$w w^\mstar \in \fgg' \subseteq \End_D(V')$.
\begin{definition}
We define the \emph{moment maps} $M \colon W \rightarrow \fgg$ and $M' \colon W
\rightarrow \fgg'$ for the dual pair $(G,G')$ by 
\[
M(w) = w^\mstar w \quad \text{and} \quad M'(w) = w w^\mstar \qquad \forall w\in W.
\] 
\end{definition}

The Lie algebras $\fgg$ and $\fgg'$ act on $W = \Hom_D(V,V')$ by 
\[
(X \cdot w)(v)
  = w(-Xv) \text{ and } (X' \cdot w)(v)
  = X'(w(v)) 
\]
for all $w \in W$, $v \in V$, $X \in \fgg$ and $X' \in \fgg'$.
We leave the proof of the following simple formulas to the readers.

\begin{lemma} \label{lem:inn}
Let $w,w_1,w_2\in W$, $X\in \fgg$ and $X'\in \fgg'$. 
Then
\begin{enumerate}[(i)]
\item $\innw{w_1}{w_2} = \trF(w_1^\mstar w_2)$,
\item $\innw{X\cdot w}{w} = 2\bB(X,M(w))$ and
\item $\innw{X'\cdot w}{w} = 2\bB(X',-M'(w))$. \qed
\end{enumerate} 
\end{lemma}
\trivial[h]{
\begin{asparaenum}[(i)]
\item 
Let $w_i = v_i\otimes v'_i$.
Note that $\Im(w_2) = v'_2 D$.  Therefore
\[
\begin{split}
  \trF(w_1^\mstar w_2) &= \trF(w_2 w_1^\mstar) =
  \trD(\innv{v_2}{w_1^\mstar(v'_2)}) = \trD(\innvp{w_1(v_2)}{v'_2})\\
  &= \trD(\innvp{v'_1\innv{v_1}{v_2}}{v'_2}) = \trD(\innv{v_1}{v_2}^\tau
  \innvp{v'_1}{v'_2}) \\
  & = \trD(\innv{v_1}{v_2}
  \innvp{v'_1}{v'_2}^\tau)
  = \innw{v_1\otimes v'_1}{v_2\otimes v'_2}
\end{split}
\]
\item $\innw{X\cdot w}{w} = \innw{w}{wX} = \trF(w^\mstar wX) = 2\bB(X,M(w))$
\item
  $\innw{X'\cdot w}{w} = \innw{w}{-X'w} = \trF(-w^\mstar X'w) = \trF(-X' w
  w^\mstar)= 2\bB(X',-M'(w))$
\end{asparaenum}
}

\subsection{Lattice functions and Bruhat-Tits Buildings} \label{sec:Lattice}
We recall some well known facts about self-dual lattice functions. We refer to
\cite[\Sec{4}]{LMS} for more details.

\begin{definition} \label{def:latticefn}
  A \emph{lattice function} $\sL$ in $V$ is a function which maps $s\in \bR\sqcup \bR^+$
  to an $\fooD$-lattice $\sL_s$ in $V$ such that
\begin{inparaenum}[(i)]
\item $\sL_s \supseteq \sL_t$ if $s<t$, 
\item $\sL_{s+\val(a)} = \sL_s a$ for all $a\in D^\times$,
\item $\sL_{s} = \bigcap_{t<s} \sL_t$ and,
\item $\sL_{s^+} = \bigcup_{t>s} \sL_t$.
\end{inparaenum}
For a lattice function $\sL$,  we set 
\[
\Jump(\sL) := \set{ r \in \bR | \sL_{r} \supsetneq \sL_{r^+} }.
\] 
For an $\fooD$-lattice $L$ in $V$, we denote its dual lattice
\[
L^* := \set{v\in V | \innv{v}{L} \subseteq \fppD}. 
\]
A lattice
function $\sL$ in $V$ is called \emph{self-dual} if $(\sL_t)^* = \sL_{-t^+}$.
\end{definition}

We always let $\BTB{G}$ denote the (extended) Bruhat-Tits building of $G$.
Then $\BTB{G}$ is naturally identified with the set of self-dual
lattice functions 
(cf. \cite{BS,BT} and \cite[\Sec{4}]{LMS}).  For any $x \in \BTB{G}$, we let
$\sL_x$ denote the corresponding lattice function.  Let $G_x$ denote the
  stabilizer of $x$ in $G$. For $r \in \bR\sqcup \bR^+$ and $r \geq 0$
    (respectively  $r \in \bR \sqcup \bR^+$), we let $G_{x,r}$
    denote the corresponding Moy-Prasad subgroup of $G$ (respectively Lie
    subalgebra of $\fgg$) \cite{MP1, MP2}. For $r < t$, we
  set
	\[
	\fgg_{x,r:t} := \fgg_{x,r}/\fgg_{x,t}.
\]
Let $\sL_x$ and $\sL'_{x'}$ be two self-dual lattice functions in $V$ and $V'$
respectively.  We define a lattice function $\sB_{x,x'}$ on $W = V\otimes_D V'$
by
\begin{equation}\label{eq:Lattice.T}
  \sB_{x,x',t} := (\sL_x\otimes_D \sL'_{x'})_t
  := \sum_{t=t_1+t_2}\sL_{x,t_1}\otimes_{\fooD} \sL'_{x',t_2}.
\end{equation}
Then $\sB_{x,x'}$ is a self-dual lattice function on $W$. We view
$\Hom_{\fooD}(\sL_{x,t},\sL'_{x',t'})$ as a lattice in $\Hom_D(V,V') = W$. Then 
$\sB_{x,x',t} = \bigcap_r \Hom_{\fooD}(\sL_{x,r},\sL'_{x',t+r})$.

Now $(x,x') \mapsto  \sB_{x,x'}$
gives a natural $G\times G'$-equivariant map\footnote{One can show that it is an
  embedding. In fact, the map is a restriction of the natural embedding 
 $\rBTB{\GL(V)}\times \rBTB{\GL(V')}\rightarrow \rBTB{\GL(V\otimes_D V'}$ between
reduced buildings.}
\[
\xymatrix{ \BTB{G}\times \BTB{G'} \ar[r]& \BTB{\Sp(W)}.  }
\]
\trivial[h]{
Here is another conceptually simple proof of the embedding. 
We would like to show that: if $\sB_{x,x'} = \sB_{y,y'}$ then $x= y$ and
$x'=y'$. 

Let $\sD$ be the standard lattice function on $D$ which is the only lattice
function in $D$ up to shifting. 
Note that $\sE_x\colon t \mapsto \innv{\sL_{x,0}}{\sL_x}$ is equal to $\sD$ and 
$\sE_y\colon t\mapsto \innv{\sL_{x,0}}{\sL_y}$ is certain shift of $\sD$, say  
$\sE_{y,t} = \sD_{t+a}$. 

WLOG, assume $\val(D) = \bZ$.  View $\sB_{x,x'}$ and $\sB_{y,y'}$ as lattices in
$\Hom_{D}(V,V')$. Then
$\sN'_{x,s}:= \sB_{x,x',s}(\sL_{x,0}) = \sum_{t} \innv{\sL_{x,0}}{\sL_{x,t}}
\sL'_{x',-t+s} = \sum_{t}\fpp_D^{\ceil{t}} \sL'_{x',-t+s} = \sum_{t}
\sL'_{x',\ceil{t}-t+s} = \sL'_{x',s}$
Similar argument shows that
$\sN'_{y,s}:=\sB_{y,y',s}(\sL_{x,0}) = \sL'_{y',s+a}$ which is the shift of the
self-dual lattice function by $a$.  But
$\sB_{x,x',s}(\sL_{x,0}) = \sB_{y,y',s}(\sL_{x,0})$, hence $\sN'_{y,s}$ is also
self-dual and $\sL'_{x',s}=\sN'_{y,s} = \sN'_{x,s} = \sL'_{y',s}$, i.e. $y=y'$.
Now switch the role of $x$ and $x'$ give $x= x'$, which finished the proof of
the claim.

Remark: the same method also shows that
$\rBTB{\GL(V)}\times \rBTB{\GL(V')}\rightarrow \rBTB{\GL(V\otimes_D V'}$ is an
embedding by fixing any non-degenerate form on $V$ and $V'$.  
}

If it is clear what $x$ and $x'$ are, then we will suppress $x,x'$ and simply
write $\sL = \sL_x$, $\sL' =\sL'_{x'}$ and $\sB = \sB_{x,x'}$. For $s < t$, we
denote 
\[
\sL_{s:t} := \sL_{s}/\sL_t, \quad  \sL'_{s:t} := \sL'_s/\sL'_t \quad  \text{and}
\quad \sB_{s:t} := \sB_s/\sB_t.
\]

\subsection{Generalized lattice model}\label{sec:LM}

\def\HW{\rH(W)}

\def\sSB{{\sS(\sB_0)}}
\def\Sww{{S_{w}}}
\def\bSb{{\bS(\bfbb)}}
\def\bSbz{{\bS(\bfbbz)}}
\def\dgbSb{{\bS(\dgbfbb)}}
\def\ssbSb{{\bS(\ssbfbb)}}

\def\bfSp{\mathbf{Sp}}
\def\bfSH{\mathbf{SH}}

\def\KSp{{\Sp_{\sB}}}
\def\KSpp{{\Sp_{\sB,0^+}}}
\def\wtKSp{{\wtSp_{\sB}}}
\def\wtKSpp{{\wtSp_{\sB,0^+}}}
\def\CSPT{\Xi}

Let $W$ be a symplectic space. Let $\HW = W \times F$ denote the corresponding
Heisenberg group and let $\wtSp(W)$ denote the metaplectic $\bC^\times$-covering
of $\Sp(W)$.

Let $(\omega, \sS)$ or simply $\omega$ denote the oscillator representation of
$\wtSp(W) \ltimes \HW$ with central character $\psi$. We recall below the
definition of the generalized lattice model of the oscillator representation.
See \cite{Wa} or \cite[\Sec{3}]{LMS} for more details.

\subsubsection{} \label{sec:LM.real}
Fix a self-dual lattice function $\sB$ in $W$.  Let $\bfbb := \sB_0/\sB_{0^+}$.
The symplectic form $\inn{}{}_W$ induces a non-degenerate symplectic form on the
$\fffF$-vector space $\bfbb$.  Let $\bfH(\bfbb) = \bfbb \times \fffF$ be the
Heisenberg group defined by $\bfbb$.  Let $(\bomegab,\bSb)$ be the oscillator
representation of $\bfSH(\bfbb) := \bfSp(\bfbb)\ltimes \bfH(\bfbb)$ with
central character $\bpsi$ (cf. \Cref{sec:notation}). See \Cref{sec:HW}.  Let
$\rH(\sB_0) := \sB_0\times \fooF \subseteq \HW$,
$\KSp := \Set{ g \in \Sp(W) | g\sB_0 = \sB_0 }$ and
$\KSpp := \Set{ g \in \Sp(W) |(g-1)\sB_0 \subseteq \sB_{0^+} }$. By an abuse of
notation, we also let $\bomegab$ denote its inflation to $\KSp\ltimes \rH(\sB_0)$
via the natural quotient map.

A generalized lattice model with respect to $\sB_0$ of the oscillator
representation $(\omega, \sS)$ is realized on the following space of functions
\[
\sSB := \Set{f \colon W \to \bSb |
\begin{array}{l}
f \text{ is locally constant and compactly supported,} \\
f(a+w) = \psi(\frac{1}{2}\innw{w}{a})\bomegab(a)f(w)\ \forall a\in \sB_0
\end{array}
}.
\]

Via the generalized
lattice model $\sSB$, we get a splitting
$\xi_{\sB} \colon \KSp \hookrightarrow \wtKSp \subseteq \wtSp(W)$ given by
\begin{equation} \label{eq:latticemodel} 
\left((\omega \circ \xi_{\sB}(k))
    f\right)(w) = \bomegab(k) f(k^{-1}\cdot w) \ \ \forall k \in \KSp, w \in W,
  f \in \sS(\sB_0).
\end{equation}

The splitting $\xi_{\sB} \colon \KSp \rightarrow \wtKSp$ extends to an isomorphism
  \[
  \txi_{\sB} \colon \KSp \times \bC^\times  \simrightarrow \wtKSp\] given by
  $(k,c) \mapsto \xi_{\sB}(k)c$. 

  If there is no fear of confusion, we will write $\omega \circ \xi_{\sB}(k)$ as
  $\omega(k)$ for $k \in \KSp$.  By \cite[Appendix~C]{LMS}, the splitting
  restricted on $\KSpp$ is independent of the choices of $\sB$ and agrees with
  Kudla's splitting.\footnote{We only checked the compatibility of splittings
    for lattice model in \cite[Appendix~C]{LMS}. We still need to check the
    compatibility between generalized lattice model and lattice model.  However
    this is clear by testing on the unique (up to scalar) fixed vector of a
    certain self-dual lattice.  \trivial[h]{ In fact, we only need to check the
      following statement: Fix a self-dual lattice $A$ such that
      $\sB_{0^+}\subseteq A\subseteq \sB_0$, then
      $\omega_\sB(g) \circ T = T\circ \omega_A(g)$ for all
      $g\in \Sp_A \cap \Sp_{\sB,0^+} = \Sp_{\sB,0^+}$. Here
      $T \colon \sS(A)\rightarrow \sS(\sB_0)$ is an intertwining operator,
      $\omega_\sB$ and $\omega_A$ are the corresponding splittings.  In general
      $\omega_{\sB}(g) \circ T = m(g) T\circ \omega_A(g)$ for a certain function
      $m$ on $\Sp_{\sB,0^+}$.  Consider the unique (up to scalar) non-zero
      vector $v$ fixed by $A$. It has support in $A$ under the lattice model
      $\sS(A)$ and obviously it is also $\sB_{0^+} \subseteq A$
      fixed. Therefore, $T(v) \in \sS(\sB_0)_{\sB_0}$ is $\Sp_{\sB,0^+}$ fixed.
      Hence we see that,
      $T(v) = \omega_{\sB}(g) \circ T(v) = m(g) T\circ \omega_A(g)(v) = m(g)
      T(v)$, i.e. $m(g) \equiv 1$ for all $g\in \Sp_{\sB,0^+}$.}

      Alternatively, one can prove this using the fact that the first and second
      cohomologies of a pro-$p$ group taking values in a $2$-group is trivial
      when $p\neq 2$. See \cite[Proposition~2.3]{HW}.
		
      We warn that the canonical splitting does not extend to the union
      $\bigcup_{\sB \in \BTB{\Sp}} \KSp$}.  In particular we have following
    canonical splitting on the pro-unipotent part of $\Sp$:
\begin{equation}\label{eq:CSPT}
\CSPT\colon \bigcup_{\sB \in \BTB{\Sp}} \KSpp \longrightarrow \bigcup_{\sB\in
  \BTB{\Sp}} \wtKSpp.
\end{equation}

For any subset $\Omega \subseteq W$ and any element $w\in W$, we set
\[
\sSB_\Omega := \Set{f\in \sSB | \supp(f) \subseteq \Omega + \sB_0}
\]
and $\sSB_w := \sSB_{\set{w}}$.

Suppose $\sB =\sB_{x,x'}$ where $(x,x')$ is a pair of points in $\BTB{G}\times
\BTB{G'}$.
Then $G_x\times G'_{x'}\subseteq \KSp$. The restriction of $\xi_{\sB}$ gives a 
splitting
\begin{equation} \label{eq:SPT} \xymatrix@R=0em{
    \xi_{x,x'}:=\xi_{\sB}|_{G_x\times G'_{x'}} \colon G_x \times G'_{x'} \ar[r]&
    \wtG_x \times \wtG'_{x'}.  }
\end{equation}
of the covering $\wtG_{x}\times \wtG'_{x'}\twoheadrightarrow G_x \times G'_{x'}$.
The restriction of $\xi_{x,x'}$ to the subgroup
$K\times K'\subseteq G_x\times G'_x$ (still called $\xi_{x,x'}$) is the
canonical splitting we referred to in \Cref{sec:mainthm}.

\subsubsection{} We now study a subspace of $\sS$ as an induced
representation which plays a key role later in this paper.

Fix an element $w \in W$ and let 
\[
\Sww := \Stab_{\KSp}(w+\sB_0) = \Set{h\in \KSp | h\cdot w -w \in \sB_0}.
\]
The evaluation at $w$ given by $f \mapsto f(w)$ induces an isomorphism
$\sSB_{w} \iso \bSb$.  Clearly $\Sww$ acts on $\sSB_{w}$ which translates to an
action on $\bSb$.  We will denote the resulting $\Sww$-action on $\bSb$ by
$\bomegaww$.

\begin{lemma} \label{lem:Saction}
The group $\Sww$ acts on $\bSb$ by
\begin{equation}\label{eq:OS}
\begin{split}
\bomegaww(h) &:= \bomegab(h) \bomegab(h^{-1}\cdot w-w) \psi(\half \innw{w}{h^{-1}w-w})\\
& =\bomegab(h) \bomegab(h^{-1}\cdot w-w) \psi(\half \innw{h\cdot w-w}{w}) 
\end{split}
\end{equation}
for all $h\in \Sww$.  \trivial[h]{ Here $\bomegab$ is the action of
  $\bfSH(\bfbb)$ on $\bSb$.}

Let $H$ be a subgroup of $\KSp$ and $S := \Stab_H(w+\sB_0) = H\cap \Sww$.
We have an isomorphism of $H$-modules
\[
\xymatrix{
\TT \colon \sSB_{H\cdot w + \sB_0} \ar[r]^<>(.5){\sim} & \cInd_S^{H} \sS(B_0)_w \cong 
\cInd_S^{H} \bomegaww
}
\]
given by $(\TT (f))(k) = (\omega(k)f)(w)$ for all $k \in H$.
\end{lemma}

\begin{proof}
Let $h\in \Sww$. Then $h^{-1}\cdot w - w \in \sB_0$. Hence, for any $f\in \sSB$, 
\[
\begin{split}
(\omega(h) f)(w) & =  \bomegab(h) f(h^{-1}\cdot w) = \bomegab(h) f(w + (h^{-1} \cdot
w -w)) \\
 & =  \bomegab(h) \bomegab( h^{-1} \cdot
w -w)  \psi(\half\innw{w}{h^{-1} \cdot
w -w})f(w) \\
 & =  \bomegaww(h) f(w).
\end{split}
\]
Observe that $\innw{w}{h^{-1}\cdot w - w}
 = \innw{w}{h^{-1}\cdot w} = \innw{h\cdot w}{w} = \innw{h\cdot w-w}{w}$. The
 second equality in \eqref{eq:OS} follows. 

Note that $h \mapsto h\cdot w +\sB_0$ defines a bijection
$H/S \cong (H\cdot w + \sB_0)/\sB_0 \subseteq W/\sB_0$. 
Hence
$\sSB_{H\cdot w + \sB_0} = \sspan\set{ \omega(H) \sSB_{w}}$ and $\TT$ is an isomorphism.
\end{proof}

\begin{remarks}
It is easy to see that \eqref{eq:OS} could be simplified greatly in some cases. 
\begin{asparaenum}[1.]
\item \label[rmk]{rmk:Sact.1}
Suppose that $w \in \sB_{-s}$ for certain $s>0$. 
Let $h := (g,g') = (\exp(X),\exp(X')) \in G_{x,s}\times G'_{x',s}$ where
$(X,X')\in \fgg_{x,s}\oplus \fgg'_{x',s}$. 
Then 
\begin{equation}\label{eq:Sym.s}
\begin{split}
 &\half \innw{w}{h^{-1}\cdot w - w} = \half \innw{h\cdot w - w}{w}\\
& \equiv \half \innw{(X,X')\cdot w + \half
    (X,X')\cdot (X,X')
    \cdot w}{w} \qquad \pmod{\fpp} \\
  &= \half \innw{(X,X')\cdot w}{w} - \frac{1}{4} \innw{(X,X')\cdot
    w}{(X,X')\cdot w}\\
  &=  \half \innw{X\cdot w}{w} + \half \innw{ X'\cdot w}{w} \\
  &=  \bB(X,M(w)) + \bB(X',-M'(w))  \qquad \text{(by \Cref{lem:inn}).} 
  \end{split}
\end{equation}
This immediately implies
$\psi(\half \innw{h\cdot w - w}{w}) = \psi_{M(w)}(g) \psi_{-M'(w)}(g')$ 
(see \eqref{eq:psi.Ga} for the definition of $\psi_{M(w)}$) and  
\begin{equation}\label{eq:bomegaS.s}
\bomegaww(h) = \bomegab(-(X,X')\cdot w)\psi_{M(w)}(g) \psi_{-M'(w)}(g').
\end{equation}

\item \label[rmk]{rmk:Sact.2} Suppose $h = (g,g') = \exp(X,X')\in G_{x,s^+}\times G_{x',s^+}$. Then $(X,X') \cdot w \in \sB_{0^+}$ and 
  \eqref{eq:bomegaS.s} could be further simplified into 
\begin{equation}\label{eq:bomegaS.sp}
\bomegaww(h) = \psi_{M(w)}(g) \psi_{-M'(w)}(g').
\end{equation}
\end{asparaenum}
\end{remarks}

\section{Preliminary: Supercuspidal representations}\label{sec:SC}

In this section, we will first review the parametrization of tamely ramified
supercuspidal representations for classical groups $G$ when $p$ is sufficiently
large. Then we will extend the notion to the covering groups $\wtG$.  We follow
closely the notations and formulation in~\cite{Kim}.

\subsection{Residue characteristics}
We assume that the residue characteristics $p$ is large enough compared to the
size of $G$ so that all the hypotheses in \cite[\S{3.4}]{Kim} hold. In this
subsection, we will find a lower bound for $p$.

Let 
\begin{asparaenum}[(i)]
\item  $e_D$ be  the absolute
  ramification index of $D/\bQ_p$ if $D$ is a field, or 
\item  $e_D = 2 e_F$ if
  $D$ is the quaternion division algebra over $F$.
\end{asparaenum}

\begin{prop}\label{prop:pbd1}
Suppose $V$ is an $\epsilon$-Hermitian space over $D$ such that $n:=\dim_D(V)$. 
Kim's hypotheses \cite[\Sec{3.4}]{Kim} are satisfied for $\rU(V)$ if
\begin{equation} \label{eqpbdd}
p \geq \max\set{2n +1, e_D n + 2}.
\end{equation}
\end{prop}

\begin{proof}
We check each of Kim's hypotheses (Hk), (HB), (HT) and (HN): 
\begin{enumerate}[1.]
\item (Hk.1) requires the exponential map to be well-defined on $\fgg_{x,0^+}$, which is
  ensured by $p \geq e_D n + 2$ \cite[Section B.1]{DebackerReeder}. \trivial[h]{First we normalize the valuation
  so that $\val(\bQ_p) = \bZ$. The building of $\GL_n(D)$ is
    identified with set of lattice function and it suffices to consider the
    barycenters of facets.  Let $\depth_x(X)$ denote the depth of $X$ with respect
    to a barycenter $x$. 
    We would like to show that $\depth_x(X^k/k!)\to \infty$ when $k\to \infty$. 
  
   Suppose $k = \sum a_j p^j$ is the $p$-adic expansion of $k$. 
    Then $\val(k!)  = \sum_{i\geq 1} \floor{k/p^i} =  (k- \sum a_j) / (p-1) \leq
    (k-1)/(p-1)$. 
    
    On the other hand, $\depth(X) \geq 1/n e_D$ for all
    $X\in \fgl_{x,0^+}$. Therefore
    $k\, \depth_x(X) - \val(k!) \geq k/n\, e_D - (k-1)/(p-1)
    \to \infty$ when $k\to \infty$.} 
		
\item (Hk.2) translates to $p \geq e_D n +2$ for $G\subset\GL_n(D)$. 
\trivial[h]{In fact the condition is ensured by
    $\depth_x(X^{p-1}/p) \geq -1 + (p-1)/n\, e_D >0$. }
		
\item  (HB) holds for classical groups when $p\neq 2$ since it holds for $\GL$ and
  classical group is the fixed point of an involution. We would like to thank
  J. Adler for the discussion.

\item  (HT) holds by the Howe factorization (cf.  \Cref{prop:GD}).
  
\item  (HN) consists of the set of hypotheses in \cite[\Sec{4.2}]{DB.Nil}.
  Hypothesis~4.2.3 holds for $p \geq 2n+1$. \trivial[h]{For $\GL_n(D)$, the
    highest weight under $\fsl(2)$-triple action is $2(n-1)$. Hence
    $\ad(X)^{2(n-1)+1} = 0$, so the condition is
    $2(n-1)+1\leq p-2 \Leftrightarrow 2n+1 \leq p$. } Hypotheses~4.2.1 holds by
  Hypotheses~4.2.3 in characteristic zero case (see \cite[Appendix~A]{DB.Nil}).
  Hypothesis~4.2.4 and Hypothesis~4.2.5 hold since $F$ is characteristic zero.
  Hypothesis~4.2.7 holds for the exponential map by (Hk). \trivial[h]{The
    assumption is ensured by
    $\depth_x(X^k/k!) \geq (k-1)(\frac{1}{n\,e_D} -\frac{1}{p}) + \depth_x(X)>
    \depth_x(X)$
    when $k>1$.  }
  \end{enumerate}
	This proves the proposition.
\end{proof}

\begin{cor} \label{cor:pbdd} 
Let $(G,G') :=(\rU(V), \rU(V'))$ be a type~I dual
  pair with $n:=\dim_DV$.  Let~$\tpi$ and $\tpi'$ be irreducible supercuspidal
  genuine representations of $\wtG$ and $\wtG'$ respectively such that
  $\tpi' = \theta_{V,V'}(\tpi)$. Then Kim's hypotheses in \cite[\Sec{3.4}]{Kim}
  are satisfied for $\rU(V)$ and~$\rU(V')$ if
\[
p \geq \max \set{4n+9, e_D (2n+4) + 2}.
\]
\end{cor}  
\begin{proof}    
  By \Cref{sec:mainthm}~\cref{it:mT}
  $\dim_D V' \leq 2 \dim_D V + a_{\cT'} \leq 2n + 4$. Then $p$ satisfying
    the inequality in the corollary will satisfy \eqref{eqpbdd} for both
    $\rU(V)$ and $\rU(V')$.
\end{proof}

\subsection{Good factorization} \label{sec:RV} Let $\Gamma\in \fgg$ be a
semisimple element.  We say that $\Gamma$ is \emph{tamely ramified} if $\Gamma$
lies in a Cartan subalgebra $\ftt$ which splits over certain tamely ramified
extension $E$ of $F$.  Let $\depth \colon \fgg \rightarrow \bQ$ denote the depth
function given by
\[
  \depth(X) = \sup_{x \in \BTB{G}} \set{ r | X \in \fgg_{x,r} \setminus
    \fgg_{x,r^+} } \quad \forall X \in \fgg.
\] 
We say that $\Gamma$ is \emph{good} or \emph{$G$-good} if for
    every root $\alpha$ of $\fgg(E) := \fgg \otimes_F E$ with respect to
    $\ftt(E)$, $\dalpha(\Gamma)$ is either zero or has valuation
    $\depth(\Gamma)$. See \cite{AD} and \cite[\Sec{2}]{KM1}.
\begin{definition}\label{def:GD}
  Suppose $\Gamma$ is a tamely ramified semisimple element in $\fgg$ with depth
  $-r<0$.  A decomposition of $\Gamma = \sum_{i=-1}^{d}\Gamma_i$ in $\fgg$ is
  called a \emph{good factorization} if the following hold:
\begin{enumerate}[(a)]
\item $\set{\Gamma,\Gamma_d,\cdots, \Gamma_{-1}}$ is a set of commuting semisimple elements in $\fgg$;

\item $\depth(\Gamma_{-1})\geq 0$ and we set $r_{-1} =0$;

\item If $0 \leq i <d $, then $\Gamma_i$ is a good element and 
 $-r_i := \depth(\Gamma_i) <0$;

\item $\Gamma_d \in \Cent{}{\fgg}$;

\item If $\Gamma_d = 0$ (called Case~I), then $-r_{d-1} < \cdots < -r_0 <0$ and we 
   set $r_d := r_{d-1} = r$;
   
\item If $\Gamma_d \neq 0$ (called Case~II), then $-r_d < -r_{d-1} < \cdots < -r_0 <0$ where $r_d := r = -\depth(\Gamma_d)$. 
\end{enumerate}

Fix a good factorization of $\Gamma$ as above. 
We define $G^d = G$ and $G^i = \Cent{G^{i+1}}{\Gamma_i}$ for $0 \leq i \leq d-1$.
\end{definition}

\begin{remarks}
\begin{asparaenum}[1.]
\item Good factorization of $\Gamma$ exists. It is not unique but the set $\set{G^i:  0\leq i \leq d}$ are independent of the choice of the good factorization
(cf. \cite[Prop. 4.7]{Kim}).


\item \label[rmk]{rmk:No.Pos}
By \cite[Remarks~5.10]{Kim}, $\Gamma_{-1}$ plays no role in the construction of
supercuspidal data.
In general, we always assume $G^0 = \Cent{G}{\Gamma}$. For example, this could
be achieved via replacing $\Gamma$ by $\Gamma - \Gamma_{-1}$. By the argument in
\Cref{sec:Efactor}, the condition $G^0 = \Cent{G}{\Gamma}$ is equivalent to $\Gamma_{-1} \in
\Cent{}{\fgg^0} = F'[\Gamma_d,\cdots, \Gamma_0]$ for one (and so for
any) good factorization $\Gamma = \sum_{i=-1}^d\Gamma_i$ where $F':=\Cent{}{D}$
is the center of $D$.
\end{asparaenum}
\end{remarks}

\subsection{Tamely ramified supercuspidal representations
  for classical groups}
We now quickly review the notion of supercuspidal data and the constructions of
supercuspidal representations.

We only study classical groups that appear in Type~I dual pairs. Let $G$ be such a
classical group. In this case, the center $\Cent{}{G}$ is anisotropic so the
reduced building and the extended building of $G$ are the same. Therefore we
will use $x$ instead of its image $[x]$ in the reduced building
(cf. \cite{Yu}).

Under our assumption that $p$ is big enough, the exponential map $\exp$ is well defined on $\bigcup_{x\in \BTB{G}} \fgg_{x,0^+}$.
Let $\log$ denote the inverse map whenever it makes sense. 

\subsubsection{}
For an element $\Gamma\in \fgg$, we define a function
$\psi_{\Gamma}$ on the domain
of $\log$ by 
\begin{equation}\label{eq:psi.Ga}
\psi_\Gamma(g):=\psi(\bB(\log(g),\Gamma)).
\end{equation}
\begin{definition} \label{def:SC.D} A \emph{supercuspidal datum} for $G$ is a
  tuple $\Sigma = (x, \Gamma, \phi, \rho)$ satisfying the following conditions:
\begin{enumerate}[(a)]
\item \label{it:SC.D.Ga} $\Gamma$ is a tamely ramified semisimple element in
  $\fgg$ which admits a good factorization $\Gamma = \sum_{i=-1}^d \Gamma_i$
  such that $\Gamma_{-1} \in F'[\Gamma_d,\cdots, \Gamma_0]$;

\item \label{it:SC.D.aniso} The center $\Cent{}{\circG^0}$ of the
    connected component $\circG^0$ of $G^0:=\Cent{\Gamma}{G}$ is
  anisotropic\footnote{For a general connected reductive group $G$, this
    condition would be that $\Cent{}{G^0}/\Cent{}{G}$ is anisotropic. However
    this is equivalent to \cref{it:SC.D.aniso} in our case.};

\item The point $x$ is a vertex in $\BTB{G^0}$, i.e. the connected
    component of $G^0_{x}:=\Stab_{G^0}(x)$ is a maximal parahoric subgroup in
  $G^0$;
	
\item $\phi \colon G^0_x \rightarrow \bC^\times$ is a character such that
  $\phi|_{G^0_{x,0^+}} = \psi_\Gamma|_{G^0_{x,0^+}}$.  Note that $G^0_{x,0^+}$
  is the pro-$p$ unipotent radical of $G_x^0$;

\item \label{it:data.cusp} $\rho$ is an irreducible cuspidal representation of
  the finite group $\sfG^0_x := G^0_{x}/G^0_{x,0^+}$.  \savesdenumi
\end{enumerate}
We define the \emph{depth} of the datum $\Sigma$, denoted by $\depth(\Sigma)$, to be
$\max\set{-\depth(\Gamma), 0}$.
Note that if $\Sigma$ is a depth zero data, then $\Gamma\in \Cent{}{\fgg}$ by definition.
\end{definition}

\begin{remark}
If we only require that the $\sfG^0_x$-module $\rho$ in \cref{it:data.cusp} is
irreducible but not necessarily cuspidal, then we
call the tuple $(x,\Gamma,\phi,\rho)$ a \emph{(refined) $K$-type datum}. We will use such $K$-type data
in \Cref{def:LDC} and \Cref{sec:exhaust}.
\end{remark}

\subsubsection{} \label{sec:Ktype}
Let $\Sigma = (x,\Gamma, \phi,\rho)$ be a supercuspidal datum. 
We fix a good factorization of $\Gamma = \sum_{i=0}^d \Gamma_i$. 
Since $\Cent{}{\circG^0}$ is anisotropic, there are canonical embeddings of buildings 
\[
\xymatrix{
\BTB{G^0} \ar@{^(->}[r] & \BTB{G^1} \ar@{^(->}[r] & \cdots \ar@{^(->}[r] &
\BTB{G^{d-1}} \ar@{^(->}[r] & \BTB{G^d}.
}
\]

We now define some notations and review the construction of supercuspidal $G$-module
 $\piSigma$ attached to $\Sigma$. 
These notations will be used freely in the rest of the paper. 

\begin{definition}\label{def:SC.N}
Let $\Sigma$ be a supercuspidal data. We set
\begin{enumerate}[(a)]
\item $s_i := r_i/2$,

\item $K^i := G^0_{x} G^1_{x,s_0} \cdots G^i_{x,s_{i-1}}$,

\item $K_{0^+}^i := G^0_{x,0^+} G^1_{x,s_0} G^2_{x,s_1} \cdots G^i_{x,s_{i-1}} =
  K^i\cap G_{x,0^+}$,

\item $K_+^i  := G^0_{x,0^+} G^1_{x,s_0^+} G^2_{x,s_1^+} \cdots
  G^i_{x,s_{i-1}^+}$,

 \item $K := K^d$, $K_{0^+} := K_{0^+}^d$ and $K_+ := K_+^d $.


 \item \label{it:SC.N.6} The character $\phi$ extends to a character of $G_x^0K_+$ by setting
   $\phi|_{K_+} = \psi_{\Gamma}$. By an abuse of notation, we still denote it by
   $\phi$. 

\item 
Let $\kappa^i$ be the canonically constructed irreducible $K^i$-module such that $\kappa^i|_{K_+^i}$ is $\psi_{\Gamma}|_{K_+^i}$-isotypic. See \Cref{sec:kappa} for
the precise definition. 
Let $\kappa := \kappa^d$. 

\item \label{it:SC.N.8} Let $\etaSigma := \rho\otimes \kappa$, which is an irreducible $K$-module. Here
  $\rho$ is identified with its inflation to $G_x^0$.

\item \label{it:SC.N.9} Let $\pi_\Sigma := \cInd_K^G \etaSigma$. 
\end{enumerate}
\end{definition}

Suppose $G$ is a connected reductive group. Yu proves that $\pi_\Sigma$ is an irreducible supercuspidal representations of $G$ \cite{Yu}.
If the residue characteristic of
$F$ is big enough (see \eqref{eqpbdd}), Kim proves that the set of $\pi_\Sigma$ exhausts all the 
supercuspidal representations of connected $G$~\cite{Kim}. 

Note that every odd orthogonal group is a direct product of a special orthogonal
group with $\set{ \pm 1 }$.  Hence, the above results of Yu and Kim, as well
as those of Hakim-Murnaghan in \Cref{sec:HM} below, extend to odd orthogonal groups.  
We will show in \Cref{sec:even_orth} that they also extend to even orthogonal groups.

We call $\piSigma$ the supercuspidal representation of $G$ constructed from the
datum~$\Sigma$.


\subsubsection{} 
\label{sec:HM}
We now describe the equivalence relation on supercuspidal data. 
\begin{definition} \label{def:SC.eq} Let $\Sigma = (x, \Gamma, \phi, \rho)$ and
  $\dotSigma = (\dotxx, \dotGamma, \dotphi, \dotrho)$ be two supercuspidal data.
  We say that $\Sigma$ and $\dotSigma$ are \emph{equivalent} with each other if there exists
  an element $g \in G$ such that
\begin{enumerate}[(a)]
\item \label{it:SC.eq.a} $x = g\cdot\dotxx$,
\item \label{it:SC.eq.b} $\Ad_g(\dotGamma) \in \Gamma+ \fgg_{x,0}$ and,
\item \label{it:SC.eq.c} $\dotrho\otimes \dotphi\cong (\rho\otimes \phi) \circ \Ad_{g}$ as
  $G^0_{\dotxx}$-modules\footnote{Note that (a) and (b) imply
    $G^0_{x}\cong \Ad_g ( G^0_{\dotxx})$}.
\end{enumerate}
\end{definition}

\begin{remark} Since we assume \Cref{def:SC.D}~\cref{it:SC.D.Ga} in the
    definition of supercuspidal data, we may further assume
    $\Ad_g(\dotGamma) \in \Gamma + (\Cent{}{\fgg^0}\cap \fgg^0_{x,0})$ in
    \Cref{def:SC.eq}~\cref{it:SC.eq.b} thanks to \cite[Lemma~5.1.3~(3)]{KM2}.
  \trivial[h]{ Up to $G$-conjugation, we assume
    $\dotGamma \in \Gamma +\fgg_{x,0}$ and $x = \dotxx$.  Recall the
    ``topological Kostant section lemma'' \cite[Lemma~5.1.3~(3)]{KM2}: For a
    good element $\ckGamma$ with depth $-r$,
    \[ {}^{G_{x,r}}(\ckGamma+ X + \ckfgg_{x,0}) = \ckGamma + X + \fgg_{x,0} \qquad
      \forall X\in \ckfgg_{x,-r}\cap \ckfgg_{-r^+}.
    \]
    Let $\Gamma = \sum_{i}\Gamma_i$ be a good decomposition of $\Gamma$. Apply the
    Kostant section lemma to the good factorization, we have
    \[ {}^{\vecG_{x,\vecrr}}(\Gamma + \fgg^0_{x,0}) = \Gamma + \fgg_{x,0}.
    \]
    Hence we can assume $\dotGamma \in \Gamma+ \fgg^0_{x,0}$ up to
    $\vecG_{x,\vecrr} \subset \vecG_{x}$-conjugation.  On the other hand,
    $\Cent{G}{\dotGamma} = G^0 =  \Cent{G}{\Gamma}$ by assumption. Hence
    $\dotGamma - \Gamma \in \Cent{}{\fgg^0}\cap \fgg^0_{x,0} \subset F'[\Gamma] \cap
    \fgg^0_{x,0}$.}
  On the other hand, a depth zero data $(x, \Gamma, \phi,\rho)$ is always
  equivalent to $(x,0, \bfone, \rho\otimes \phi)$ which is considered as a
  typical representative of the equivalence class.
\end{remark}

Let $\Sigma$ and $\dotSigma$ be two supercuspidal data.  
Hakim and Murnaghan show that $\piSigma$ and $\pidotSigma$ are isomorphic if and only if $\Sigma$ and $\dotSigma$ differ by an elementary transform, conjugation and
refactorization. \Cref{def:SC.eq} could be read off from \cite[Lemma~6.4, Theorem~6.7]{HM}
by observing that, in our situation, 
\begin{inparaenum}[a).]
\item   $\Cent{}{G}$ is anisotropic so there is no elementary transform;
\item the refactorization corresponds to a refactorization of the semisimple element $\Gamma$ in terms of $G^i$-good elements, so the notion of  ``refactorization'' also could be suppressed.
\end{inparaenum}
We now record their theorem as follows. 

\begin{thm}[Hakim-Murnaghan] \label{thm:HM} Suppose $G$ is a connected
    classical group, a special orthogonal group or an odd orthogonal group.
  Let $\Sigma$ and $\dotSigma$ be two supercuspidal data for $G$.  Then
  $\piSigma \simeq \pidotSigma$ if and only if $\Sigma$ and $\dotSigma$ are
  equivalent with each other. \qed
\end{thm}

We record following easy consequence of the equivalence of data.

\begin{lemma}\label{lem:mult1}
  Suppose $G$ be as in \Cref{thm:HM}.  Let $\Sigma = (x, \Gamma, \phi, \rho)$ be
  a supercuspidal data for $G$ and let $\kappa = \kappa_\Sigma$.  Then the
  multiplicity space
\begin{equation}\label{eq:mult}
\rho':= \Hom_{K_{0^+}}(\kappa, \piSigma)
\end{equation}
is isomorphic to $\rho$ as $\sfG_x^0$-modules. 
\end{lemma}

\begin{proof}
Let $\dotrho$ be an irreducible component of $\rho'$.
  By \cite[Proposition~17.2]{Kim}, $\dotrho$ is a cuspidal $\sfG_x^0$-module. By Frobenius reciprocity and Yu's
  construction, $\piSigma \cong \pidotSigma$ where
  $\dotSigma = (x,\Gamma, \phi, \dotrho)$. Now by \Cref{thm:HM}, $\rho\cong
  \dotrho$. Hence, $\rho'$ is $\rho$-isotypic with multiplicity 
  \[
  m_\rho = \Hom_{K}(\rho, \Hom_{K_{0^+}}(\kappa,
\piSigma)) =  \Hom_K(\etaSigma, \piSigma) = \Hom_G(\piSigma,\piSigma) = 1.
  \]
\end{proof}

\subsection{Even orthogonal groups} \label{sec:even_orth} 
We now show that the
results of Yu, Kim and Hakim-Murnaghan extend to the even orthogonal groups.
The contents in this subsection is well known to the experts. We include the proofs for
completeness.

Let $V$ be an even dimensional quadratic space. Let $G = \rO(V)$ and
$\circG = \SO(V)$. 

Suppose $V$ is anisotropic, then there is nothing to prove. Suppose $V$ is a two
dimensional hyperbolic space, then $\rO(V) = \rO(1,1)$.  The subgroup $\SO(1,1)$
in $\rO(1,1)$ is defined to be a parabolic subgroup since it is the stabilizer
of an isotropic subspace. This implies that all representations of $\rO(1,1)$
are non-supercuspidal.  Therefore it suffices to consider $\dim V >2$.

For any subgroup $H$ of $G$, we denote $\circH := \circG \cap H$.  For
a subgroup $H$ of $G$ and a $H$-module $\tau$, we will let ${}^c \tau$ denote
the $\Ad_c(H)$-module defined by ${}^c\tau(h) = \tau(\Ad_{c^{-1}} H)$.

\subsubsection{} \label{sec:two}
In this section, we only assume that $G$ is a group and $\circG$ is an index two
normal subgroup of $G$.  We first review some simple relationships between
irreducible representations of group $G$ and $\circG$. Let
$c\in G \setminus \circG$.  Suppose $\circpi$ is an irreducible representation
of $\circG$.
Then $\Ind_{\circG}^{G}\circpi|_{\circG} \cong \circpi \oplus {}^c(\circpi)$.
The induced representation $\Ind_{\circG}^{G}\circpi$ is either
(I) an irreducible
representation of $G$, which happens if and only if $\circpi$ and ${}^c(\circpi)$ are
non-isomorphic as $\circG$-modules or (II) it is a direct sum of
two irreducible $G$-modules.

Conversely, the restriction of an irreducible representation $\pi$ of $G$
to~$\circG$ is either (I)  a direct sum of two non-isomorphic irreducible
$\circG$-modules or (II) an irreducible $\circG$-module.

\def\circrho{{{}^\circ \! \rho}}

\subsubsection{} \label{sec:restoG}
As the basic step, we first show that the theory of depth zero supercuspidal
representation of connected group extends to $G$. 
Note that $G_x \cap (G\setminus \circG) \neq \emptyset$ for each vertex
$x \in \BTB{G}$.

Let $\pi$ be a depth zero supercuspidal $G$-module.  We consider its restriction
to $\circG$ and relate to the two cases (I) and (II) in \Cref{sec:two} above.
\begin{enumerate}[(I)]
\item Suppose $\pi|_\circG = \circpi_1\oplus \circpi_2$. Then there is a depth zero
  minimal $K$-type $(x,\circrho)$ of $\circpi_1$ where $x$ is a vertex in
  $\BTB{G}$ and $\circrho$ is a
  cuspidal $\circG_{x}$-module. 
  We fix $c\in G_x \cap (G\setminus \circG)$.
  By \cite{MP2},
  ${}^c\circrho\not\cong \circrho$ since
  ${}^c\circpi_1 \cong \circpi_2\not\cong \circpi_1$. Hence,
  $\pi = \cInd_{G_x}^G\rho$ where $\rho := \Ind_{\circG_x}^{G_x} \circrho$ is an
  irreducible cuspidal $G_x$-module. 
	
\item Suppose $\pi|_{\circG}$ is an irreducible supercuspidal. Then it has a minimal
  $K$-type $(x,\circrho)$ where $x\in \BTB{G}$ is a vertex. Let $\rho$ be
  the natural representation of $G_{x}$ on the $G_{x,0^+}$-invariant subspace of
  $\pi$.  Clearly $\rho|_{\circG_x} = \circrho$. Hence $\pi = \cInd_{G_x}^G\rho$
  since $\cInd_{G_x}^G \rho|_{\circG} = \cInd_{\circG_x}^{\circG} \circrho = \circpi$ is
  irreducible.
\end{enumerate}
In summary, Condition {\bf D4} and related claims in \cite[p. 590]{Yu} hold for
even orthogonal groups.

\subsubsection{}

The centralizer $\Cent{G}{\gamma}$ of a semisimple element $\gamma\in\fgg$ is
called a \emph{twisted Levi subgroup} of $G$ if $\Cent{\circG}{\gamma}$ is a
twisted Levi subgroup of $\circG$ \cite[p. 586]{Yu}. Therefore a twisted Levi
subgroup is a product of general linear groups, unitary groups and at most one
even orthogonal group.  Combining with \Cref{sec:restoG} above, Yu's definition
of generic $G$-datum in \cite[p. 615]{Yu} extends to the orthogonal groups
without any change.

Our formulation of supercuspidal data follows Kim's simplification \cite[Section 5]{Kim}. 
To translate between Kim and Yu's formulations, we see that
the following variation of \cite[Lemma~5.5]{Kim} holds for all twisted Levi
subgroups appearing in the construction of supercuspidal representations: 

\begin{lemma}\label{lem:phi.O}
  Suppose $\ckG$ is a twisted Levi subgroup of $G$ such that
  $\Cent{}{\circckG}/\Cent{}{G}$ is anisotropic, $\gamma$ is a negative depth
  element in the center of the Lie algebra of $\ckG$ and $x\in \BTB{\ckG}$, then
  there exists a character $\phi$ of $\ckG$ such that
  $\phi|_{\ckG_{x,0^+}}(g) = \psi(\bB(\gamma,\log(g)))$ for every
  $g\in \ckG_{x,0^+}$.
\end{lemma}

\begin{proof}
  Since $\Cent{}{\circckG}/\Cent{}{G}$ is anisotropic, $\ckG$ cannot have any
  $\rO(1,1)$ factor or general linear group factor. Therefore, the center of
  the $\Lie(\ckG)$ is contained in a product of unitary Lie algebra factors. 
	Now the lemma follows immediately from its connected group version
  \cite[Lemma~5.5]{Kim}.
\end{proof}

\subsubsection{}\label{sec:GE2}
 In this subsection, we refer to Conditions {\bf GE1} and {\bf GE2} and the notation in~\cite[\Sec{8}]{Yu}. 
Let $X$ be a good element in $G$ and let
   $\ckG := \Cent{G}{X}$ be the corresponding twisted Levi subgroup.
  This is {\bf GE1} under our settings. The following modification of 
  {\bf GE2} is clearly implied by {\bf GE1} for orthogonal groups:

  \begin{claim*}[\bf GE2']
    Let $T\subset \ckG\subset G$ be maximal torus of $\ckG$ and $X \in
    \Lie(T)$. Let $\barF$ be the algebraic closure of $F$. 
		Let $\tX^*$ be as in \cite[p. 596]{Yu}.
		Let $W := N_{G(\barF)}(T(\barF))/T(\barF)$ and
    $\ckW := N_{\ckG(\barF)}(T(\barF))/T(\barF)$ be the absolute Weyl groups of
    $G$ and $\ckG$ respectively. 
		Then $\Cent{W}{\tX^*} = \ckW$. \qed
  \end{claim*}

\subsubsection{}
Let $\Sigma = (x,\Gamma,\phi,\rho)$ be a supercuspidal datum of $G$ as in
\Cref{def:SC.D}.  Argue as in \cite[Remarks~5.10]{Kim}, there is a datum
$(\vecG,x, \vecrr,\vecphi,\dotrho)$ such that $\phi_i$ is represented by the
good element $\Gamma_i$ and
$\rho\otimes \phi = \dotrho\otimes \prod_{i=0}^d\phi_i$.  In particular,
$\etaSigma$ constructed in \Cref{def:SC.N}~\cref{it:SC.N.8} is the same as the
$K$-type constructed following Yu's recipe.

\subsubsection{} 
We now explain how to extend the proofs in \cite{Yu} to $G$.




\begin{thm} \label{thm:extendYu} The representation
  $\piSigma := \cInd_K^G \etaSigma$ constructed in
  \Cref{def:SC.N}~\cref{it:SC.N.9} is an irreducible supercuspidal
  representation of $G = \rO(V)$.
\end{thm}

\begin{proof}
  In \cite[\Sec{4}]{Yu}, Yu defines conditions \SC{1}, \SC{2} and \SC{3} which
  do not assume that the group is connected so they are applicable to $G$.
  We now verify these conditions and then \cite[Proposition~4.6]{Yu} will imply 
  that $\piSigma$ is an irreducible supercuspidal representation of $G$.

  First we consider \SC{1}. Its proof in \cite[Theorem~9.4]{Yu} relies on
  \cite[Lemma~8.3]{Yu} which still applies. In the proof of
  \cite[Lemma~8.3]{Yu}, Yu uses conditions {\bf GE1} and {\bf GE2} which are
  satisfied by the discussion in \Cref{sec:GE2}. In addition, one also needs the
  existence of certain integral model of the Moy-Prasad groups. This is clear by
  viewing the orthogonal group as a symmetric subgroup of the general linear
  group, see \cite{BS}.

  The condition \SC{2} is about the existence of Heisenberg-Weil
  representation. This is taken care of by \Cref{sec:Special}.

  The proof of \SC{3} takes up \cite[\Sec{12-13}]{Yu}. Though it is long but the
  proof extends without change to our case.
\end{proof}

\subsubsection{} Next we extend the exhaustion result of \cite{Kim} to $G$.

\begin{thm} \label{thm:extendKim}
Given \eqref{eqpbdd}, the set of $\pi_\Sigma$ exhausts all the supercuspidal
representations of $G = \rO(V)$.
\end{thm}

\begin{proof}
  Let $\pi$ be an irreducible supercuspidal representation of $G$.  Then $\pi$
  contains an irreducible supercuspidal representation $\circpi$ of $\circG$. 
  By
  \cite{Kim}, $\circpi = \circpi_{\circSigma}$ for some supercuspidal datum
  $\circSigma = (x, \Gamma, \circphi, \circrho)$.  Using $\Gamma$ and $x$ we
  define~$G_x^0$ and $K$ etc. By \Cref{lem:phi.O}, we can assume $\circphi$
  extends to a character $\phi$ of $G^0$. Let $\kappa$ be the $K$-module defined
  by the procedure in \Cref{sec:kappa}.

  We note that $K_{0^+} = \circK_{0^+}$.
	Define $\rho' = \Hom_{K_{0^+}}(\kappa,\pi)$ to be the
  multiplicity space of $\kappa|_{K_{0^+}}$. It is a natural $G_x^0$-module and
  the $\circrho$-isotypic subspace $\rho'[\circrho]\neq 0$. Pick any irreducible $G^0_x$-submodule $\rho$ in 
  $\rho'[\circrho]$ and define $\Sigma = (x,\Gamma, \phi,
  \rho)$.
  Then $\pi$ is a submodule of $\piSigma = \cInd_{K}^G\eta_\Sigma$. 
  By \Cref{thm:extendYu}, $\piSigma$ is irreducible so $\pi = \piSigma$. 
	This completes the proof.
\end{proof}

\subsubsection{}
Finally we extend \Cref{thm:HM}  to $G$.

\begin{thm} \label{thm:extendHM} Let $\Sigma$ and $\dotSigma$ be two
  supercuspidal data for $G = \rO(V)$.  Then $\piSigma \simeq \pi_{\dotSigma}$
  if and only if $\Sigma$ and $\dotSigma$ are equivalent with each other in the
  sense of \Cref{def:SC.eq}.
\end{thm}

\begin{proof}
Suppose $\Sigma = (x,\Gamma,\phi,\rho)$ and $\dotSigma =
(\dot{x},\dotGamma,\dotphi,\dotrho)$. We argue case by case. 

\begin{asparaenum}[({Case} A)]
\item First we suppose $G_x^0\subset \circG$.  Then $K:=K_\Sigma\subset \circG$
  and $\Sigma$ is also a supercuspidal datum for $\circG$.  Fixing
  $c\in G\setminus \circG$, then
  $\piSigma|_{\circG} = \cInd_{\circK}^{\circG}\etaSigma \oplus
  \cInd_{\Ad_c(\circK)}^{\circG}\eta_{\cSigma}$, where
  $\cSigma: = (c\cdot x, \Ad_c\Gamma, {}^c\phi ,{}^c\rho)$.  On the other hand,
  let $\cdrho_1$ be an irreducible $\cdK$-submodule in~$\dotrho$ and
  $\cdSigma_1 = (\dotxx,\dotGamma, \dotphi|_{\circG_{\dotxx}^0},\cdrho_1)$. Then
  $\piSigma|_{\circG} \cong\pi_{\dotSigma}|_{\circG}$ contains
  $\cInd_{\cdK}^{\circG}\eta_{\cdSigma_1}$.  Now by \cite{HM}, there is a
  $g\in \circG$ such that $\cdSigma_1$ is equivalent to either $\Sigma$ or
  $\cSigma$. In particular, $\dotK\subset \circG$ so that $\dotrho = \cdrho_1$
  is an irreducible $\circG_{\dotxx}^0 = G_{\dotxx}^0$-module. Hence, $\Sigma$
  and $\dotSigma$ are equivalent.

\item Next, we suppose that $G_{x}^0 \cap (G\setminus \circG) \neq \emptyset$. 
  Fix an element $c$ in $G_{x}^0\cap (G\setminus\circG)$.
  \begin{asparaenum}[(1)]
  \item Suppose $\rho|_{\circG_{x}^0} \cong \circrho \oplus {}^c(\circrho)$
    where $\circrho$ is an irreducible cuspidal $\circG_{x}^0$-module.  Using
    similar proof in Case~A, we see that, up to $G$-conjugacy,
    $\dotxx=x$, $\dotGamma \in \Gamma + \fgg_{x,0} $ and
    $(\dotrho\otimes \dotphi)|_{\circG_{x}^0}$ contains
    $\circrho\otimes (\phi|_{\circG_{x}^0})$. Hence, $\Sigma$ and
    $\dotSigma$ are equivalent by applying the discussion in \Cref{sec:two} to
    $K > \circK$.

  \item Suppose $\circrho:=\rho|_{\circG_{x}^0}$ is already irreducible. Let
    $\circSigma = (x,\Gamma, \phi|_{\circG_{x}^0}, \circrho)$ be the
    corresponding supercuspidal datum for $\circG$.  Then
    $\eta_{\circSigma} = \etaSigma|_{\circK}$ is the supercuspidal type
    of~$\circG$ determined by $\circSigma$ and
    $\piSigma|_{\circG} = \Ind_{\circK}^{\circG} \eta_{\circSigma}$.  Again by
    the argument in Case A, up to $\circG$-conjugacy, we could assume
    $\dotxx = x$, $\dotGamma \in \Gamma + \fgg_{x,0}$, $\dotphi=\phi$ and
    $\rho|_{\circG_{x}^0} \cong \dotrho|_{\circG_{x}^0}$. One observes that both $\rho$
    and $\dotrho$ must be $G_{x}^0$-submodules in the multiplicity space $\rho'$
    in \cref{eq:mult}. However
    $\rho'|_{\circG_{x}^0} \simeq \rho|_{\circG_{x}^0}$ is irreducible. Hence,
    $\rho= \dotrho$.
\end{asparaenum}
\end{asparaenum}
\end{proof}

\begin{remark}
As a consequence of \Cref{thm:extendHM}, \Cref{lem:mult1} also extends to 
even orthogonal groups. We leave the details to the reader. 
\end{remark}

\subsection{Tamely ramified supercuspidal representation for covering groups}
  \label{sec:SC.cover}
  We now state and review some results on tamely ramified supercuspidal
  representations of $\wtG$.  We will supply some proofs although they follow
  almost immediately from those in the algebraic group case.\footnote{The
    essences of most proofs for the algebraic groups are related to the positive
    depth parts. Hence these proofs translate to our case by identifying the
    positive depth parts via the canonical splitting \eqref{eq:CSPT}.}  Depth
  zero representations of non-linear covers of $p$-adic groups were studied by
  Howard and Weissman \cite{HW}.

\subsubsection{} The supercuspidal data for $\wtG$ is an extension of the
supercuspidal data for $G$ by a splitting of the covering:

\begin{definition} \label{def:SC.D.c} A \emph{supercuspidal datum} for $\wtG$ is
  a tuple $\tSigma = (x, \Gamma, \phi, \rho,\xi)$ such that
  $\Sigma = (x, \Gamma, \phi, \rho)$ is a supercuspidal datum for $G$ as in
  \Cref{def:SC.D} and
\begin{enumerate}[(a)]
\resumesdenumi
\item $\xi \colon K \rightarrow \wtK$ is a splitting of the
  $\bC^\times$-covering $\wtK\surj K$ such that
  $\xi|_{K_{0^+}} = \CSPT|_{K_{0^+}}$ where $\CSPT$ is the canonical splitting
  defined in \eqref{eq:CSPT} and, $K$ and $K_{0^+}$ are defined in
  \Cref{def:SC.N} with respect to $\Sigma$. The splitting $\xi$ induces an
    identification of $\wtK$ with $K\times \bC^\times$. Let
    \begin{equation}\label{eq:tspt}
      \tspt  \colon K \times \bC^\times \iso \wtK.
    \end{equation}
    denote the corresponding isomorphism.  
\end{enumerate}
\end{definition}

\def\MU#1#2{\mu_{{#1},{#2}}}
\begin{remark}
Suppose $\xi_1$ is another splitting of $K$, then $\xi$ and $\xi_1$ differ
by a character. 
More precisely we have a character 
\begin{equation}\label{eq:MU.def}
\MU{\xi_1}{\xi}\colon K\rightarrow \bC^\times \quad
 \text{given by} \quad
 \MU{\xi_1}{\xi}(k)\xi_1(k) = \xi(k) \quad \text{for all } k \in K. 
\end{equation}
\end{remark}

\subsubsection{}
Let $\tSigma = (x,\Gamma, \phi,\rho, \xi)$ be a supercuspidal datum for $\wtG$. 
We assume all the notations in \Cref{sec:Ktype}. We let
\begin{equation} \label{eqKtildetype}
\tetaSigma := (\etaSigma \boxtimes\idCx) \circ \tspt^{-1}
\end{equation}
which is an irreducible $\wtK$-module and let 
\[
\tpiSigma:=\cInd_{\wtK}^{\wtG}\tetaSigma.
\]

\begin{thm}
The representation $\tpiSigma$ is an irreducible supercuspidal representation.
\end{thm}

\begin{proof}
Let 
$
 I(\etaSigma) :=\Set{g\in G | \Hom_{{}^gK \cap K}(\etaSigma, {}^g\etaSigma)\neq 0} 
$
be the set of intertwiners of $\etaSigma$.
We will follow the proof of \cite[Proposition~4.6]{Yu} in which Yu proves that 
\begin{equation} \label{eqIeta}
I(\etaSigma) = K.
\end{equation}
(Also see the proof of \Cref{thm:extendYu} for even orthogonal groups.)
In order to adapt Yu's proof to our theorem, it suffices to show that the set 
\[
I(\tetaSigma):= \Set{\tgg\in
\wtG | \Hom_{{}^\tgg\wtK \cap \wtK}(\tetaSigma, {}^\tgg\tetaSigma)\neq 0}
\] 
of intertwiners of $\tetaSigma$ is exactly $\wtK$. 
Fix  $\tgg\in \wtG$ and let $g$ be its image in $G$. 
Note that the adjoint action of $\wtG$ factors through $G$. Then 
\begin{enumerate}[(a)]
\item 
${}^\tgg \wtK \cap \wtK  = \widetilde{{}^gK\cap K}$ which we identified with
$({}^g K\cap K) \times \bC^\times$ using $\tspt$ in \eqref{eq:tspt}, and

\item ${}^\tgg \tetaSigma|_{{}^\tgg\wtK \cap \wtK} = \left({}^g \etaSigma
    \boxtimes \id_\bC\right)\circ \tspt^{-1}|_{{}^\tgg\wtK \cap \wtK}$.
\end{enumerate}
Therefore, we have $I(\tetaSigma) = \widetilde{I(\etaSigma)}$ which is $\wtK$ by \eqref{eqIeta}. 
\end{proof}

\begin{definition}
  We call $\tpiSigma$ the supercuspidal representation of $\wtG$ attached
    to the datum~$\tSigma$.
\end{definition}

\subsubsection{}  \label{sec:D.eq.c}
Now we describe the equivalence of supercuspidal data for covering groups.

\begin{definition} \label{def:cSC.eq} Let
  $\tSigma = (x, \Gamma, \phi, \rho, \xi)$ and
  $\dottSigma = (\dotxx, \dotGamma, \dotphi, \dotrho,\dotxi)$ be two
  supercuspidal data for $\wtG$.  We say that $\tSigma$ and $\dottSigma$ are
  \emph{equivalent} data if there exists an element $g \in G$ such that
\begin{enumerate}[(a)]
\item $x = g\cdot \dotxx$,
\item $\Ad_g(\dotGamma) \in \Gamma+ \fgg_{x,0}$ and 
\item 
$((\dotrho\otimes \dotphi)\boxtimes \idCx)\circ \tdotxi^{-1} \cong ((\rho\otimes
\phi)\boxtimes \idCx) \circ \tspt^{-1} \circ \Ad_{g}$ as
  $\wtG^0_{\dotxx}$-module\footnote{Here $\Ad_{g}$ acts on $\wtG$
    since the adjoint action factors through the center.}.
\end{enumerate}
\end{definition}
We remark that Condition (c) is equivalent to 
\begin{enumerate}[(a')]
\setcounter{enumi}{2}
\item
  $\dotrho\otimes \dotphi \otimes \MU{\dotxi}{\xi^{g}} \cong (\rho\otimes \phi)
  \circ \Ad_{g}$ as $G^0_{\dotxx}$-module where $\MU{\dotxi}{\xi^{g}}$ is
  defined by \eqref{eq:MU.def} 
  and $\xi^g:= \Ad_{g^{-1}}\circ \xi\circ
  \Ad_{g}$. 
\end{enumerate}
Since our choices of splittings agree on the pro-unipotent part,
$\MU{\dotxi}{\xi^{g}}$ is a character that is trivial on
$G_{\dotxx,0^+}^0$, i.e. a character of
$\sfG^0_{\dotxx}:= G_{\dotxx,0}^0/G_{\dotxx,0^+}^0$.  Condition (c') is simpler
and seems easier to check because $\MU{\dotxi}{\xi^{g}}$ is trivial in
most of the cases (cf. \cite{Pan01}).

The following theorem is a variation of a result in \cite{HM}.  The reader may
consult \cite{HM} for notations when reading the proof and should note that the
notations in the proof may not agree with other parts of our paper.

\def\dotK{\dot{K}}
\def\dotwtK{\dot{\wtK}}

\begin{thm} \label{thm:cHM}
Let $\tSigma = (x,\Gamma, \rho,\phi,\xi)$ and $\dottSigma = (\dotxx,
\dotGamma,\dotrho, \dotphi,\dotxi)$ be two supercuspidal data. 
Then $\tpiSigma \cong \tpidotSigma$ if and only if $\tSigma$ and $\dottSigma$
are equivalent with each other. 
\end{thm}

\begin{proof}
  Suppose $\tSigma$ and $\dottSigma$ are equivalent. By definition $\tetaSigma$
  and $\teta_\dottSigma$ are isomorphic up to $G$-conjugacy so
  $\tpiSigma \cong \tpidotSigma$.

We now assume $\tpiSigma \cong \tpidotSigma$.  Let $K$ and $\dotK$ be the
compact subgroups determined by $(x,\Gamma)$ and $(\dotxx, \dotGamma)$
respectively.  Let $\bullet \mapsto (\bullet)^\vee$ denote the operation of
taking contragredient.  

In order to apply the result in \cite[\Sec{5}]{HM}, we make following
definitions. We let $\cG:= G\times G$ equipped with an involution $\theta$
sending $(g_1,g_2)$ to $(g_2,g_1)$ and identify $G$ with the diagonal subgroup
of $\cG$. Let $\Sigma = (x,\Gamma,\rho,\phi)$ and
$\dotSigma^\vee = (\dotxx,-\dotGamma,\dotrho^\vee,\dotphi^\vee)$. We view
$\Psi := \Sigma\times \dotSigma^\vee$ as a supercuspidal datum for $\cG$.
 
  The $\wtG$-module $\tpiSigma \otimes (\dottpiSigma)^\vee$
  factors to a $G$-module and we have
  \[
  \begin{split}
    \bC &=  \Hom_G(\tpiSigma\otimes \dottpiSigma^\vee, \bfone) = \Hom_G(
    \cInd_{\wtK\times \dotwtK}^{\wtG\times \wtG}
    \left(\tetaSigma \otimes\teta_\dottSigma^\vee\right),\bfone ) \\
    & = \sum_{g\in K\backslash G /\dotK} \Hom_{\wtK\cap{}^g\dotwtK}(\tetaSigma
    \otimes \left.^g\teta_{\dottSigma}^\vee\right. , \bfone).
  \end{split}
  \]
  Therefore, exactly one term of the above summation is non-vanishing and of
  dimension~1.  By replacing $\dottSigma$ by its $G$-conjugate, we may assume that 
  \begin{equation}\label{eq:equiv.K}
  \Hom_{\wtK\cap \dotwtK}(\tetaSigma \otimes \teta_{\dottSigma}^\vee,
  \bfone) = \bC.
  \end{equation}
   By \Cref{sec:LM.real} $\spt|_{K_+\cap \dotK_+} = \dotxi|_{K_+\cap \dotK_+}$. 
	Hence \eqref{eq:equiv.K} implies
  \[
  \Hom_{K_+\cap \dotK_+}(\psi_\Gamma \otimes
  \psi_{-\dotGamma}, \bfone) \neq 0.
  \] 
  This means $\Psi$ is \emph{weakly compatible} with the involution $\theta$
  in the sense of \cite[Definition~5.6]{HM}. Now \cite[Proposition~5.7]{HM} implies  that
  $\Psi$ is \emph{weakly $\theta$-symmetric} up to a
  conjugation of $K \times \dotK$. 
  By the definition of weakly $\theta$-symmetric in \cite[Definition~3.13]{HM}, 
  we may assume $\Gamma = \dotGamma$ and so $G^0 = \Cent{G}{\Gamma} = \Cent{G}{\dotGamma}= \dotG^0$.

 Since we are in the ``group case'', the theorem could be proven by reducing to the
  depth zero case: Thanks to \cite[Lemma~5.5]{Kim}, we can fix a character $\phi_0$
  of $G^0$ extending $\psi_{\Gamma}|_{G^0_{0^+}}$. 
  Consider the depth zero supercuspidal data $\tSigma_0 := (x, 0, \bfone, \rho
  \otimes \phi \otimes \phi_0^{-1},\xi)$ and $\dottSigma_0 := (\dotxx, 0, \bfone, \dotrho
  \otimes \dotphi \otimes \phi_0^{-1},\dotxi)$. Let $\teta_{\tSigma_0}$ and
  $\teta_{\dottSigma_0}$ be depth zero supercuspidal $K$-types of $\wtG^0$
  defined by data $\tSigma_0$ and $\dottSigma_0$ respectively.
   Restricting \eqref{eq:equiv.K} to
  $\wtG^0_x \cap \wtG^0_{\dotxx}$ gives
  \[
0  \neq \Hom_{\wtG^0_x\cap \wtG^0_{\dotxx}}(\tetaSigma \otimes \teta_{\dottSigma}^\vee,\bfone)
= \Hom_{\wtG^0_x\cap \wtG^0_{\dotxx}}(\teta_{\tSigma_0} \otimes \teta_{\dottSigma_0}^\vee, \bfone)^{\oplus m}
  \] 
  where $m$ is certain multiplicity. On the other hand,
  $\Hom_{\wtG^0_x\cap \wtG^0_{\dotxx}}(\teta_{\tSigma_0} \otimes
  \teta_{\dottSigma_0}^\vee, \bfone)\neq 0 $ implies the depth zero
  supercuspidal $\wtG^0$-modules
  $\tpi_{\tSigma_0}=\cInd_{\wtG^0_x}^{\wtG^0}\teta_{\tSigma_0}$ and
  $\tpi_{\dottSigma_0}=\cInd_{\wtG^0_{\dotxx}}^{\wtG^0}\teta_{\dottSigma_0}$ are
  isomorphic to each other. Since all depth zero unrefined minimal
    $K$-types are associates (see
    \cite[Proposition~3.6]{HW}), 
  there is an element $g\in G^0$ such that
  $x = g\cdot \dotxx \in \BTB{G^0}=\rBTB{G^0}$ and 
    $((\rho\otimes \phi\otimes\phi_0^{-1})\boxtimes
    \id_{\bC^\times})\circ\txi^{-1} \circ \Ad_g \cong ((\dotrho\otimes
    \dotphi\otimes \phi_0^{-1})\boxtimes \id_{\bC^\times})\circ \tdotxi^{-1}$
  as  $\wtG^0_{\dotxx}$-modules. This finishes the proof of the
  theorem.
\end{proof}


\subsubsection{} \label{sec:SC.E.c} We will extend the results of \cite{Kim} to
every covering group $\wtG$ appearing in a type I dual pair. More precisely we
will show that the set of $\tpiSigma$ exhausts all the genuine\footnote{Here
  ``genuine'' means the restriction of the representation on $\bC^\times$ is the
  scalar multiplication $\id_{\bC^\times}$.} supercuspidal representations of
$\wtG$ under the assumption that $p$ is big enough.  Gan and Kim are currently
preparing a manuscript on such kinds of non-linear covering groups
\cite{GanKim}.

Since all covering groups in this paper occur in a type~I dual pair, these are
split central covering groups except the odd orthogonal-metaplectic dual
pairs. Therefore Kim's exhaustion result applies except for metaplectic groups.
We now show that the exhaustion for metaplectic groups could be obtained from
that of odd-orthogonal groups:

\def\Mpn{\Mp}

\begin{cor}
  Let $V'$ be a symplectic space over $F$ of even dimension $n$.  Let $\Mpn$ be
  the metaplectic $\bC^\times$-cover of $\Sp(V')$.  Suppose
  $p\geq \max\set{2n+3, e_F(n+1)+2}$ where $e_F$ is the ramification index of
  $F/\bQ_p$.  Then every genuine supercuspidal representation of $\Mpn$ is of
  the form $\tpiSigma$ where $\tSigma$ is a supercuspidal data of $\Mpn$.
\end{cor}

\begin{proof}
Let $\tpi'$ be an irreducible supercuspidal genuine $\Mpn$-module. 
By the conservation relation \cite{SZ}, there is an odd dimensional
quadratic space $V$ such that 
\begin{inparaenum}[(i)]
\item 
$\dim_FV \leq n+1$, 
\item $(G,G') = (\rU(V),\rU(V'))$ form a type~I dual pair so that
  $\wtG' = \Mpn$ and
\item $\tpi'=\thetaVVp(\tpi)$ for an irreducible supercuspidal genuine
  $\wtG$-module. 
\end{inparaenum}
By \Cref{prop:pbd1} both $\rU(V)$ and $\rU(V')$ satisfy Kim's hypotheses
\cite[\Sec{3.4}]{Kim}.  Hence there is a supercuspidal data $\tD$ such that
$\tpi:= \tpi_{\tD}$ by \cite{Kim}. This is the starting point of the proof of
\Cref{it:Main.2} of the main theorem in \Cref{sec:pf.main.2}. Using
\Cref{prop:exhaust}, the proof gives supercuspidal data $\Sigma$ and $\Sigmap$
such that $\tD$ and $\tSigma$ are equivalent (i.e. $\tpi = \tpiSigma$) and
$\tpi'=\tpiSigmap$.  In particular, $\tpi'$ is realized as a supercuspidal
representation of $\wtG'$ attached to the supercuspidal datum $\tSigma'$. This
finishes the proof.
\end{proof}

\begin{remark}
The proof of \Cref{prop:exhaust} does not depend on Kim's work on
exhaustion except a variation of \cite[Proposition~17.2]{Kim}. 
Hence there is no circular reasoning.
\end{remark}

\section{Good factorizations and block decompositions}
\label{sec:GoodFact}
In this section, we
first construct $\GL$-good
factorizations which will be used in \Cref{sec:OB1}. 
Then we will define a notion of block decompositions for supercuspidal data. 
The theta lifting map of supercuspidal
data in \Cref{sec:LD} is defined based on this notion. We remark that parts of the
treatment resemble those of  \cite{S05}*{\Sec{3}}.

\subsection{$\GL$-good factorization} \label{sec:Efactor} We now
construct a good factorization of a tamely ramified semisimple element
$\Gamma \in \fgg$ following Howe~\cite{Howe}. Let $F' := \Cent{}{D}$
  be the center of $D$ which is also identified with the
center of $\EndDV$. 

Let $A := F'[\Gamma] \subseteq \End_D(V)$. Then $A$ is isomorphic to a product
$\prod_{j\in \fJ} F_j$ of (tamely ramified) finite extensions of $F'$ where
$\fJ$ is a finite index set. Furthermore, we have factorization of the
$A$-module $V = \bigoplus_j V_j$ where $V_j$ is an $(F_j,D)$-bimodule.  Since
$\Gamma^* = -\Gamma$, $*$ induces an involution on $A$ and on the set $\fJ$
respectively. An orbit of the $*$-action on $\fJ$ has at most $2$ elements.
Therefore we have a decomposition of $\fJ$ and $A$ such that
\crefformat{itA}{{\rm (#2A#1#3)}}
\begin{enumerate}[({A}1)]
\item \label[itA]{it:A1} $\fJ = \fJ_0 \sqcup \fJ_1 \sqcup * (\fJ_1)$ and
  $A = F'[\Gamma] = \prod_{j\in \fJ_0} F_j \times \prod_{j\in \fJ_1}
  (F^+_{j}\times F^-_{j})$
  where $F^+_j := F_j$ and $F^-_j := F_{*(j)}$ for $j\in \fJ_1$;

\item \label[itA]{it:A2} the $*$-action is an involution or the
  identity\footnote{This is because $F_j$ could be equal to $F$ in certain
    cases. For example, if  $D=F$ and $\Gamma$ is not a full rank matrix, then 
    $F[\Gamma] \cong F \oplus (F[x]/P(x))$ where $P$ is the minimal polynomial
    of $\Gamma$.} on the field~$F_j$  when $j \in \fJ_0$. 

\item \label[itA]{it:A3} the $*$-action permutes $F_j^+$ and $F_j^-$ when $j\in \fJ_1$.
\end{enumerate}

  \begin{lemma} \label{lem:aniso} Let $\circG^0$ be the connected component of
    $G^0:=\Cent{G}{\Gamma}$. Let $\fgg^0$ be its Lie algebra and let
    $\fgl^0 := \Cent{\fgl_D(V)}{\Gamma}$.  Then
    \begin{enumerate}[(i)]
    \item \label{it:aniso.1} $\Cent{}{\fgl^0} = F'[\Gamma]$;
    \item \label{it:aniso.2} 
the
    center $\Cent{}{\circG^0}$ is anisotropic if and only if (a)
    $\fJ_1 = \emptyset$ in the decomposition \cref{it:A1} and (b) there is no
    $\SO(1,1)(F)$ factor in $\circG^0$; 
    \trivial[h]{In the case of $\fso(1,1)$,
      $\Cent{}{\fgg^0} \neq F'[\Gamma]\cap \fgg$. }
  \item \label{it:aniso.3} $\Cent{}{\fgg^0} = F'[\Gamma] \cap \fgg$ when the
    equivalent conditions in \Cref{it:aniso.2} holds.
  \end{enumerate}
\end{lemma}

\begin{proof} 
  Let $D_j := F_j\otimes_{F'}D$.  \trivial[h]{Note that
    $F_j \cap D = F' \subset \fgl_{F'}(V)$ and so that $D_j$ is exactly the
    algebra generated by right multiplication by $D$ and the left action
    $\Gamma$ in $\fgl_{F'}(V_j)$.}  Then $D_j$ is a central simple algebra over
  $F_j$ and $D[\Gamma] = \prod_{j} D_j$. Hence
  $\fgl^0 = \Cent{\fgl_D(V)}{\Gamma} = \prod_i \End_{D_i} V_i$ where each factor
  $\fgl_i := \End_{D_i} V_i$ is a central simple algebra over $F_i$.  This
  proves \Cref{it:aniso.1}.

  The
  $*$-action permutes $\fgl_j$ and $\fgl_{*(j)}$.  If $j\in \fJ_0$, then the
  form $\innv{}{}$ restricted on $V_j$ is non-degenerate. \trivial[h]{Since the
    idempotent $\bfone_j$ is $*$-invariant.}  If $j\in \fJ_0$ and $F_j\neq F'$,
  then the $*$-action on $F_j$ is nontrivial since
  $(\Gamma|_{V_j})^* = - \Gamma|_{V_j}$. In this case,
  $\rU_{j} := \set{g\in \End_{D_j}(V_j)| g^* g = \id_{V_j}}$ is a unitary group
  defined over the $*$-fixed point sub-field $F_j^*$ of $F_j$.  In summary, we
  have
  \[
    G^0 = \prod_{j\in \fJ_1}\GL_{D_j}(V_j) \times \prod_{\substack{j\in \fJ_0\\
        F_j\neq F'}} \rU_j \times \prod_{\substack{j\in \fJ_0\\ F_j =
        F'}}\rU(V_j).
  \]
  Now \Cref{it:aniso.2} and \Cref{it:aniso.3} follow.
\end{proof}

By \Cref{lem:aniso} and \Cref{def:SC.D}~(b), we may and will assume that
$\fJ_1 = \emptyset$ and \mbox{(A3)} will not happen since we only study those
$\Gamma$ which are contained in supercuspidal data.  \Cref{lem:GLgood} and
\Cref{prop:GD} below also apply to $\fJ_1 \neq \emptyset$. We will leave
the details to the reader.


\begin{lemma} \label{lem:GLgood} Let $\gamma \in \fgg$ be a tamely ramified
  semisimple element. If it is $\GL_D(V)$-good, then it is $G$-good.
\end{lemma}

\begin{proof}
  It is enough to prove this lemma after a base change to a tamely ramified
  extension of $F$ such that $\Gamma$ is contained in a split Cartan subalgebra
  of $\fgg$.  Since the set of roots of~$\fgg$ is the restrictions of a subset
  of roots of $\fgl_D(V)$, the lemma follows.
\end{proof}

\begin{prop} \label{prop:GD}
Let $\Gamma$ be a tamely ramified semisimple element in $\fgg$. Then there is a $G$-good factorization $\Gamma =
\sum_{i=-1}^d \Gamma_i$ such that 
$\Gamma_i$ is $\GL_D(V)$-good for $0 \leq i \leq d$.
\end{prop}
The construction of the factorization is essentially the \emph{Howe factorization}.
\begin{proof}
  We fix a uniformizer $\varpiF$ of $F$.
  We recall that $A = F'[\Gamma] = \prod_{j \in \fJ} F_j$ is a product of fields. 

  First we assume that the product has only one factor, i.e. $\Gamma \neq 0$ and
  $A = F'[\Gamma]$ is a field.  Let~$e_A$ be the ramification index of $A/F$.
  Suppose $\Gamma$ has valuation $\frac{k}{e}$ such that $k$ and $e$ are
  coprime.  Consider $b := \varpiF^{-k}\Gamma^e$. Then
  $b \in \foo_A \setminus \fpp_A$.  Let $C$ be the set of roots of unity in
  $F'[\Gamma]$. Then $(b+\fpp_A) \cap C$ has a unique element, say $\hatbb$.
  \def\pmodAke{\pmod{\fpp_A^{\frac{k}{e}e_A+1}}}
\begin{claim} The equation $\varpiF^{-k}\gamma^e = \hatbb$ has a unique
  solution $\gamma = \hatgamma$ in $A$ such that
  $\Gamma \equiv \hatgamma \pmodAke$. Moreover $\hatgamma^* = -\hatgamma$, i.e. $\hatgamma \in \fgg$.
\end{claim}

\begin{proof}
  Note that $p \nmid e$. Hence the map
  $1+ \fpp_A \rightarrow b +\fpp_A = b(1+\fpp_A)$ given by
  $1+x \mapsto b(1+x)^e$ is a surjection (in fact it is a bijection) for any
  $b\in \foo_A\setminus \fpp_A$.  Therefore $\varpiF^{-k}\gamma^e = \hatbb$ has
  a solution $\hatgamma = \Gamma(1+x)$ for certain $x\in \fpp_A$.  On the other
  hand, all solutions of $\varpiF^{-k}\gamma^e = \hatbb$ are of the form
  $\gamma:=c\hatgamma$ such that $c^e=1$.  However among them
  only $\hatgamma$ satisfies $\Gamma \equiv \gamma \pmodAke$.  This proves the
  first assertion of the claim.

  Both $\hatgamma^*$ and $-\hatgamma$ are solutions of
  $\varpiF^{-k}\gamma^e = \hatbb^* = (-1)^e\hatbb$ and
  $\hatgamma^* \equiv \Gamma^* \equiv -\Gamma \equiv -\hatgamma \pmodAke$. Hence
  $\hatgamma^* = -\hatgamma$ by the uniqueness of the first part. This finishes
  the proof of the claim
\end{proof}

Next we consider the general case when $F'[\Gamma ] = A = \prod_{j \in \fJ} F_j$
 is a product of fields.  
Let~$\Gamma^j$ denote the $F_j$ component of $\Gamma$ to $F_j$. Let $-r_j$ 
denote the valuation of $\Gamma^j$ and let
$-r = \frac{k}{e}$ denote the depth of $\Gamma$ where $k$ and $e$ are coprime
integers. Let $b_j = \varpiF^{-k} \Gamma_j^e \in \foo_{F_j} \setminus \fpp_{F_j}$ when
$r_j = r$. Define $\hatbb_j$ and $\hatgamma_j$ as before. Let $\beta_1 = \sum_{r_j = r} \hatgamma_j$.

\begin{claim}
The element $\beta_1$ is $\GL_D(V)$-good. 
\end{claim}

\begin{proof}
  We could base change to a tamely ramified extension $E$ of $F$ which at least
  contains a $e$-th root $\varpi_E$ of $\varpiF$ and sufficiently many roots of
  unity so that $\beta_1$ splits over $E$.

  Each factor $\hatgamma_j$ satisfies $\varpiF^{-k} \hatgamma_j^e = \hatbb_j$
  and $\hatbb_j$ is a certain root of unity. The eigenvalue of $\beta_1$ is of
  the form $\zeta_{l} \varpi_E^{k}$ where $\zeta_l$ is a root of unity.  The
  evaluation of $\beta_1$ on any root (with respect to any split maximal torus
  in $\fgl(V)\otimes_F E$ containing $\beta_1$) is a difference of two
  eigenvalues.  It is either $0$ or has valuation $\frac{k}{e} = -r$.  This
  proves the claim.
\end{proof}

Let $\Gamma' = \Gamma - \beta_1$.  Apply the same construction to $\Gamma'$ and
stop if $\depth(\Gamma') \geq 0$. Note that $\depth(\Gamma')> \depth(\Gamma)$.
Moreover the denominator of $\depth(\Gamma')$ is bounded by the maximum of the
ramification index of $F_j/F$, since $\Gamma' \subseteq F'[\Gamma]$. Therefore
the procedure stops after finitely many steps.  We get
$\beta_1, \beta_2, \cdots, \beta_{d'}$ for certain $d' \geq 1$.  If
$\beta_1 \not \in F'$, then we define $d = d'$, $\Gamma_d = 0$ and
$\Gamma_{d-i} = \beta_i$ for $1\leq i\leq d$.  If $\beta_1 \in F'$, then we
define $d = d'-1$, $\Gamma_d = \beta_1$ and $\Gamma_{d-i} = \beta_{1+i}$ for
$1\leq i \leq d$.  Finally
we set $\Gamma_{-1} = \Gamma - \sum_{i=0}^d \Gamma_i$.

Now $\Gamma = \sum_{i=-1}^d \Gamma_i$ is the required factorization.  
\end{proof}


\def\lid{{}^l \bfone}

\subsection{Block decomposition}  \label{sec:BD}
Let $\Sigma = (x,\Gamma, \phi, \rho)$ be a supercuspidal data of $G =\rU(V)$. 

\def\jval{\val}
Let $\Gamma^j$ be the $F_j$ component of $\Gamma$ in $\prod_{j \in \fJ} F_j$ in
\Cref{sec:Efactor}~\mbox{(A1)}.  Let 
\begin{equation}\label{eq:Gamma.j}
  \set{\bbr >  \cdots > \oor} = \set{-\val(\Gamma^j)>0 | j\in \fJ}
\end{equation}
be a set of positive numbers arranged in decreasing order where $b$ is the size
of the set. 
We set 
\[
\llA := \prod_{\jval(\Gamma^j) = -\llr} F_j\qquad \forall 1\leq l\leq b.
\]
For $l = 0$, we define $\zzA := \prod_j F_j$ where product is taken over those
$j \in \fJ=\fJ_0$ such that  $\Gamma^j = 0$ or $\jval(\Gamma^j) \geq 0$.

Let $\bfone$ and $\lid$ be the multiplicative identity elements 
of $A$ and $\llA$ respectively. 
Then we set ${}^l \Gamma = \lid \cdot \Gamma \cdot \lid$ and $\llV = \lid 
\cdot V$.
These give
\begin{equation} \label{eq:BD.Ga} A = F'[\Gamma] = \prod_{l=0}^b \llA, \quad
  \bfone = \bigoplus_{l=0}^b \lid, \quad \Gamma = \bigoplus_{l=0}^b
  \llGamma\quad 
  \text{and}\quad V = \bigoplus_{l=0}^b \llV.
\end{equation}


It is easy to see that $\innv{}{}$ restricted to $\llV$ is non-degenerate and $V = \bigoplus_{l=0}^b \llV$ is an orthogonal decomposition\footnote{If we set $\Gamma_{-1}= 0$, then $\zzV$ is the kernel of $\Gamma$.}.
We call \eqref{eq:BD.Ga} the \emph{block decompositions} of $A$, $\bfone$, $\Gamma$ and $V$ respectively.

\begin{definition} \label{def:BD0}
\begin{asparaenum}[(a)]
\item We say that a supercuspidal datum $\Sigma = (x,\Gamma, \phi, \rho)$ is a
  \emph{single block of positive depth $r$} if $\Gamma^j$  has the same valuation $-r<0$ for all
  $j\in \fJ$ in \Cref{sec:Efactor}~\mbox{(A1)}. Equivalently this means
  $\depth(\Gamma)= -r$, $\Gamma = \ooGamma$ and $V = \ooV$ in
  \cref{eq:BD.Ga}.
 
\item A depth zero supercuspidal datum $\Sigma = (x, \Gamma, \bfone, \rho)$ is
  called a \emph{single block of zero depth}.
\end{asparaenum}
\end{definition}

\begin{prop} \label{prop:decSigma} 
Let $\Sigma = (x,\Gamma, \phi, \rho)$ be a supercuspidal datum of $G$.
Then there exists an orthogonal decomposition of $V = \bigoplus_{l=0}^b \llV$ such that 
\begin{enumerate}[(i)]
\item \label{it:BDD.1} $\Gamma = \bigoplus_{l=0}^b  \llGamma$ where $\llGamma \in \End_D(\llV) \cap \fgg$;
\item \label{it:BDD.2} $\bbr > \cdots > \oor>0=\zzr$ where $\llr = \max\set{
      -\depth(\llGamma),0}$;

\item \label{it:BDD.3} $G^0 = \prod_{l = 0}^b \llG^0$; 

\item \label{it:BDD.4} $x  = (\zzxx, \cdots, \bbxx) \in \BTB{\zzG^0}\times \cdots \times \BTB{\bbG^0} \hookrightarrow \BTB{G}$;

\item \label{it:BDD.5} $\phi = \zzphi \boxtimes \dots \boxtimes \bbphi$ and $\rho =
  \zzrho\boxtimes \cdots \boxtimes \bbrho$ as $G^0_x = \prod_{l=0}^b
  \llG^0_{\llxx}$-modules;
	
\item  $\llSigma = (\llxx,\llGamma, \llphi, \llrho)$ is a single block supercuspidal datum of positive
  depth $\llr$ for $1 \leq l \leq b$; and 
\item $\zzSigma = (\zzxx, \zzGamma, \zzphi, \zzrho)$ is a single block
  supercuspidal datum of zero depth.
\end{enumerate}
Here $G^0 := \Cent{G}{\Gamma}$, $\llG := G\cap \End_D(\llV) = \rU(\llV)$ and $\llG^0 =
\Cent{\llG}{\llGamma}$ for $0 \leq l \leq b$ (cf.  \Cref{sec:Efactor}). 
\end{prop}

\begin{proof}
The proposition is straightforward probably except \cref{it:BDD.3}.
Part \cref{it:BDD.3} follows easily from the fact that $\lid \in F'[\Gamma]$ for all $1\leq
l\leq b$ (cf. \eqref{eq:BD.Ga}).  
\end{proof}

Motivated by the above proposition, we make the following definition.

\begin{definition} \label{def:BD}
Retaining the notation in \Cref{prop:decSigma}, we write $\Sigma = \bigoplus_{l=0}^b \llSigma$ and we call it
  \emph{the block decomposition} of $\Sigma$ (cf. next section). In this
    situation, we say that
  $\Sigma$ has $b$ blocks (by counting from $0$).
\end{definition}

In fact, the block decomposition is unique. See \cite[Remark~3.3~(iii)]{S05} for
an elementary argument.

\subsection{Direct sum of supercuspidal data}\label{sec:DS.D}
We now define the direct sum of single block
supercuspidal data with different depths. 

\begin{lemma}\label{lem:DS.D}
Suppose $b$ is a positive integer and $\set{\llSigma:0\leq l\leq b}$ is a set of supercuspidal data
  for $\rU(\llV)$ such that
  \begin{enumerate}[(a)]
	\item $\llV$ is an $\epsilon$-Hermitian space;
  \item $\zzSigma$ has zero depth;
  \item $\llSigma$ is single block of positive depth $\llr$ for $1\leq l\leq b$;
  \item $\bbr > \cdots >  \oor >0$.
  \end{enumerate}
  Let $V :=\bigoplus_{l=0}^b \llV$ be the orthogonal direct sum of
  $\epsilon$-Hermitian spaces.  Define $x,\Gamma, \phi$ and $\rho$ by the
  formula in \Cref{prop:decSigma}~\cref{it:BDD.1,it:BDD.4,it:BDD.5}
  respectively. 

  Then $\Sigma = (x,\Gamma, \phi,\rho)$ is a supercuspidal
  datum with depth $\bbr$ for $\rU(V)$ called the direct sum of $\set{\llSigma}$
  and we write $\Sigma = \bigoplus_l\llSigma$.
\end{lemma}

\begin{remark}
This construction also induces a notion of direct sum on the set of equivalence classes of data.
\end{remark}

\begin{proof}
  We recall $G^0 = \Cent{\rU(V)}{\Gamma}$ and $\llG = \rU(\llV)$ as in
  \Cref{prop:decSigma}. We claim that $G^0 = \prod_{l}\llG^0$. Indeed this
  follows from the observation that, after a certain base change, $\llV$ is
  exactly the direct sum of eigenspaces of $\Gamma$ whose eigenvalues have
  valuation $-\llr$ if $1\leq l\leq b$. The lemma now follows from this claim
  and we will leave the details to the reader.
\end{proof}

\section{Theta Lifts of supercuspidal data}\label{sec:LD}
The purpose of this section is to define the notion of theta lifts of
supercuspidal data. 
We first define the lift for a single block of zero depth or positive depth. 
The general case is defined using direct sum (cf. \Cref{sec:DS.D}).

\def\epsilonD{\epsilon_D}
Recall that $\fffF$ is the residual field of
$F$ and $\fffD$ is the residual field of $D$ which is at most a quadratic
extension of $\fffF$. We fix an uniformizer $\varpi_D\in \fppD$ such that $\tau(\varpi_D)
= \epsilonD\varpi_D$ and $\epsilonD \in \set{\pm 1}$. We retain the notation of \Cref{sec:DP}.

\subsection{Theta lifts of depth zero representations}\label{sec:TL.zero}
Local theta lifts between depth zero supercuspidal representations were studied by
Pan in \cite{Pan02J}.  It is reduced to theta correspondences over finite fields.
We summarize these results below.

\subsubsection{}\label{sec:GoodL}
We recall the definition of the dual lattice in \Cref{def:latticefn}.  A lattice
$L$ in $V$ is called a \emph{good lattice} if
$L^* \fpp_D \subseteq L \subseteq L^*$. 

The set of vertices in $\BTB{G}$ naturally corresponds to a subset of
  good lattices. \footnote{Suppose $x$ is a vertex in $\BTB{G}$ and $\sL$ is the
    corresponding lattice function, then $\sL\mapsto \sL_{0^+}$ gives the
    correspondence. The subset of good lattices could be proper, see
    \cite[Example~2.2.3.1]{YuB}.}


Let $L$ be a good lattice in $V$ corresponding to a vertex $x\in \BTB{G}$.
Then 
\begin{enumerate}[(a)]
\item $G_x = \set{ g \in G | g L = L }$. 

\item $G_{x,0^+} = \set{g \in G  | (g-1) L^* \subseteq L, (g-1) L\subseteq
    L^*\fpp_D }$.
\item The $\fffD$-modules $\ell := L/ L^* \fppD$ and $\ell^* := L^*/L$ are equipped with
  $\fffD$-sesquilinear forms induced by $\varpi_D^{-1}\innv{}{}$ and $\innv{}{}$
  respectively. 
  Clearly $\ell$ is $\epsilon_D\epsilon$-Hermitian and $\ell^*$ is $\epsilon$-Hermitian.
\item We have $G_{x}/G_{x,0^+} \cong \rU(\ell)\times \rU(\ell^*)$.
\end{enumerate}
 
The Witt classes of $\ell$ and $\ell^*$ are completely determined by the Witt
class of $V$ but independent of the choices of $L$ in $V$.  Indeed, the
anisotropic kernel of the Witt class of $\ell$ (resp. $\ell^*$) is equal to
$L_{\min}/ L_{\min}^* \fppD$ (resp. $L_{\max}^*/L_{\max}$) where $L_{\min}$
(resp. $L_{\max}$) is a minimal (resp. maximal) good lattice.

Let $\cT$ be a Witt class of $\epsilon$-Hermitian $D$-modules. 
Let $\sT$ and $\sT^*$ be the corresponding Witt classes of $\ell$ and $\ell^*$ determined by $\cT$. Then it is clear that there is a map 
\begin{equation}\label{eq:TTT}
\vcenter{
\xymatrix@R=0em@C=0em{
\Upsilon \colon& \sT\times
\sT^* \ar@{->>}[rr]&\hspace{2em} &  \cT \\
&(\ell,\ell^*) \ar@{|->}[rr] && V
}}
\end{equation}
such that $\dim_\fffD \ell + \dim_\fffD \ell^* = \dim_D V$, and $\ell$ and
$\ell^*$ are constructed from a vertex $x\in \BTB{\rU(V)}$.

\def\tsfG{\widetilde{\sfG}}

\begin{definition}\label{def:DZ.D}
\begin{asparaenum}[(a)]
\item \label{it:Z.type}
We say that a pair $(x,\rho)$ is a \emph{depth zero $K$-type} for $G := \rU(V)$ when $x$ is a
vertex in $\BTB{G}$ and $\rho$ is an irreducible $\sfG_x := G_x/G_{x,0^+}$-module.
 
\item Suppose  $\wtG$ is a certain central $\bC^\times$-covering of $G$ in (a).
  We say that a pair $(x,\trho)$ is a \emph{depth zero $K$-type} for $\wtG$ when $x$
  is a vertex in $\BTB{G}$ and $\trho$ is an irreducible 
$\tsfG_x := \wtG_x/G_{x,0^+}$-module.

\item We equip an equivalence relation on the set of depth zero $K$-types  by
  $G$-conjugacy.

\item \label{it:tZ.type} A $\wtG$-module $\tpi$ is said to have a depth zero
  $K$-type $(x,\trho)$ if $\trho$ occurs in $\tpi|_{\wtG_x}$.
\end{asparaenum}
\end{definition}

We warn that the depth zero $K$-type in \cref{it:Z.type} and \cref{it:tZ.type}
above is more general than the depth zero minimal $K$-type in \cite{MP1,MP2}
where~$\rho$ must be cuspidal.

Clearly, a depth zero supercuspidal datum $\Sigma = (x,0, \bfone, \rho)$
corresponds to the depth zero $K$-type $(x,\rho)$.  This gives an
embedding of the set of equivalence classes of depth zero supercuspidal datum to
the set of equivalence classes of depth zero $K$-types.  The image
precisely consists of the equivalence classes $[(x,\rho)]$ where $\rho$ is
cuspidal.

\subsubsection{}
We review some basic results of theta correspondences over a finite field. Let
$(\rU(\ell),\rU(\ell'))$ be a type~I reductive dual pair in a symplectic
$\fffF$-space $\bfW$ and let $\bomega_\bfW$ be the oscillator representation of
$\bfSp(\bfW)$ with respect to the additive character $\bpsi$
(cf. \Cref{sec:notation}).
  
\begin{definition}\label{def:theta.f}
  Let $\rho$ and $\rho'$ be irreducible representations of $\rU(\ell)$ and
  $\rU(\ell')$ respectively. Then $\rho$ and $\rho'$ are said to correspond with
  each other if $\rho\otimes \rho'$ is a submodule of
  $\bomega_\bfW|_{\rU(\ell)\times \rU(\ell')}$.  Such correspondence is not
  one-to-one in general, so we can only say that $\rho'$ is \emph{``a'' theta
    lift } of $\rho$.
\end{definition}

  If $\rho$ is cuspidal, then there is at most one $\rho'$ which
  corresponds to $\rho$. In particular, when restricted to cuspidal
  representations, theta lift still provide an one-to-one correspondence
  (cf. \cite[Section~3.IV.4]{MVW}).  In this case, we write
  $\rho' = \theta_{\ell,\ell'}(\rho)$.  

\subsubsection{} 
In the definition of lift of data, a zero dual pair (i.e. $\ell = 0$ or $\ell' = 0$)
may occur as the degenerate case\footnote{Similar notion also applies to dual pairs over other
  fields.}. 
\begin{definition}\label{def:DP.Z}
  A type~I reductive dual pair $(\rU(\ell),\rU(\ell'))$ defined over the field
  $\fffF$ is called a \emph{zero dual pair} if $\ell$ or $\ell'$ is the zero
  vector space.
\end{definition}
In this case, $\bfSp(\ell\otimes_\fffD \ell')$ degenerates to the trivial group
and the corresponding oscillator representation is the trivial representation.
Since the roles of $\ell$ and $\ell'$ are symmetric, we will assume that $\ell =
0$. Then $\rU(\ell)$ is the trivial group and it has only one
representation, namely, the trivial representation $\pi:=\bfone_{\rU(\ell)}$.
Now $\theta_{\ell,\ell'}(\pi)$ is the trivial representation of $\rU(\ell')$.
Note that, the trivial representation of $\rU(\ell')$ is cuspidal if and only if
$\ell'$ is anisotropic.

\subsubsection{} 
Fix a Witt tower $\cT'$ of $\epsilon'$-Hermitian $D$-modules. The discussion in
\Cref{sec:GoodL} also applies to $\cT'$ and we add primes $'$ to extend the
corresponding notations.

We fix $V'\in \cT'$ and a vertex $x'\in \BTB{G'}$.  Then
$(\rU(\ell),\rU(\ell'^*))$ and $(\rU(\ell^*), \rU(\ell'))$ form two reductive
dual pairs over the finite field $\fffD$. 
Here the zero dual pair may appear.

\def\llll{\diamond} 
\begin{definition} \label{def:liftzerodatum} Let $(x,\rho)$ and $(x',\rho')$ be
  two depth zero $K$-types of $G$ and $G'$ respectively.  We write
  $\rho = \rhol \boxtimes \rhols$ and $\rho' = \rholp \boxtimes \rholsp$ where
  $\rho_\llll$ is an irreducible $\rU(\llll)$-module with
  $\llll \in \set{ \ell,\ell^*,\ell',\ell'^*}$.  We say $(x',\rho')$ is a
  \emph{(theta) lift} of $(x,\rho)$ if $\rholsp$ is a theta lift of $\rhol$
    and $\rholp$ is a theta lift of $\rhols$.
\end{definition}

Now we are ready to state a theorem of Pan.

\begin{thm}[{\cite[Theorem~5.6]{Pan02J}}]\label{thm:Pan0}
Let $(G,G') = (\rU(V),\rU(V'))$ be a type~I dual pair over $F$ and $\tpi' =
\theta_{V,V'}(\tpi)$. Suppose $\tpi$ has a depth zero $K$-type $(x,\trho)$.
Then there exists a depth zero $K$-type $(x,\rho)$ and a theta lift $(x',\rho')$ of it
such that  (see \eqref{eq:SPT} for the definition of $\txi_{x,x'}$)
\begin{enumerate}[(a)]
\item \label{it:pan.1} $\trho = (\rho\boxtimes \idCx) \circ \txi_{x,x'}^{-1}$,
  and

\item \label{it:pan.2} $\tpi'$ has depth zero $K$-type $(x',\trho')$ where
  $\trho' := (\rho'\boxtimes \idCx)\circ \txi_{x,x'}^{-1}$. \qed
\end{enumerate}
\end{thm}

On the other hand, suppose $(x',\rho')$ is a theta lift of $(x, \rho)$. Let
$\trho$ and $\trho'$ be defined as in \Cref{thm:Pan0}~\cref{it:pan.1,it:pan.2}. 
Then the $\wtG_x\times \wtG'_{x'}$-module $\trho\boxtimes \trho'$ occurs in
$\sSB_{\sB_0}$ where $\sB = \sB_{x,x'}$.

\subsubsection{} 
Now let $\Sigma = (x,0,\bfone, \rho)$ be a depth zero supercuspidal datum of
$\rU(V)$ and let $\cT'$ be a fixed Witt tower of $\epsilon'$-Hermitian spaces.
Then we have $\ell$, $\ell^*$, $\sT'$, $\sT'^*$ defined in \Cref{sec:GoodL}.
Moreover, $\rho = \rhol\boxtimes \rhols$ with $\rhol$ and $\rhols$ cuspidal. 

Let $\rholsp:= \theta_{\ell,\sT'^*}(\rhol)$ be the theta lift of $\rhol$ such that
it is at the first occurrence, say $\ell'^*\in \sT'^*$, with respect to the Witt
tower $\sT'^*$. Likewise define $\ell'$ and the $\rU(\ell')$-module
$\rholp := \theta_{\ell^*,\sT'}(\rhols)$. Note that $\rholp \otimes \rholsp$ is
a cuspidal representation. By \eqref{eq:TTT}, let
$V' := \Upsilon'(\ell',\ell'^*) \in \cT'$ and let $x'$ be the corresponding
vertex in $\BTB{\rU(V')}$.

\begin{definition}\label{def:DTL.0}
Define $\Sigma':= (x', 0, \bfone, \rholp\boxtimes \rholsp)
\in \sD_{V'}$
and call it the \emph{theta lift} of $\Sigma$ with respect to the Witt tower
$\cT'$. 
Furthermore, the theta lift of data map $\dthetaVT$ is defined by 
\[
\dthetaVT([\Sigma]) = [\Sigma']
\] when restricted on the set of depth zero supercuspidal data.  
By an abuse of notation, we also write $\Sigma'' = \dtheta_{V,V'}(\Sigma) =  \dthetaVT(\Sigma)$ where
  $\Sigma''$ is any element in the equivalence class $[\Sigma']$.
\end{definition}

The next corollary follows immediately from the above discussion (cf. \cite[\Sec{9}]{Pan02J}).
\begin{cor}
The \Cref{thm:main} holds when restricted on the set of depth zero supercuspidal
representations. \qed
\end{cor}

\subsection{Theta lift of a single block of positive depth} \label{sec:DTL.pos}
Throughout this subsection, we assume that $\Sigma = (x, \Gamma, \phi, \rho)$ is
a single block datum of positive depth $r$ for $G = \rU(V)$. Let $s := r/2$ and
$\sL = \sL_x$.
Since $x$ is a point in $\BTB{G^0}$, we have
\crefformat{itGa}{{\rm (#2$\Gamma$#1#3)}}
\begin{enumerate}[({$\Gamma$}1)]
\item \label[itGa]{it:Ga.1} $\sL_{t-r} = \Gamma \sL_t$ for all
  $t\in \bR\sqcup\bR^+$, and
\item \label[itGa]{it:Ga.2} each element in $\Gamma+\fgg_{x,-r^+}$ is invertible
  with depth $-r$.
\end{enumerate}

\begin{definition}\label{def:VGa}
  For $\Gamma \in \fgg$ which is invertible in $\fgl(V)$, we define $V_\Gamma$ to be the
  $(-\epsilon)$-Hermitian $D$-module with the same underlying $D$-module as $V$
  and equipped with the form $\innvga{v_1}{v_2} := \innv{v_1}{\Gamma v_2}$ for
  $v_1,v_2\in V_\Gamma$.
\end{definition}

\begin{remark} In fact, there is an element
$w\in W:= \Hom_D(V,\VGa)$ such that $M(w) = \Gamma$. Let
$\iota \in \Hom_D(V,\VGa)$ be the identity map with respect to the underlying $D$-modules of $\VGa$
and $V$. Let $w := \iota$. Then $w^\mstar = \Gamma\circ \iota^{-1}$ and 
$M(w) = w^\mstar w = \Gamma$. 
\end{remark}

\subsubsection{} \label{sec:OBlattice} 
In this subsection, 
we let $V'$ be an $\epsilon'$-Hermitian $D$-module such
that 
\[
\dim_D V' = \dim_D V \quad \text{and} \quad \epsilon'=-\epsilon. 
\]
\begin{lemma} \label{lem:L'} Suppose that there is a $w \in \Hom_D(V,V')$ such
  that $M(w) \in \Gamma + \fgg_{x,-r^+}$.  By \cref{it:Ga.2}, $w$ is an isomorphism of
  $D$-modules. We define a lattice function in $V'$
  by
\begin{equation}\label{eq:sLp}
\sL'_t := w \sL_{t+s}.
\end{equation}
Then 
\begin{enumerate}[(i)]
\item \label{it:L'.1}$\Jump(\sL') = \Jump(\sL)-s$,
\item \label{it:L'.2} the lattice function $\sL'$ is self-dual and
\item \label{it:L'.3} $\sL'$ is the unique
self-dual lattice function on $V'$ such that $w \in (\sL\otimes_D \sL')_{-s}$.
\end{enumerate}
\end{lemma}

\begin{proof} 
\begin{asparaenum}[(i)]
\item This is clear from the definition of $\sL'$.
\item  
  For any $v\in V$,
  $\innvp{w v}{\sL'_t} = \innv{w v}{w \sL_{t+s}} = \innv{v}{w^\mstar
    w \sL_{t+s}} = \innv{v}{\Gamma \sL_{t+s}} = \innv{v}{\sL_{t-s}}$.
  Therefore $(\sL'_t)^* = (w \sL_{-(t-s)^+}) = \sL'_{-t^+}$,
  i.e. $\sL'$ is self-dual.
\item Clearly, $w \in (\sL\otimes_D \sL')_{-s}$ by the definition of $\sL'$.
  Suppose $\cksL'$ is another self-dual lattice function such that
  $w \in (\sL\otimes \cksL')_{-s}$. Then
  $\sL'_t = w \sL_{t+s} \subseteq \cksL'_t$ for all $t \in \bR$. Taking the dual
  lattices gives
  $\cksL'_{-t^+} = (\cksL'_t)^* \subseteq (\sL'_{t})^* = \sL'_{-t^+}$ for all
  $t\in \bR$. Hence $\sL'_t = \cksL'_t$.
  \qedhere
\end{asparaenum}
\end{proof}

\subsubsection{} \label{sec:surj}
The following proposition shows the surjectivity of moment maps on certain cosets. 
This is a key proposition in the single block of positive depth case. 

\def\wwzz{w_0}

\begin{prop}\label{lem:surj}
Let $\sL'$ be a self-dual lattice in $V'$ and $\sB := \sL\otimes_D\sL'$.  
Let $w\in \sB_{-s}$. Suppose that $M(w) \in \Gamma+\fgg_{x,-r^+}$.  Then
\[
M(w+\sB_{t}) = M(w)+\fgg_{x,-s+t} \qquad \forall t> -s.\]
\end{prop}
\begin{proof}
  We first prove the following claim.

  \setcounter{claim}{0}
  \begin{claim}\label{claim:surj1}
    The map
    \[ 
    w+\sB_t \rightarrow \left(M(w) + \fgg_{x,-s+t}\right)/\fgg_{x,-s+t^+}\subsetneq
    \fgg/\fgg_{x,-s+t^+}
    \]
    given by $w' \mapsto M(w') + \fgg_{x,-s+t^+}$ is a surjection.
  \end{claim}
	
  \begin{proof}
    Let $b\in \sB_{t}$. Since $t>-s$, we have 
    \begin{equation} \label{eq:Mw}
      \begin{split}
        M(w+b) &=  (w+b)^\mstar(w+b) 
        = w^\mstar w + w^\mstar b + b^\mstar w + b^\mstar b \\
        & \equiv M(w) + w^\mstar b+ b^\mstar w
        \pmod{\fgg_{x,-s+t^+}}.
      \end{split}
    \end{equation}
    On the other hand, by \Cref{lem:L'}, $X\mapsto w X$ gives an
    isomorphism $\fglDV_{x,s+t} \simrightarrow \sB_{t}$. Hence we assume that
    $b = w X$ for some $X\in \fglDV_{x,s+t}$.
    
    Pick a good element
    $\ckGamma \in M(w) + \fgg_{x,-r^+} = \Gamma +\fgg_{x,-r^+}$.  We have
    \begin{equation} \label{eq:Mw1}
    \begin{split}
      w^\mstar b+ b^\mstar w &= w^\mstar w X + (wX)^\mstar
      w =
      M(w) X + X^* M(w) \\
      & \equiv \ckGamma X + X^* \ckGamma \qquad \pmod{\fgg_{x,-s+t^+}}.
    \end{split}
    \end{equation}
    We set $\fglDV_{x,t_1:t_2} := \fglDV_{x,t_1}/\fglDV_{x,t_2}$ for $t_1<t_2$.  Now
    \Cref{claim:surj1} reduces to the following claim.
	 	
    \begin{claim} 
      The map
      $\beta\colon \fglDV_{x,s+t:s+t^+} \rightarrow \fgg_{x,-s+t:-s+t^+}$ defined
      by $X \mapsto \ckGamma X + X^* \ckGamma$ is surjective.
    \end{claim}

    \def\sinv{*,+1} 
    \def\ssinv{*,-1}
    
    Under the $*$-action,
    $\fglDV_{x,-s+t} = \fgg_{x,-s+t} \oplus \fglDV_{x,-s+t}^{\sinv}$ where
    $\fglDV_{x,-s+t}^{\sinv}$ is the $*$-invariant subspace.
    Under the $\ckGamma$-action, we have decomposition
    $\fgg = \ckfgg \oplus \ckfgg^\perp$ where $\ckfgg = \Cent{\fgg}{\ckGamma}$
    and $\ckfgg^\perp$ is the orthogonal complement of $\ckfgg$ in $\fgg$ under
    the $G$-invariant bilinear form $\bB$.  We also have a similar decomposition
    of $\fglDV$.  Since $\ckGamma$ is $*$-skew invariant, these decompositions
    are compatible with each other.

    First assume that $X \in \fgg_{x,-s+t}$, i.e. $X^* = -X$.  Then
    $\beta(X) = \ckGamma X - X \ckGamma = [\ckGamma,X]$.

    Now \cite[Lemma~2.3.4]{KM1} states that $X \mapsto [\ckGamma,X]$ induces an
    isomorphism
    \begin{equation}\label{eq:adGamma}
      \beta \colon \ckfgg^\perp_{x,s+t:s+t^+} \simrightarrow \ckfgg^\perp_{x,-s+t:-s+t^+}.
    \end{equation}
    Since
    $\fgg_{x,-s+t:-s+t^+} = \ckfgg_{x,-s+t:-s+t^+} \oplus
    \ckfgg^\perp_{x,-s+t:-s+t^+}$
    it is remains to show that $\ckfgg_{x,-s+t:-s+t^+}$ is in the image of
    $\beta$.  Let $\ckfgl = \Cent{\fglDV}{\ckGamma}$.  Suppose
    $X \in \ckfgl_{x,s+t}^{\sinv} \subseteq \fgl^{\sinv}_{x,s+t}$.  Then
    \begin{equation} \label{eq:2GammaX}
    \beta(X) = \ckGamma X + X^* \ckGamma = \ckGamma X + X\ckGamma = 2\ckGamma X.
    \end{equation}
    Therefore $\beta$ restrict to an an isomorphism
    $\ckfgl_{x,s+t}^{\sinv} \xrightarrow{\ \beta\ } \ckfgl_{x,-s+t}^{\ssinv} =
    \ckfgg_{x,-s+t}$. This proves Claim 2 and also Claim 1.
\end{proof}

  We now prove \Cref{lem:surj}. By \eqref{eq:Mw} we have
  $M(w+\sB_{t}) \subseteq M(w)+\fgg_{x,-s+t}$.  Fix an element
  $\gamma \in M(w)+\fgg_{x,-s+t}$. Clearly $M(w) \in \gamma + \fgg_{x,-s+t}$.
  Let $w_1 := w$ and $t_1 := t$. We construct sequences $\set{w_i}$ and
  $\set{t_i}$ inductively.  Suppose we have
  $M(w_i) \in \gamma + 
\fgg_{x,-s+t_i}$ for some $w_i \in w + \sB_{t}$. Apply
  the above claim with $w=w_i$ and $t=t_i$, we get a certain $b_i \in \sB_{x,t_i}$ such that
  $M(w_i+b_i) \in \gamma + \fgg_{x,-s+t_i^+}$.  Let $w_{i+1} := w_i+b_i$ and
  $t_{i+1} = \max\set{t| t >t_i, \fgg_{x,-s+t} = \fgg_{x,-s+t_i^+}} \in
  \Jump(\fgg_x)$
  where $\fgg_x$ denote the lattice function on $\fgg$ corresponding  to the
  Moy-Prasad filtration.  Clearly
  $M(w_{i+1}) \in \gamma + \fgg_{x,-s+t_i^+} = \gamma+ \fgg_{x,-s+t_{i+1}}$.
  Since $\Jump(\fgg_x)$ is a discrete set in $\bR$, $t_i \rightarrow \infty$
  and $w_{i}$ converges to some $w_\infty \in w + \sB_t$.  Since the moment
  map is continuous, we have
  $M(w_{\infty}) = \lim_{i\rightarrow \infty}M(w_i) = \gamma$. This proves the
  proposition.
\end{proof}

Now we present some corollaries of \Cref{lem:surj}.

\begin{cor} \label{cor:Gporbit} The set of self-dual lattice functions $\sL'$ in
  $V'$ such that
\[(\Gamma+\fgg_{x,r^+}) \cap M((\sL\otimes_D \sL')_{-s}) \neq \emptyset
\]
is a $G'$-orbit.
\end{cor}

\begin{proof}
  Let $\sL'$ and $\sL''$ be two self-dual lattice functions in the set.  Let
  $w \in (\sL\otimes \sL')_{-s}$ such that $M(w) \in \Gamma+\fgg_{x,r^+}$.  By
  \Cref{lem:surj}, we may assume that $M(w) = \Gamma$. By \Cref{lem:L'}
  $\sL'_{t} = w\sL_{t+s}$ for all $t$.  Similarly we pick
  a~$\ooww \in (\sL\otimes \sL'')_{-s}$ such that $M(\ooww) = \Gamma$ and
  $\sL''_t = \ooww \sL_{t+s}$ for all $t$. Note that $\Gamma$ is invertible. By
  Witt's theorem (see \cite[Section 1.11]{Dieu} and \cite[Thm 3.7.1]{Howe95}) there is $g'\in G'$ such that $\ooww = g' w$.  Hence
  $\sL''_t = \ooww \sL_{t+s} = g' w\sL_{t+s} = g' \sL'_{t}$ for all $t$.
\end{proof}

\def\clVp#1{{\fV'_{#1}}}
\def\clVpGamma{\clVp{\Gamma}}

We recall the definition of $V_\Gamma$ in \Cref{def:VGa}.

\begin{cor}\label{cor:cVpGamma}
The set of $\epsilon'$-Hermitian $D$-modules
\[
\clVpGamma = \set{V' | \dim_D V'= \dim_D V \text{ and } M^{-1}(\Gamma+\fgg_{x,-r^+})
  \neq \emptyset}
\]
is the isometry class of $V_\Gamma$. 
\end{cor}

\begin{proof}
By the remark after \Cref{def:VGa}, we see that $\clVpGamma$ contains $V_\Gamma$. 
Let $V' \in \clVpGamma$.  By \Cref{lem:surj} there exists a  $w\in \Hom(V,V')$ such that $M(w) = \Gamma$.
Now $\innvp{w v_1}{w v_2}= \innv{v_1}{\Gamma v_2} = \innvga{v_1}{v_2}$ for all $v_1, v_2 \in V$. 
In other words $w$ gives an isometry from $V_\Gamma$ to $V'$. This proves the corollary.
\end{proof}

\Cref{cor:cVpGamma} shows that $V_\Gamma$ is the unique $\epsilon'$-Hermitian
$D$-module $V'$ up to isometry such that $\dim_DV' = \dim_DV$ and there exists a
$w\in W$ such that $M(w) \in \Gamma +\fgg_{x,r^+}$.

\subsubsection{} \label{sec:singleblockdatum} Recall the notation in the
  beginning of \Cref{sec:DTL.pos}.
where $\Sigma = (x,\Gamma, \phi, \rho)$ is a single block supercuspidal datum of
positive depth~$r$.  We assume that $V'$ is isomorphic to $V_\Gamma$. Let
$W = V\otimes_D V'$ and fix a $w\in W$ such that $M(w) = \Gamma$. Let
$G' = \rU(V')$ and let $x' \in \BTB{G'}$ be the point corresponding to the
lattice function $\sL'$ defined by \eqref{eq:sLp}. We define
\[
 \Gamma' := M'(w) = w w^\mstar = w\Gamma w^{-1} \in \fgg'.
\]

\begin{prop}\label{lem:LiftGD}
Let $\Gamma = \sum_{i=-1}^d \Gamma_i$ be a $\GL_D(V)$-good factorization of
$\Gamma$ in $\fgg$
(which always exists by \Cref{prop:GD}). 
Let 
\[
\Gamma'_j := w\Gamma_j
  w^{-1}.
\] 
Then $\Gamma'_j \in \fgg'$ and $\Gamma' = \sum_{j=-1}^d \Gamma'_j$ is a
$\GL_D(V')$-good factorization of $\Gamma'$ in $\fgg'$.  
\end{prop}

\begin{proof}
  Since $\Gamma$ commutes  with $\Gamma_j$, we have 
\[
\begin{split}
(\Gamma'_j)^* &= (w^{-1})^\mstar \Gamma^*_j w^\mstar = 
  -(w^\mstar)^{-1}\Gamma_j w^\mstar 
  = -(w^\mstar)^{-1}\Gamma_j \Gamma w^{-1} \\
  &=
  -(w^\mstar)^{-1}\Gamma \Gamma_j w^{-1} 
= -w\Gamma_j w^{-1} = -\Gamma'_j.
\end{split}
\]
 This shows that $\Gamma'_j\in \fgg'$. 
 Note that $w\colon V \rightarrow V'$ is an isomorphism of $D$-modules. 
 Hence $\Gamma' = \sum_{j=-1}^d
  \Gamma_j'$ is a $\GL(V')$-good factorization.
\end{proof}

\begin{remarks}
We collect some easy consequences of \Cref{lem:LiftGD}.
\begin{enumerate}[1.]
\item By \Cref{lem:GLgood}, $\Gamma' = \sum_{j=-1}^d \Gamma'_j$ is also a
  $G'$-good factorization.

\item $\Gamma'$ satisfies  \Cref{def:SC.D}~(a) with respect to
  $G'$ and therefore
\[
G'^0 := \Cent{G'}{\Gamma'_{d}, \ldots, \Gamma'_{0}} = \Cent{G'}{\Gamma'}.
\]

\item We have an isomorphism 
\begin{equation}\label{eq:alpha.def}
\alpha \colon G^0\simrightarrow G'^0 \quad\text{defined by} \quad g\mapsto w g w^{-1}.
\end{equation}
Thanks to \Cref{lem:L'}, $\alpha$ restricted to an isomorphism $\alpha|_{G^0_x} \colon
  G^0_x\iso G'^0_{x'}$.

\item The point $x' \in \BTB{G'^0}$ is also a vertex. 
\end{enumerate}
\end{remarks}

Let $\phi' := \phi^*\circ \alpha^{-1}$ and $\rho' := \rho^*\circ \alpha^{-1}$
viewed as a character and a cuspidal representation of $G'^0_{x'}/G'^0_{x',0^+}$
respectively.  Clearly
\begin{equation}\label{eq:SC.L}
\Sigma'_{\Sigma,w}:= (x', -\Gamma', \phi', \rho')
\end{equation}
is a single block supercuspidal datum of positive depth $r$ for $G' = \rU(V')$. 

The following lemma shows that \eqref{eq:SC.L} is well-defined up to equivalence classes.
\begin{lemma}\label{lem:LD.wdef}
  Let $\Sigma := (x,\Gamma,\phi, \rho)$ be a single block datum of positive
  depth $r$. Let $w\in M^{-1}(\Gamma)$ and define $\Sigma'_{\Sigma,w}$ as in
  \eqref{eq:SC.L}. Then the equivalence class $[\Sigma'_{\Sigma,w}]$ is
  independent of the choice of the element in the equivalence class $[\Sigma]$ and $w$.
\end{lemma}

\begin{proof}
  First we fix $\Sigma = (x,\Gamma,\phi, \rho)$ in $[\Sigma]$.  For any
  $w\in M^{-1}(\Gamma)$, let $\Sigma'_{\Sigma,w}$ denote the datumn defined via
  \eqref{eq:SC.L}.  Since $M^{-1}(\Gamma)$ is a single $G'$-orbit, we see that
  elements in $\set{\Sigma'_{\Sigma,w} | w\in M^{-1}(\Gamma)}$ are
  $G'$-conjugates of one another.

  Suppose $\dotSigma = (\dotxx,\dotGamma,\dotphi,\dotrho) \in [\Sigma]$.  We
  will show that $\Sigma'_{\Sigma,w}$ and $\Sigma'_{\dotSigma,\dotww}$ are
  equivalent.  By $G$-conjugacy, we could assume $\dotxx = x$ and
  $\dotGamma = \Gamma + \gamma$ such that
  $\gamma \in \Cent{}{\fgg^0}\cap \fgg_{x,0} \subseteq F'[\Gamma]$ (see Remark
  after \Cref{def:SC.eq} and \Cref{lem:aniso}).
  
  \def\FGrr{F'[\Gamma]_{r}^*}
  \begin{claim*}
    Let $\FGrr:= F'[\Gamma] \cap \fgl(V)_{x,r}^{*,+1}$ be the set of elements in
    $F'[\Gamma]$ which are $*$-invariant and whose depth is not smaller than
    $r$.  Then there is an element $c \in \FGrr$ such that
    $M(w(1+c)) = \dotGamma$.
  \end{claim*}

  \begin{proof} By a similar calculation as in \eqref{eq:Mw} and \eqref{eq:Mw1},
    we have $M(w(1+c)) = \Gamma +\Gamma(2c+ c^2)$. Then $M(w(1+c)) = \dotGamma$
    if and only if $2c+ c^2 = \Gamma^{-1}\gamma$. Observe that
    $F'[\Gamma]\cap \fgl(V)^{*,+1}$ is a product of non-Archimedean local
    field(s) and $\FGrr$ is an ideal in its integral ring.  Since $p\neq 2$, the
    map $\FGrr \mapsto \FGrr$ defined by $c \mapsto 2c+c^2$ is a bijection by
    Hensel's lemma. Now the claim follows because $\Gamma^{-1}\gamma \in \FGrr$.
  \end{proof}

Let $c$ be the element given by the above claim and $\dotww := w(1+c)$. 
It is straightforward to check that $\dotGamma' = \Gamma' + \gamma'$ with
$\gamma' := w (2c+c^2) w^\mstar \in \Cent{\fgg'}{\Gamma'}\cap
\fgg'_{x',0}$. Moreover $G'^0$, $x'$, $\alpha$, $\phi'$, $\rho'$ are exactly the
same objects for $w$ and $\dotww$.  \trivial[h]{ $\gamma' \in \fgg'_{x',0}$ is
  clear. It is in $\Cent{\fgg'}{\Gamma'}$ because
  $[\Gamma',\dotGamma'] = \Gamma' w (1+c)(1+c^*) w^\mstar - w (1+c)(1+c^*)
  w^\mstar \Gamma' = w[\Gamma,(1+c)(1+c^*)] w^\mstar = 0$.  $x'$ is independent
  of $w$ or $\dotww$ is follows from (the proof of) \Cref{lem:L'}~(iii).
  $\alpha_{\dotww}(g) = w(1+c) g (1+c)^{-1} w^{-1} = w g w^{-1} = \alpha_w(g)$
  for any $g\in G^0$. Hence $G'^0$ is independent since it is the image of
  $\alpha$ (it also follows from $[\Gamma', \dotGamma']$ with some argument).
  Now $\phi'$ and $\rho'$ are independent since $\alpha$ and $G'^0$ are
  physically the same.} In other words, $\Sigma'_{\Sigma,w}$ and
$\Sigma'_{\dotSigma,\dotww}$ are equivalent.  This completes the proof of the
lemma.
\end{proof}

\begin{definition} \label{def:LD.pos}
We retain the notation in \Cref{lem:LD.wdef}. The isomorphism class of
$V'$ is independent of the choice of the element in $[\Sigma]$ by \Cref{cor:cVpGamma}.
We define $\dthetap([\Sigma])$ to be the equivalence class $[\Sigma'_{\Sigma,w}]
\in \bDVp$.
By an abuse of notation, we will also write $\Sigma' = \dthetap(\Sigma)$ where
$\Sigma' := \Sigma'_{\Sigma,w}$ and $w$ is
implicitly fixed.
\end{definition}

\subsection{The general case} \label{sec:changesplitting2} Let
$\Sigma= (x,\Gamma, \phi, \rho)$ be a supercuspidal datum of $G := \rU(V)$.  By
\Cref{def:BD}, let $\Sigma = \bigoplus_{l=0}^b \llSigma$ be the block
decomposition of $\Sigma$ into $b$ positive depth blocks
$\set{\llSigma | 1\leq l \leq b}$ and a depth zero block $\zzSigma$.
In addition, we have $\Gamma = \bigoplus \llGamma$ and $V
= \bigoplus \llV$.

For any $\epsilon'$-Hermitian $D$-module $V'$, let $[V']$ represent its Witt
class in the Witt group. 

\begin{definition}\label{def:LD}
Let $\cT'$ be a fixed Witt class of $\epsilon'$-Hermitian $D$-modules. We recall $\bDTp$ in~\eqref{eqbDTp}.
We set
\begin{enumerate}[(a)]
\item $\llSigmap := 
\dthetap(\llSigma)\in \sD_{\llV_{\llGamma}}$ for $1 \leq l \leq b$;
\item $\zzcTp := \cT' - \sum_{l=1}^{b} [\llV_{\llGamma}]$ and
\item $\zzSigmap := \dtheta_{\zzV,\zzcTp}(\zzSigma)$ (cf. \Cref{def:DTL.0}).
\end{enumerate}
Then we define 
\[
\dthetaVT([\Sigma]) :=  \left[ \textstyle\bigoplus_{i=0}^{b} \llSigmap \right]\in \bDTp. 
\]
By an abuse of notation again, we also write $\Sigma' = \dthetaVT(\Sigma)$ where
$\Sigma' := \bigoplus_{i=0}^{b} \llSigmap$.
\end{definition}

\begin{remarks}
\begin{asparaenum}[1.]
\item Note that the $\llSigmap$s have different depths. It follows from
  \Cref{lem:DS.D} that
  $\Sigma' := \dthetaVT(\Sigma) = (x',-\Gamma', \phi', \rho')$ is a
  supercuspidal datum of $\rU(V')$ for a well-defined $[V'] \in \cT'$.

\item 
  In the construction, we get an element $\llww\in \llV\otimes_D\llVp$ such
  that $\llGamma = M(\llww)$ and $\llGamma' = M'(\llww)$ for
  each $0< l\leq b$. Therefore we get an element 
\begin{equation}\label{eq:w.gen}
w := \bigoplus_{0<l\leq b} \llww \in
  \bigoplus_{0<l \leq b} \llV\otimes_D\llVp \subseteq V\otimes_D V'
\end{equation}
so that $\Gamma \equiv M(w) \pmod{\fgg_{x,0}}$ and
$\Gamma' \equiv M'(w) \pmod{\fgg'_{x',0}}$.

\item In the above definition of $\dtheta_{V,\cTp}$, the key is the
  correspondence of semisimple elements via the
  moment maps.  We expect an explicit description of the correspondences between
  cuspidal representations of dual pairs over finite fields using similar
  construction.  Indeed there are some partial results in this aspect by
    Pan \cite{PanU,Pan16}.

\item \label[rmk]{rmk:dt.genD} The discussions in
  \Cref{sec:SC,sec:GoodFact,sec:LD} extend to $K$-type data defined in the
    remark to \Cref{def:SC.D}. More precisely, the notion of $K$-type data
  extends to the covering group and the notions of equivalence
  relation, block decomposition, direct sum etc. extend under exactly the same
  definition as well.

\item   Suppose $\Sigma$ (resp. $\tSigma$) is a $K$-type data for $G$
  (resp. $\wtG$). Then $\etaSigma$ and $\tetaSigma$ are also defined in the same
  way. 
\end{asparaenum}
\end{remarks}

\begin{definition}\label{def:LDC} 
  A $K$-type datum $\Sigmap$ is a \emph{theta lift of} a supercuspidal
    datum~$\Sigma$ for the dual pair $(\rU(V),\rU(V'))$, if
  \begin{enumerate}[(a)]
  \item $\Sigma = \bigoplus_{l=0}^b \llSigma$ is a block
    decomposition of a $K$-type datum;
    
  \item $\llSigmap = \dthetap(\llSigma)$ for $1\leq l\leq b$;
    
  \item $\zzSigmap$ is a (not necessary supercuspidal) depth zero data  which is
    a theta lift of $\zzSigma$ (cf. \Cref{def:liftzerodatum});
    
  \item $V' = \bigoplus_{l=0}^b \llVp$ and $\Sigmap := \bigoplus_{l=0}^b \llSigmap$.
  \end{enumerate}
\end{definition}

\subsection{An example}\label{sec:eg.Gen}
To illustrate the content of the definitions made above, we provide the
following example which could be considered as a generic case:
 \begin{eg*} Let $\tSigma$ be a supercuspidal datum of $G$ such that
      $\zzV$ is the zero space under the block decomposition
      (cf. \Cref{sec:BD}).  Equivalently, this means every eigenvalue of
      $\Gamma$ over $\barF$ has negative valuation when we view $V$ as an
      $F$-vector space and $\Gamma$ as an $F$-linear map on $V$.  Since $\Gamma$
      is invertible, we let $V_\Gamma$ denote the $\epsilon'$-Hermitian space in
      \Cref{def:VGa}.  \def\VTG{V^\circ_{\cT'-[V_\Gamma]}} Let $\VTG$ be the
      anisotropic $\epsilon'$-Hermitian space in the Witt tower
      $\cT'- [V_\Gamma]$.  Then the first occurrence of $\tpiSigma$ is at
      $V':=V_\Gamma \oplus \VTG$ in the Witt tower $\cT'$.  Using this explicit
      formula, one may check the conservation relation~\cite{SZ} of the first
      occurrence indices directly in this case.  If $\cT' = [V_\Gamma]$, then
      $\vartheta_{V,\cT'}(\Sigma)$ is essentially the ``contragredient'' of
      $\Sigma$ (cf. \eqref{eq:SC.L}).  Otherwise, $\vartheta_{V,\cT'}(\Sigma)$
      is the direct sum of $\vartheta_{V,[V_\Gamma]}(\Sigma)$ and the datum
      attached to the trivial representation of $\rU(\VTG)$. 
  \end{eg*}

\section{One positive depth block case I: orbit structure}\label{sec:OB1}

\subsection{Assumptions}
\label{sec:assumptions}
Throughout this section, we retain the notation in \Cref{sec:DTL.pos} and make
the following assumptions. 
\begin{enumerate}[(I)]
\item Let $\Sigma = (x, \Gamma, \phi, \rho)$ be a single block datum with
  positive depth $r = 2s$. In particular $\Gamma$ is an invertible element in $\End(V)$.
  We fix a $\GL(V)$-good factorization $\Gamma = \sum_{i=-1}^d \Gamma_i$.

\item The space $(V', \innvp{}{})$ is isomorphic to
  $(V_\Gamma, \inn{}{}_{V_\Gamma})$ and $w \in \Hom_D(V,V')$ is a fixed element
  such that $M(w) = \Gamma$. In particular $\dim_D V = \dim_D V'$.
\end{enumerate}

\def\IGamma{{\fI_\Gamma}}

We make following definitions:

\begin{definition} \label{def:dalpha}
\begin{enumerate}[(a)]
\item \label{it:Gamma.cases}
We refer to Cases I and II in \Cref{def:GD}. 
If we are in Case I, i.e. $\Gamma_d = 0$, then we set $\IGamma := \set{0, \cdots, d-1}$ and $\ckGamma = \Gamma_{d-1}$.
Otherwise if we are in Case II, i.e. $\Gamma_d \neq 0$, then we set
$\IGamma:=\set{0, \cdots, d}$ and $\ckGamma = \Gamma_d$. 
Under this definition $\ckGamma$ is a nonzero good element in $\Gamma +
\fgg_{x,-r^+}$. Let $\ckG := \Cent{G}{\ckGamma}$.

\item \label{it:PB.2} Define $G^i$ with its Lie algebra $\fgg^i$ as in \Cref{def:GD}.
  In particular $G^0 = \Cent{G}{\Gamma}$.

\item \label{it:def.fggiperp}Let $\fgg^{i-1 \perp}$ be the orthogonal complement
  of $\fgg^{i-1}$ in $\fgg^i$ with respect to the invariant bilinear form $\bB$
  in \Cref{sec:CG}, i.e.  $\fgg^i = \fgg^{i-1} \oplus \fgg^{i-1 \perp}$. Let
  $\fgg^{i \perp}_{x,r} = \fgg^{i \perp} \cap \fgg_{x,r}$.

\item \label{it:PB.3} Let $\fgl :=\fgl^d := \fgl(V)$ and
  $\fgl^i = \Cent{\fgl^{i+1}}{\Gamma_i}$ for $0\leq i<d$.

\item \label{it:PB.4} Let $\sL$ be the self-dual lattice function on $V$
  corresponding to $x$.  Let $\sL'$ be the self dual lattice function in $V'$
  defined by $\sL'_t = w \sL_{t+s}$ as in \Cref{lem:L'} and let $x'$ be
  the corresponding point in $\BTB{G'^0}$ .

\item \label{it:PB.5} Let $\Gamma' = M'(w)$ and $\Gamma' = \sum_{i=-1}^{d}\Gamma'_i$ be the good
  factorization of $\Gamma'$ given by \Cref{lem:LiftGD}. 
  
\item \label{it:PB.6} We define similar notations for $G'$ as in \crefrange{it:PB.2}{it:PB.3}.

\item Let $\sB := \sL\otimes_D \sL'$. Then $w \in \sB_{-s}$ by
  \Cref{lem:L'}~\cref{it:L'.3}.

\item \label{it:defalpha.iso} Let $\alpha\colon G^0 \simrightarrow G'^0$ be the group isomorphism given by
  $\alpha(g) = w g w^{-1}$ (cf. \eqref{eq:alpha.def}).

\item Let $\fgg'^\perp := \fgl(V')^{*,+1}$. Then $\fgl':= \fgl(V') =
  \fgg' \oplus \fgg'^\perp$ under the $*$-action. 

\item For each $X \in \fgg$, we define $\dalpha\colon \fgg\to \fgg'$ and $\dalphap \colon \fgg \to \fgg'^\perp
$ by
\[
\dalpha(X) = \half(wXw^{-1} - (wXw^{-1})^*) \quad \text{and} \quad
\dalphap(X) = \half(wXw^{-1} + (wXw^{-1})^*).
\] 
Clearly $wXw^{-1} = \dalpha(X) + \dalphap(X)$. 
\end{enumerate}
\end{definition}

\subsection{Structure of orbits}
\def\odalphat{\odalpha|_{\fgg^i_{x,t:t^+}}}
We apply \Cref{def:SC.N}~(a)-(e) to data $(x, \Gamma)$ and $(x',\Gamma')$.  The
purpose of this section is to study the $K\times K'$-orbit of the coset
$w+\sB_0$ in $W/\sB_0$.

\subsubsection{}
We start by investigating some properties of $\dalpha$ and $\dalphap$ by
elementary linear algebra.
\begin{lemma}\label{lem:alpha}
  For any $t\in \bR$, we set $t'_i := r -r_{i-1}+t$. Then the following statements hold:
\begin{enumerate}[(i)]
\item \label{it:dalpha.fil} 
  $\dalpha(\fgg_{x,t}^i) \subseteq \fgg'^i_{x',t}$ for $0\leq i\leq d$.

\item \label{it:dalpha.p} $\dalphap(\fgg^i_{x,t}) \subseteq \fgl'_{x',t'_i}$ for
  $0 < i \leq d$.

\item $\dalpha \colon \fgg^0 \rightarrow \fgg'^0$ is an isomorphism which is the
  differential of $\alpha$ and $\dalphap(\fgg^0) = 0$.


\item \label{it:dalpha.JG1}
  $\dalpha(\fgg^i_{x,t} \setminus \fgg^i_{x,t^+}) \subseteq
  \fgg'^i_{x',t} \setminus \fgg'^i_{x',t^+}$
  for $i \in \IGamma$. 
  
\item \label{it:dalpha.JG2} The map
  $\odalphat \colon \fgg^i_{x,t:t^+} \longrightarrow \fgg'^i_{x',t:t^+}$ induced
  by $\dalpha |_{\fgg^i_{x,t}}$ is an isomorphism for $i\in \IGamma$. Hence,
  $\dalpha |_{\fgg^i_{x,t}} \colon \fgg^i_{x,t} \longrightarrow \fgg'^i_{x',t}$
  and $\dalpha|_{\fgg^i}\colon \fgg^i \longrightarrow \fgg'^i$ are also
  isomorphisms for $i\in \IGamma$.

\item \label{it:dalpha.JG3} 
  $\dalpha(\fgg^{i \perp}) \subseteq \fgg'^{i\perp}$ for
  $0\leq i<d$ and so $\dalpha(\fgg^{i\perp}_{x,t}) \subseteq
  \fgg'^{i\perp}_{x',t}$ by \cref{it:dalpha.fil}. 
\end{enumerate}
\end{lemma}

\begin{proof}
We set $\tX := wXw^{-1} = \dalpha(X) + \dalphap(X)$. By the definition of
  $\sL'$, it is clear that $\tX \in \fgl'_{x',t}$ if and only if $X\in \fgl_{x,t}$.
Note that $(w^{-1})^\mstar = (w^\mstar)^{-1}$ and  $\tX^* = (wXw^{-1})^* =
(w^\mstar)^{-1}X^*w^\mstar$.   

\begin{asparaenum}[(i)]
\item Suppose $X\in \fgg^i$.  Then
  $[\tX,\Gamma'_l] = [wXw^{-1}, w\Gamma_l w^{-1}] = w [X,\Gamma_l] w^{-1} = 0$
  for all $i \leq l \leq d$, i.e. $\tX \in \fgl'^i$. Similar argument gives
  $\tX^* \in \fgl'^i$.  
  \trivial[h]{ Since $\Gamma'_i$ is $*$-skew invariant,
    $*$-stabilize $\fgl'^i$.  In fact,
    $[X^*, \Gamma'_i] = - [X^*, \Gamma'^*_i] = [X,\Gamma'_i]^*$. Hence,
    $X^* \in \Cent{\Gamma'_i}{\fgl'^{i+1}} = \fgl'^i$ for
    $X\in \Cent{\Gamma'_i}{\fgl'^{i+1}} = \fgl'^{i}$. }  
  Now
  $\dalpha(X) = (\tX - \tX^*)/2 \in \fgg'^i$. This shows that
  $\dalpha(\fgg^i) \subseteq \fgg'^i$.  Since $\tX^* \in \fgl_{x',t}$ for
  $X \in \fgg_{x,t}$, $\dalpha(\fgg_{x,t}^i) \subseteq \fgg'^i_{x',t}$.

\item Let $X \in \fgg_{x,t}^i$. Then
\[
\begin{split}
\dalphap(X) = & \half( wXw^{-1} + (wXw^{-1})^*) =  \half (wXw^{-1} -
(w^\mstar )^{-1}Xw^\mstar )\\
= & \half (w^\mstar )^{-1}(w^\mstar w X - X w^\mstar  w) w^{-1} = \half
(w^\mstar )^{-1}[\Gamma,X] w^{-1} \\ 
= & \half (w^\mstar )^{-1}(\sum_{l=-1}^{d} [\Gamma_l,X] )w^{-1} 
= \half (w^\mstar )^{-1}(\sum_{l=-1}^{i-1} [\Gamma_l,X] )w^{-1}.
\end{split}
\]
\def\tttt{\tilde{t}}
Note that, for any $\tttt \in \bR$, we have  $w^{-1} \sL'_{\tttt} = \sL_{\tttt+s}$,
$(w^{\mstar})^{-1} \sL_{\tttt} = \sL'_{\tttt+s}$,
$\Gamma_l \sL_{\tttt} \subseteq \sL_{\tttt-r_l}$ and
$X \sL_{\tttt} \subseteq \sL_{\tttt+t}$ (cf. \Cref{sec:OBlattice}). Hence
$\dalphap(X) \in \fgl'_{r - r_{i-1} + t}$. 

\item This is clear from the definition of $\alpha$, $G^0$
  and $G'^0$ (cf. \Cref{def:dalpha} \cref{it:PB.2} and \cref{it:defalpha.iso}).

\item If $X \in \fgg_{x,t}^i \setminus \fgg_{x,t^+}^i$, then
  $\tX = wX w^{-1} \in \fgl'_{x',t} \setminus \fgl'_{x',t^+}$. Since $r > r_{i-1}$, 
    $\dalphap(X) \in \fgl'_{x',t^+}$ by \cref{it:dalpha.p}. Therefore,
    $\dalpha(X) = \tX - \dalphap(X) \in \fgl'_{x',t} \setminus \fgl'_{x',t^+}$. 
    
  \item The injectivity of $\odalphat$ is a restatement of \cref{it:dalpha.JG1}.
    The surjectivity follows by dimension counting.  Note that the roles of $G$
    and $G'$ are symmetric. We could define
    $\dalpha'\colon \fgg' \rightarrow \fgg$ in the same way.  \trivial[h]{By
      viewing $w$ as an element in $\Hom_D(V',V)$.}  Now
    \[
    \xymatrix@R=0em{\fgg^i_{x,t:t^+}\ar[r]^{\odalpha}& \fgg'^i_{x',t:t^+}
      \ar[r]^{\odalphaprime} &\fgg^i_{x,t:t^+}}
    \]
    where $\odalpha$ and $\odalphaprime$ are injections. Hence
    $\odalphaprime \circ \odalpha$ is an injection,
    $\dim_{\fff}\fgg^i_{x,t:t^+} = \dim_{\fff} \fgg'^i_{x',t:t^+}$ and
    $\odalphat$ is an isomorphism.  The claim for $\dalpha|_{\fgg^i_{x,t}}$ and
    $\dalpha|_{\fgg^i}$ immediately follows.
    
      
\item 
Note that $(w^\mstar)^{-1}\Gamma_i w^\mstar = (w^\mstar)^{-1}w^{-1} \Gamma'_i w
w^\mstar = \Gamma'^{-1} \Gamma'_i \Gamma' = \Gamma'_i$.
Therefore, for all $X \in \fgg$,
\begin{equation}\label{eq:dalpha.ad}
\begin{split}
\dalpha\circ \ad_{\Gamma_{i}}(X) &= \half(w[\Gamma_{i},X]w^{-1} - 
(w^\mstar)^{-1}[\Gamma_i,X]^* w^\mstar)\\
&= \half([\Gamma'_{i},\tX] - 
(w^\mstar)^{-1}[\Gamma_i,-X] w^\mstar)\\
& = \half([\Gamma'_{i},\tX] - [\Gamma'_i, -(w^\mstar)^{-1}X w^\mstar])\\
& =  \ad_{\Gamma'_{i}}\circ \dalpha(X). 
\end{split}
\end{equation}
Now \cref{it:dalpha.JG3} follows since 
 $\fgg^{i\perp}$ and $\fgg'^{i\perp}$
  are the sums of non-trivial isotypic components in $\fgg^{i+1}$ and
  $\fgg'^{i+1}$ under the actions of $\ad_{\Gamma_i}$ and $\ad_{\Gamma'_i}$ respectively. 
\qedhere  
\end{asparaenum}
\end{proof}

\subsubsection{}

We define symplectic forms on $\fgg$ and $\fgg'$ respectively by\footnote{Do not confuse with $\inn{}{}_{V_\Gamma}$ on $V_\Gamma$ in \Cref{def:VGa}.}
\begin{equation}\label{eq:innGa.def}
  \begin{split}
    \innGa{X_1}{X_2} & = \bB([X_1,X_2],\Gamma)
    \qquad \forall X_1, X_2, \in \fgg \quad \text{and}  \\
    \innGap{X'_1}{X'_2} & = \bB([X'_1,X'_2],-\Gamma') \quad \forall X'_1, X'_2
    \in \fgg'.
\end{split}
\end{equation}
We equip $\fgg\oplus\fgg'$ with the form $\innGa{}{}\oplus \innGap{}{}$.

\begin{lemma} \label{lem:isometry} The map  
$\iota \colon \fgg\oplus \fgg'\longrightarrow W$ given by
\[
(X,X') \mapsto X\cdot w + X'\cdot w = -wX + X'w
\]  
is an isometry. 
\end{lemma}

\begin{proof} Let $X,X_1,X_2\in \fgg$ and $X',X'_1,X'_2 \in \fgg'$. Then
 \[
\begin{split}
\innGa{X_1}{X_2} &= \bB([X_1,X_2],\Gamma) =  \half \trF(X_1X_2\Gamma-X_2X_1\Gamma) \\
& = \half\trF(X_1X_2\Gamma-(X_2X_1\Gamma)^*) = \half\trF(X_1 X_2\Gamma + \Gamma X_1X_2) \\
& = \trF(X_1X_2 w^\mstar w) = -\trF((-wX_2)^\mstar (-wX_1))\\
& = -\innw{X_2 \cdot w}{X_1 \cdot w} = \innw{\iota(X_1)}{\iota(X_2)}.
\end{split}
\]

Similarly, we have
\[
\begin{split}
\innGap{X'_1}{X'_2} &=  -\bB([X'_1,X'_2],\Gamma') = - \trF(X'_1X'_2\Gamma') 
= - \trF(X'_1X'_2 w w^\mstar) \\ 
&= \trF((X'_1w)^\mstar X'_2w)
=  \innw{X'_1\cdot w}{X'_2\cdot w} = \innw{\iota(X'_1)}{\iota(X'_2)}.
\end{split}
\]

On the other hand
\[
\begin{split}
  \innw{\iota(X)}{\iota(X')} &= \innw{-wX}{X'w} = \trF((-wX)^\mstar X'w)
= \trF(Xw^\mstar X'w) \\
&= \trF(w^\mstar X' w X) = \trF((X'w)^\mstar(-wX)) =
\innw{\iota(X')}{\iota(X)} \\
&= -\innw{\iota(X)}{\iota(X')}.
\end{split}
\]
Hence $\innw{\iota(X)}{\iota(X')} = 0$. Therefore, $\iota(\fgg)$ and $\iota(\fgg')$ are orthogonal to each other. 
\end{proof}

\subsubsection{}\label{sec:radical}
Let $\binnGa{}{}$ denote the $\fff$-symplectic form on $\fgg_{x,s:s+}$ induced by $\innGa{}{}$.
\trivial[h]{
Obviously, $\binnGa{}{}$ and $\overline{\inn{}{}}_{\ckGamma}$
  are the same. 
}
Let $\frr := \Rad(\binnGa{}{})$ be the radical of $\binnGa{}{}$ in $\fgg_{x,s:s^+}$. 
Likewise we define $\binnGap{}{}$ and $\frr' := \Rad(\binnGap{}{})$.

\def\dtriangle{{\mathrm{d}\triangle}}
\def\dtriangle{{\blacktriangle}}
\def\bdtriangle{{\overline{\dtriangle}}}
\def\bdalpha{{\overline{\dalpha}}}

\begin{lemma} \label{lem:radical}
\begin{enumerate}[(i)]
\item \label{it:radical.1} In Case I (i.e. $\Gamma_d = 0$), we have $\frr = \fgg^{d-1}_{x,s:s^+}$. In
  Case II (i.e. $\Gamma_d \neq 0$), we have $\binnGa{}{} \equiv 0$ and
  $\frr = \fgg_{x,s:s^+}$.

\item \label{it:radical.2} Likewise $\frr' = \fgg'^{d-1}_{x',s:s^+}$ in Case I and $\frr' = \fgg'_{x',s:s^+}$ in Case II.
\end{enumerate} 
\end{lemma}
 
\begin{proof}
  For $X_1,X_2\in \fgg_{x,s}$, 
  \[
  \innGa{X_1}{X_2} = \bB([X_1,X_2],\Gamma) = - \bB(X_2,[X_1,\Gamma]).
  \]
  Hence $\frr = \set{X_1 +\fgg_{x,s^+} | \ad_\Gamma(X_1) \in \fgg_{x,-s^+}}$. However $\ad_\Gamma |_{\fgg_{x,s:s^+}} =
  \ad_{\ckGamma} |_{\fgg_{x,s:s^+}}$. 
  We get the conclusion in Case~I by \eqref{eq:adGamma}. 
  In Case~II, we have $\ckGamma = \Gamma_d \in \Cent{}{\fgg}$. So $[X_1,\Gamma]\equiv
  [X_1, \ckGamma] \equiv 0 \pmod{\fgg_{x,-s^+}}$ and $\innGa{X_1}{X_2} 
  \in \fppF$. 
\end{proof}

\subsubsection{}\label{sec:lem.S1}

We recall that $\bfbb :=\sB_{0:0^+}$ is an $\fff$-symplectic space. By
\Cref{lem:isometry}, $\iota$ induces an isometry
\[
\biota\colon \fgg_{x,s:s^+}\oplus \fgg'_{x',s:s^+}\longrightarrow \bfbb.
\]
Define $ \dtriangle \colon \fgg \longrightarrow \fgg\oplus \fgg'$ by
$X \mapsto (X,\dalpha(X)).$ It induces an injection
\begin{equation} \label{eq:bdtriangle}
\xymatrix{
\bdtriangle \colon \ggss \ar@{^(->}[r]&  \ggss\oplus \ggssp.
}
\end{equation}

Let $\bfbbpp := \biota(\frr\oplus \frr')$. Then $\bfbbpp$  is the radical of
$\Im(\biota)$. 
Let $\bfbbz := \bfbbpp^\perp/\bfbbpp$.

\begin{lemma}\label{lem:S1}
The following statements hold.
\begin{enumerate}[(i)]
\item \label{it:lem.S1.1} $\bdalpha\colon \frr \rightarrow \frr'$ is an isomorphism.

\item \label{it:lem.S1.2} $\xymatrix{0 \ar[r] & \frr \ar[r]^<>(.5){\bdtriangle} & \ggss \oplus
  \ggssp \ar[r]^<>(.5){\biota} & \bfbb}$ is exact.

\item \label{it:lem.S1.3} 
  $\bfbbpp = \biota(\frr) = \biota(\frr')$.

\item \label{it:lem.S1.4} $\dim_\fffF \ggss + \dim_\fffF \ggssp= \dim_\fffF \bfbb$.

\item \label{it:lem.S1.5} $\bfbbpp^\perp = {\mathrm{Im}}(\biota)$. 

\item \label{it:lem.S1.6} In Case I, the composition of maps
\[
\xymatrix{\biotabz\colon \fgg^{d-1\perp}_{x,s:s^+}\oplus \fgg'^{d-1\perp}_{x',s:s^+} \ar@{^(->}[r]& 
  \fgg_{x,s:s^+}\oplus \fgg'_{x',s:s^+}
  \ar[r]^<>(.5){\biota} &\bfbbpp^\perp \ar@{->>}[r]&\bfbbz}
\]
is an isomorphism of non-degenerate symplectic spaces over $\fffF$.

\item In Case II,
  ${\mathrm{Im}}(\biota) = \bfbbpp = \bfbbpp^\perp$ is a maximal isotropic subspace of $\bfbb$ and
  $\bfbbz = 0$.
\end{enumerate}
\end{lemma}

\begin{proof}
  Throughout this proof, we let $(X,X')\in \ggs\oplus \ggsp$ and we let
  $(\barX,\barX')$ denote its image in $\fgg_{x,s:s^+}\oplus \fgg'_{x',s:s^+}$.
\begin{enumerate}[(i)]
\item This follows from \Cref{lem:alpha}~\cref{it:dalpha.JG2} and \Cref{lem:radical}.

\item Let  $\barX \in \frr \subseteq \fgg_{x,s:s^+}$. Then
\[
\begin{split}
  \biota \circ \bdtriangle(\barX) &\equiv -wX + \dalpha(X)w = (-wXw^{-1} +
  \dalpha(X))w \\
  &= - \dalphap(X) w \pmod{\sB_{0^+}}.
\end{split}
\] 
Since $r > r_{d-2}$ in Case~I (resp. $r > r_{d-1}$ in Case~II),
\Cref{lem:alpha}~\cref{it:dalpha.p} implies $\dalphap(X) \in
\fgl'_{x',s^+}$.
Hence $\dalphap(X) w\in \sB_{0^+}$.  This proves
$\bdtriangle(\frr) \subseteq \Ker \biota$.

Now we show the opposite inclusion. 
Suppose $\biota(\barX,\barX')=0$.
Since $\biota$ is an isometry, $(\barX,\barX') \in \frr\oplus \frr' =
\Rad(\binnGa{}{}\oplus \binnGap{}{})$. Now
$\bdtriangle(\barX) = (\barX,\bdalpha(\barX)) \in \Ker \biota$ so 
$\biota(0,\barX'-\bdalpha(\barX)) = 0$. Note that 
$\biota|_{\fgg'_{x',s:s^+}}$ is an injection.
\trivial[h]{This follows from $\sL'_t = w\sL_{t+s}$. In fact, $X' \mapsto X'w$
  induces an isomorphism $\fgl_{x',t} = \bigcap_{l}\Hom(\sL'_l,\sL'_{l+t})
  \rightarrow  \bigcap_{l}\Hom(\sL_{l+s},\sL'_{l+s +(t-s)}) = \sB_{t-s}$}
This implies $\barX' =
\bdalpha(\barX)$ and proves the exactness of (ii). 

\item
  By (ii), $\biota(\barX) = -\biota(\bdalpha(\barX))$ for $\barX \in \frr$.
  Now $\biota(\frr) = \biota(\frr')$ by part (i). 

\item  Since $\sL'_{t} = w\sL_{t+s}$, we have $\fgl_{x,s:s^+} \cong \bfbb$ by $X \mapsto
  wX$.
Therefore, 
\begin{equation}\label{eq:Dim1}
\dim_\fff\bfbb = \dim_\fff \fgl_{x,s:s^+}.
\end{equation}

\def\bareta{{\overline{\eta}}} 
\def\sinvariant{{*,+1}}
\def\ssinvariant{{*,-1}} 

Consider the isomorphism $\eta\colon \fgl \simrightarrow \fgl'$ defined by
$X \mapsto w^\mstar X w$.  Since $(w^\mstar)^\mstar = -w$, we have
$\eta(X)^* = (w^\mstar X w)^* = -w^\mstar X^* w = \eta(-X^*)$.  By reducing to the
residue field, $\bareta \colon \fgl_{x,s:s^+} \simrightarrow \fgl'_{x',-s:-s^+}$
induces an isomorphism 
\[
\bareta \colon \fgl_{x,s:s^+}^{\sinvariant} \simrightarrow \fgl'^{\ssinvariant}_{x',-s:-s^+} = \fgg'_{x',-s:-s^+} \cong \Hom_{\fffF}(\ggssp,\fffF).
\] 
Therefore
\begin{equation}\label{eq:Dim2}
\dim_\fffF \fgl_{x,s:s^+} = \dim_\fffF \ggss + \dim_\fffF \fgl_{x,s:s^+}^{\sinvariant}
 = \dim_\fffF \ggss + \dim_\fffF  \ggssp.
\end{equation}
Combining \eqref{eq:Dim1} and \eqref{eq:Dim2} yields \cref{it:lem.S1.4}. 

\item 
 Note that $\biota$ is an isometry.  So
    $\Im(\biota) \subseteq \biota(\frr)^\perp =\bfbb_+^\perp$ by
    \cref{it:lem.S1.3}. Since $\biota|_{\fgg_{x,s:s^+}}$ is an injection and
    $\bfbb$ is a non-degenerate symplectic space,
    $\dim \bfbb_+^\perp = \dim \bfbb - \dim \bfbb_+ =
    \dim \fgg_{x,s:s^+} \oplus \fgg'_{x',s:s^+} - \dim \frr = \dim\Im(\biota)$
    by \cref{it:lem.S1.4} and \cref{it:lem.S1.2}.  
\item By \Cref{lem:radical}~\cref{it:radical.1},
    $\fgg^{d-1\perp}_{x,s:s^+} \oplus \fgg'^{d-1\perp}_{x',s:s^+}$ is a maximal
    non-degenerate symplectic subspace of
    $\fgg_{x,s:s^+}\oplus \fgg_{x',s:s^+}$. Now the claim follows from
    \cref{it:lem.S1.5}, since $\biota$ is an isometry.
\item This follows by a similar argument as that of \Cref{it:lem.S1.6}
    using \Cref{lem:radical}~\cref{it:radical.2}.
\end{enumerate}
\end{proof}

\subsubsection{}
We begin with definitions which will be used in the rest of the section. 
\begin{definition}
In order to simplify the notation, let 
\begin{enumerate}[(a)]
\item $\dbK := K\times K'$, $\dbK_{0^+} = K_{0^+}\times K'_{0^+}$ and
  $\dbK_+ := K_+ \times K'_+$;
\item $\dbG^i_{s_{i-1}} := G^i_{x,s_{i-1}} \times G'^i_{x',s_{i-1}}$ and
  $\dbG^i_{s_{i-1}^+} := G^i_{x,s_{i-1}^+}\times G'^i_{x',s_{i-1}^+}$ for
  $0<i\leq d$; and
\item  $\dbG^0 :=
  G^0_{x} \times G'^0_{x'}$ and $\dbG^0_{0^+} :=
  G^0_{x,0^+}\times G'^0_{x',0^+}$. 
\savemyenumi
\end{enumerate}
We define
\begin{enumerate}[(a)]
\resumemyenumi
\item
  $S := \Stab_{\dbK}(w+\sB_0) = \Set{(h,h')\in \dbK | (h, h')\cdot w \in w
    +\sB_0}$,

\item $S^i := S\cap \dbG^i_{s_{i-1}}$ and $S^i_+ :=
  S \cap \dbG^i_{s_{i-1}^+}$ for $i>0$, 

\item $S^0 := S\cap \dbG^0$, $S_+ :=  S \cap \dbK_+$ and $S_{0^+} :=
  S \cap \dbK_{0^+}$.
\savemyenumi
\end{enumerate}
We extend the group isomorphism $\alpha \colon G^0_x \iso G'^0_{x'}$ in
\Cref{def:dalpha}~\cref{it:defalpha.iso} to a map
\begin{equation} \label{eq:alpha}
\alpha  \colon G_{x}^0 G_{x,0^+} \longrightarrow G'^0_{x'} G'_{x',0^+}  \quad
\text{by} \quad g \exp(X) \mapsto (wgw^{-1}) \exp(\dalpha(X))
\end{equation}
for all $g\in G^0_{x}$ and $\exp(X) \in G_{x,0^+}$. This map $\alpha$ is
well-defined because $w\exp(X)w^{-1} = \exp(\dalpha(X))$ for
  $\exp{X}\in G^0_x \cap G_{x,0+} = G^0_{x,0^+}$. We warn that the map
$\alpha$ is not necessarily a group homomorphism.  We define
\begin{enumerate}[(a)]
\resumemyenumi
\item 
$\triangle^0 := \set{(g,\alpha(g)) | g\in G^0_{x}} \cong G^0_{x} \cong
G'^0_{x'}$,

\item $\triangle^0_+ := \set{(g,\alpha(g)) | g\in G^0_{x,0^+}} \cong G^0_{x,0^+}
  \cong G'^0_{x',0^+}$,

\item $\triangle^i := \set{(g,\alpha(g)) | g\in G^i_{x,s_{i-1}}}$ for $i > 0$ and

\item  $ \triangle^{i}_+ := \set{(g,\alpha(g)) | g\in G^i_{x,s_{i-1}^+}}$ for $i>0$.
\end{enumerate}
\end{definition}

By \Cref{lem:alpha}, the map $\alpha$ restricts to a map\footnote{We reiterate
  that $\alpha$ is not necessarily a group homomorphism}
$\alpha|_{G^i_{x,s_{i-1}}} \colon G^i_{x,s_{i-1}} \rightarrow G'^i_{x',s_{i-1}}$
for $1 \leq i \leq d$ and this restricted map is a bijection
if $i \in \IGamma$.  Therefore $\triangle^i \subseteq \dbG_{s_{i-1}}^i$ and
$\triangle_+^i \subseteq \dbG_{s_{i-1}^+}^i$ for all $0 \leq i\leq d$.

Obviously $G_{x,s} G'_{x',s}\subseteq S$. 
The following is a key lemma which claims that $S$ is generated by these sets. 

\begin{lemma}\label{lem:SS}
\begin{enumerate}[(i)]
\item \label{it:lem.SS.1}If $(g,g') \in \triangle^0$, then $(g,g') \cdot w = w$. Moreover,
  $\iota$ and $\biota$ are $\triangle^0$-equivariant.
\item \label{it:lem.SS.2}
Suppose $0< i \in \IGamma$. Let $g = \exp(X)$ with $X\in \fgg^i_{x,s_{i-1}}$.
Then $\dalphap(X)w \in \sB_{0^+}$,  $[\dalpha(X), \dalphap(X)]w\in \sB_{s}$ and 
\[
(g,\alpha(g)) \cdot w \subseteq w - \dalphap(X) w -\half  [\dalpha(X), \dalphap(X)]w+\sB_{s^+} \subseteq w+ \sB_{0^+}.
\]
In particular, $\triangle^i \subseteq S$.

\item 
\label{it:lem.SS.3}
$S\cap G = G_{x,s}$ and $S\cap G' = G'_{x',s}$.

\item 
\label{it:lem.SS.4}
For each $i \in \IGamma$, $S^i =\triangle^i G'^i_{x',s} =  \triangle^i G^i_{x,s}G'^i_{x',s}$.

\item  
\label{it:lem.SS.5}
In Case~I, i.e. $\Gamma_d = 0$ and $r=r_d=r_{d-1}$,
  \begin{align*}
    S &= \triangle^0 \triangle^1 \cdots \triangle^{d-1} G_{x,s} G'_{x,s} \quad
         \text{and} \\
    S_+&= \triangle^0_+ \triangle^1_+ \cdots \triangle^{d-1}_+
          G_{x,s}^{d-1} G'^{d-1}_{x',s} G_{x,s^+}G'_{x',s^+}.
  \end{align*}
  In Case~II, i.e. $\Gamma_d \neq 0$ and $r=r_d>r_{d-1}$,  
\begin{align*} 
S &= \triangle^0 \triangle^1 \cdots \triangle^{d-1} \triangle^d G_{x,s} G'_{x,s}
     \quad \text{and} \\
  S_+&= \triangle^0_+ \triangle^1_+ \cdots \triangle^{d-1}_+ \triangle^d_+ G_{x,s} G'_{x',s}.
\end{align*}
\end{enumerate}
\end{lemma}

\begin{proof}
\begin{asparaenum}[(i)]
\item Clearly, $(g,g') \cdot w = g' w g^{-1} = \alpha(g) w g^{-1} = w g w^{-1} w g^{-1} = w$. 
The claim that $\iota$ and $\biota$ are $\triangle^0$-equivariant also follows
by a straightforward computation which we will leave to the
reader.
\item 
Let $X' := \dalpha(X)$, $Y' := \dalphap(X)$ and $g' := \exp(X')$. Then $(g, g')=(g,\alpha(g)) \in \triangle^i$.
Note that $r>r_{i-1}>0$ and so $r - s_{i-1}>s$,.
By \Cref{lem:alpha}~(i) and (ii), 
$X' \in \fgl'_{x',s_{i-1}}\subseteq \fgl'_{x',0^+}$ and $Y' \in
\fgl'_{x',r-s_{i-1}} \subseteq  \fgl'_{x',s^+} \subseteq \fgl'_{x',0^+}$. Therefore $X'Y', Y'X' \in \fgl'_{x',r}$, $Y' w \in \sB_{0^+}$ and $[X', Y'] w \in \sB_s$.   
By Zassenhaus formula, 
\[
\begin{split}
  ((g,g')\cdot w)w^{-1} = & g'wg^{-1}w^{-1} = \exp(X')\exp(-wXw^{-1})
  = \exp(X')\exp(-X'-Y') \\
  = & \exp(X') \exp(-X')\exp(-Y') \exp(-\half [X',Y']) g'_{r^+}  \\
  = & \exp(-Y') \exp(-\half [X',Y'])g'_{r^+}.
\end{split}
\]
where $g'_{r^+} \in \GL(V')_{x',r^+}$.  Hence
$(g,g')\cdot w -w \equiv \exp(-Y') \exp(-\half[X',Y']) w -w \equiv -Y'w -
\half[X',Y']w \pmod{\sB_{s^+}}$.  This finishes the proof of (ii).

\item We only prove the identity for $G'$ and the proof for $G$ is similar.
  Note that $T\mapsto T w^{-1}$ induces an isomorphism
  $\sB_{t} \rightarrow \fgl'_{x',t+s}$. Therefore, for $g'\in G'$,
  $g' w + \sB_0 = w+\sB_0$ if and only if $g'\in 1+\fgl'_{x',s}$. Hence
  $S \cap G' = \GL(V')_{x',s}\cap G' = G'_{x',s}$.
  
\item By \cref{it:lem.SS.1} to \cref{it:lem.SS.3} ,
  $S^i \supseteq \triangle^i G^i_{x,s} G'^i_{x',s} \supseteq \triangle^i
  G'^i_{x',s}$.
  It remains to show $S^i \subseteq \triangle^i G'^i_{x',s}$.  Indeed suppose
  $(g,g') \in S^i\subseteq \dbG^i_{s_{i-1}}$.  Then
  $(1,\alpha(g)^{-1}g')= (g, \alpha(g))^{-1} (g,g') \in S^i\cap G'^i \subseteq
  G'^i_{x',s}$
  by \cref{it:lem.SS.3}.  Hence $(g,g') \in \triangle^i G'^i_{x',s}$.
  
\item Let $S' = \triangle^0 \cdots \triangle^d G_{x,s} G'_{x',s}$. Note that
  $\triangle^d \subseteq G_{x,s} G'_{x',s}$ in Case I. 
	By (i)-(iv), $S \supseteq S'$. Observe that the projection to the first factor
  $\pr_1 \colon S' \rightarrow K$ is surjective.  Suppose $(g,g') \in S$. Then
  there is a $g''\in K'$ such that $(g,g'')\in S'$.  Therefore
  $(1,g''^{-1}g') \in S\cap G' = G'_{x,s}$ by (iii).  Hence
  $S \subseteq S' G'_{x',s} = S'$.  This proves that $S = S'$. Intersecting $S$
  with $\dbK_+$ gives~$S_+$.  \qedhere
\end{asparaenum}
\end{proof}

\subsection[A key identity of psi]{A key identity of $\psi$}
We recall the function $\psi_\Gamma$ in \eqref{eq:psi.Ga} which is clearly well-defined on
$\dbK_{0^+}$.

\begin{lemma}\label{lem:psi}
Suppose $1 \leq i\in \IGamma$ and $(g,g') \in \triangle^i$ or $(g,g')\in
G_{x,s}\times G'_{x',s}$. Then
\begin{equation}\label{eq:psi}
\psi(\half\innw{w}{(g,g')^{-1}\cdot w-w})= \psi_{\Gamma}(g)\psi_{-\Gamma'}(g').
\end{equation}
\end{lemma}

\begin{remark} 
  For $(g,g')\in \triangle^0$,
  $\psi(\half\innw{w}{(g,g')^{-1}\cdot w-w}) = 1$ by
  \Cref{lem:SS}~\cref{it:lem.SS.1}.
\end{remark}

\begin{proof}
  Suppose $(g,g') = (\exp(X),\exp(X'))\in \triangle^i$ where
  $X\in \fgg^i_{x,s_{i-1}}$ and $X' = \dalpha(X)$. Let $Y' := \dalpha^\perp(X)$.
  By \Cref{lem:SS}~\cref{it:lem.SS.2},
  $(g,g')^{-1}\cdot w - w \equiv Y'w - \half[X',Y']w \pmod{\sB_{s^+}}$.

We claim that $\half \innw{w}{[X',Y']w} \in \fppF$. Indeed by \Cref{lem:inn},
\[
\begin{split}
 \half \innw{w}{[X',Y']w} & = \bB(\Gamma',[Y',X'])= \bB([\Gamma',X'],Y')\\
& \equiv \bB([\ckGamma',X'],Y') \quad  \pmod{\fppF}\\
& =  0  \quad \text{(because $X'\in \Cent{\fgg'}{\ckGamma'}$ by
  \Cref{lem:alpha}~\cref{it:dalpha.fil} and $i\in \IGamma$)}.
\end{split}
\]

Note that $wXw^{-1} = X'+Y'$ and so we have
\[
\begin{split}
&\psi_{\Gamma}(g^{-1})\psi_{-\Gamma'}(g'^{-1})\psi(\half\innw{w}{(g,g')^{-1}\cdot
   w-w}) \\
 & =  \psi \left(\bB(\Gamma,-X) + \bB(-\Gamma',-X') + \half\innw{w}{Y'w} 
 -\frac{1}{4}\innw{w}{[X',Y']w}\right) \\
 & =  \psi\left(\half \trF(-w^\mstar w X) + \half \trF(w w^\mstar X') + \half \trF(w^\mstar Y' w) \right) \\
 & =  \psi\left(\half\trF(-ww^\mstar (wXw^{-1}) + ww^\mstar (X' + Y')) \right)\\
 & =  \psi(0) = 1.
\end{split}
\]
This proves the lemma for $(g,g') \in \triangle^i$.

Next suppose $(g,g') := (\exp(X),\exp(X')) \in G_{x,s}G'_{x',s}$.  Then
\[
  \begin{split}
    \psi(\half \innw{w}{(g,g')^{-1}\cdot w - w}) & = \psi(\bB(X,M(w))+\psi(\bB(X',-M'(w))
    \quad
    \text{(by \eqref{eq:Sym.s})} \\
    & = \psi_{M(w)}(g) \psi_{-M'(w)}(g') = \psi_{\Gamma}(g)\psi_{-\Gamma'}(g').
\end{split}
\]
This proves the lemma.
\end{proof}

\subsection{A maximal totally isotropic subspace}\label{sec:iso}
We refer to \Cref{sec:Special} for the notation. Also see \cite[p.~591]{Yu} and
\cite[Section~12]{Kim}.
Suppose $1\leq i\in \IGamma$. Let $J^i := (G^{i-1},G^i)_{x,(r_{i-1},s_{i-1})}$ and
$J^i_+ := (G^{i-1},G^i)_{x,(r_{i-1},s_{i-1}^+)}$.  Likewise we have subgroups $J'^i$
and $J'^i_+$ in $G'$.  Let $\dbJ^i := J^i \times J'^i$ and
$\dbJ^i_+ := J^i_+ \times J'^i_+$.  Note that $\exp$ induces a group isomorphism
$\fgg^{i-1\perp}_{x,s_{i-1}:s_{i-1}^+} \simrightarrow J^i/J^i_+$ and we identify
both sides from now on.  We let
$\bfW^i:= J^i/J^i_+ = \fgg^{i-1\perp}_{x.s_{i-1}:s_{i-1}^+}$,
$\bfW'^i := J'^i/J'^i_+ = \fgg'^{i-1\perp}_{x',s_{i-1}:s_{i-1}^+}$ and
$\dbbfW^i := \dbJ^i/\dbJ^i_+ = \bfW^i\times \bfW'^i$.  Note that $\dbbfW^i$ has
a natural non-degenerate symplectic space structure over $\fffF$ induced by
$\innGa{}{}\oplus \innGap{}{}$ (cf. \cref{eq:innGa.def}).  Let $\bfD^i$ be the image of
$\triangle^i\cap \dbJ^i$ under the natural quotient map
$\dbJ^i \twoheadrightarrow \dbbfW^i$.  By \Cref{lem:alpha},
$\triangle^i\cap \dbJ^i = \set{(\exp(X),\exp(\dalpha(X))|X\in
  \fgg^{i-1}_{x,r_{i-1}} \oplus \fgg^{i-1\perp}_{x,s_{i-1}}}$.
Although $\triangle^i\cap \dbJ^i$ is not a subgroup of $G\times G'$, $\bfD^i$ is
an $\fffF$-subspace in $\dbbfW$ isomorphic to
$\fgg^{i-1\perp}_{x,s_{i-1}:s_{i-1^+}}$.

\begin{lemma}\label{lem:Di}
  For $1\leq i \in \IGamma$, $\bfD^i$ is a maximal totally isotropic subspace of
  $\dbbfW^i$.
\end{lemma}
\begin{proof}
  By \Cref{lem:alpha}~\cref{it:dalpha.JG2},
  \[\dim_\fffF \fgg^{i-1\perp}_{s_{i-1}:s_{i-1}^+} =
  \dim_{\fffF}\fgg^{i}_{s_{i-1}:s_{i-1}^+}/\fgg^{i-1}_{s_{i-1}:s_{i-1}^+}
  =\dim_{\fffF}\fgg'^{i}_{s_{i-1}:s_{i-1}^+}/\fgg'^{i-1}_{s_{i-1}:s_{i-1}^+} =
  \dim_\fffF \fgg'^{i-1\perp}_{s_{i-1}:s_{i-1}^+}.
  \]
  Therefore, by \Cref{lem:alpha}~\cref{it:dalpha.JG3}, $\bdalpha$ induces an
  isomorphism
  $\fgg^{i-1\perp}_{s_{i-1}:s_{i-1}^+} \simrightarrow
  \fgg'^{i-1\perp}_{s_{i-1}:s_{i-1}^+}$
  and $\dim_{\fffF}\bfD^i = \half \dim_{\fffF}\dbbfW^i$.  It remains to show that
  $\bfD^i$ is isotropic.  Let $X_1,X_2\in \fgg^{i-1\perp}_{s_{i-1}}$ and let
  $Y_1 = \dalpha(X_1), Y_2 = \dalpha(X_2)$. Then the symplectic form is given by
  \[
  \begin{split}
    \inn{(X_1,Y_1)}{(X_2,Y_2)} & \equiv \bB([X_1,X_2],\Gamma_{i-1}) +  \bB([Y_1,Y_2],-\Gamma'_{i-1}) \\
    & \equiv \bB(X_2,\ad_{\Gamma_{i-1}}(X_1)) - \bB(Y_2,\ad_{\Gamma'_{i-1}}(Y_1))\\
    & \equiv \bB(X_2,\ad_{\Gamma_{i-1}}(X_1)) -
    \bB(Y_2,\dalpha(\ad_{\Gamma_{i-1}}(X_1)))
    \quad \text{(by \eqref{eq:dalpha.ad})} \\
    & \equiv \bB(X_2,\ad_{\Gamma_{i-1}}(X_1)) -
    \bB(wX_2w^{-1},w(\ad_{\Gamma_{i-1}}(X_1))w^{-1}) \\
    & \qquad \text{(by $r-r_{i-1}>0$, \Cref{lem:alpha}~\cref{it:dalpha.fil} and
      \cref{it:dalpha.p}) }\\
    & \equiv 0 \quad \pmod{\fppF}.
  \end{split}
  \]
  This finishes the proof. 
  \trivial[h]{ For the second last and last equation: Let $A = \dalphap(X_2)$,
    $B=\dalpha([\Gamma_{i-1},X_1])$ and $C = \dalphap([\Gamma_{i-1},X_1])$.
    Then $A \in \fgg_{X,r-r_{i-1}+s_{i-1}}$,
    $B\in \fgg_{x,s_{i-1}-r_{i-1}} = \fgg_{x,-s_{i-1}}$ and
    $C \in \fgg_{x,r-r_{i-1}-s_{i-1}}$. Now $AB,Y_2C \in \fgl_{x,r-r_{i-1}}$ and
    $AC \in \fgl_{x,2(r-r_{i-1})}$. Hence
  \[
  \begin{split}
   & \bB(X_2,[\Gamma_{i-1},X_1]) =\bB(wX_2w^{-1}, w[\Gamma_{i-1},X_1]w^{-1}) \\
    & = \half \trF((Y_2+A)(B+C)) \equiv \half \trF(Y_2B) = \bB(Y_2,
    \dalpha([\Gamma_{i-1},X_1])).
  \end{split}
  \] 
  }
\end{proof}

\subsection{Triviality of $\chi^{\fbb_+}$}
\def\chibp{\chi^{\bfbb_+}}
We recall the $\fffF$-vector space $\bfbbpp = \biota(\frr) = \biota(\frr')$ in \Cref{sec:lem.S1}. 
The space $\bfbbpp$ is an isotropic subspace in $\bfbb$ and $\triangle^0$ acts on it. 
Let $\chi^{\bfbb_+}$ be the character of $\triangle^0$ as defined in \Cref{sec:HW}. 
More precisely, $\chibp(g,\alpha(g)) = \det((g,\alpha(g))|_{\bfbbpp})^{(q-1)/2}$ where $(g,\alpha(g)) \in \triangle^0$
and $q = \abs{\fffF}$. 

\begin{lemma}\label{lem:chib}
We have $\chi^{\bfbb_+}(g,\alpha(g)) = 1$ for all $(g,\alpha(g)) \in \triangle^0$. 
\end{lemma}

The rest of this section is devoted to the proof of the above lemma.  The proof
does not affect the rest part of the paper, so the reader may skip it without
loss of continuity.

First we introduce some notation.
Suppose $\fff'$ is an extension of $\fffF$ and $\fV$ is an $\fff'$-module. 
Let $\sfG$ be a group acting $\fff'$-linearly on $\fV$. 
Let $\det_\fffF(g|_{\fV})$ denote the determinant of $g\in \sfG$ when we view $\fV$ as an $\fffF$-vector space. 
Let $\chi^{\fV}_{\fffF}$ be the character $g \mapsto \det_\fffF(g|_\fV)^{(\abs{\fffF}-1)/2}$. 
More conceptually, $\chi^{\fV}_{\fffF}(g)$ is $1$ if $\det_\fffF(g|_\fV)$ is a
square in $\fffF^\times$ and is $-1$ if otherwise, i.e. it is the Legendre symbol of
$\det_\fffF(g|_{\fV})$ in $\fffF$. 

\medskip

Note that $G_x^0 \cong \triangle^0$ via $g\mapsto (g,\alpha(g))$ and $\frr \cong
\bfbbpp$ via $X \mapsto \biota(X)$.
Clearly for $g \in G_x^0$ and $X \in \frr$,
\[
\biota(g \cdot X) = -w g X g^{-1} = -\alpha(g) w X g^{-1} =  (g,\alpha(g)) \cdot \biota(X)
\]
so
$\chi^{\frr}_{\fffF}(g) = (\det_\fffF(g|_{\frr}))^{(q-1)/2} =
\det_\fffF((g,\alpha(g))|_{\bfbbpp})^{(q-1)/2}= \chibp((g,\alpha(g))$.  Then
\Cref{lem:chib} is equivalent to the following statement:
\begin{equation}\label{eq:chib.1}
\chi^{\frr}_{\fffF}(g) = 1 \qquad \forall\; g \in G_x^0. 
\end{equation}


Recall $\ckGamma$ and $\ckG$ in \Cref{def:dalpha}~\cref{it:Gamma.cases}. Note that $G_x^0$ is a subgroup
of $\ckG_x$ and $\ckG_x$ acts on $\frr = \ckfgg_{x,s:s^+}$.  Hence
\eqref{eq:chib.1} is a consequence of the following:
\begin{equation}\label{eq:chib.2}
\chi^{\frr}_{\fffF}(g) = 1 \qquad  \forall\; g \in \ckG_x. 
\end{equation}
If we replace $\Gamma$ by $\ckGamma$, then every object in \eqref{eq:chib.2} remains unchanged.
Therefore, we could assume that $\Gamma = \ckGamma$. In this case $\ckG_x
= G_x^0$.

\medskip
We recall 
that $F' = \Cent{}{D}$. We consider Case I and Case II separately.
\def\Fci{F^\circ_i}
\begin{enumerate}[1.]
\item
In Case I, $F'[\Gamma] = \prod_i F_i$ is a product of fields $F_i$ with
involution $*$ and $V = \bigoplus_i V^i$ such that each $V^i$ is a certain
Hermitian space over $F_i$. 
Let $\Fci$ be the $*$ fixed subfield of $F_i$. 
Then $\rU(V^i)$ is an algebraic group defined over $\Fci$. 
Under this decomposition, $G^0= \prod_i \rU(V^i)$, $x = (x_i)\in \prod_i\BTB{\rU(V^i)}$ and $\chi^{\frr}_{\fffF} = \prod_i \chi^{\fuu(V^i)_{x_i,s:s^+}}_{\fffF}$.

The residue field $\fff_i^\circ$ of $\Fci$ could be a finite extension of
$\fffF$.  Note 
$\chi^{\fuu(V^i)_{x_i,s:s^+}}_{\fff_i^\circ}(g) = 1$ means that
$\det_{\fff_i^\circ}(g) = a^2$ is a square in $(\fff_i^\circ)^\times$. Hence
$\det_\fffF(g) = \Norm_{\fff_i^\circ/\fffF} \circ \det_{\fff_i^\circ}(g) =
(\Norm_{\fff_i^\circ/\fffF}(a))^2$ is a square in $\fffF^\times$. So
$\chi^{\fuu(V^i)_{x_i,s:s^+}}_{\fffF}(g) = 1$. 
In order to prove \eqref{eq:chib.2}, it suffices to check
that $\left.\chi^{\fuu(V^i)_{x_i,s:s^+}}_{\fff_i^\circ}\right|_{\rU(V^i)_{x_i}}$ is trivial for each~$i$.

\item
In Case II, $G^0 = G$ is a unitary group over $D = F' = F[\Gamma]$.
\end{enumerate}

\trivial[h]{
Note that  $\Gamma \in D$ in Case II and $\Gamma|_{V^i} \in F_i$ in Case I. 
We see that $r\in \val(D)$ and $r\in \val(F_i)$ respectively.
}

To summarize, we have reduced \Cref{lem:chib} to the following claim. 

\def\chigse{\chi^{\fgg_{x,s:s^+}}_{\fffF}}
\begin{claim*}
 Suppose 
\begin{enumerate}[(a)]
\item 
$D$ is a quadratic field extension of a certain $p$-adic field $F$;

\item $\tau$ is the nontrivial element in $\Gal(D/F)$;

\item $V$ is a $D$-vector space with a Hermitian form $\inn{}{}$; 

\item $G = \rU(V)$, $\fgg = \fuu(V)$ and $x\in \BTB[F]{G}$;

\item $\Gamma$ is an element in $\Cent{}{\fgg} = D^{\tau,-1}$ with valuation $-r = -2s$.
\end{enumerate}
Then $\chigse(g) = 1$ for all $g \in G_x$. 
\end{claim*} 

\begin{proof}[Proof of the Claim]
Without loss of generality, we may assume $\val(D) = \bZ$. 
Let $\sL$ be the self-dual lattice function corresponding to $x$. 
Define the $\fff_D$-space $\sfL_t := \sL_t/\sL_{t^+}$.

\begin{enumerate}
\item 
If $D/F$ is unramified, we let $\varpi_D$ be the fixed uniformizer
of both $D$
and $F$. 
The residue field $\fff_D$ is a quadratic extension of $\fffF$. 
Moreover $\sfL_0$ and $\sfL_{\half}$ (possibly zero spaces) are Hermitian spaces
over~$\fff_D$.

\item
If $D/F$ is ramified then $\fff_D = \fffF$. We fix a uniformizer $\varpiD$ of $D$ such that $\varpiD^\tau = - \varpiD$.
Then $\sfL_0$ is an orthogonal space over $\fffF$ whose form is induced by $\inn{}{}$ and $\sfL_{\half}$ is a symplectic space over $\fffF$ whose form is induced by $\varpi_D^{-1}\inn{}{}$. 
\end{enumerate}
Note that $\chigse$ factors
through the group 
\begin{equation}\label{eq:sfG0}
\sfG := G_{x}/G_{x,0^+} = \left( \prod_{t\in \Jump(\sL)\cap (0,\half)} \GL_{\fff_D}(\sfL_t) \right) \times \rU(\sfL_0)\times \rU(\sfL_{\half}).
\end{equation}

\def\chiggs{\chi^{\fgg_{x,s:s^+}}_{\fffF}}

Now we consider two separate cases in the next two subsections. 

\subsubsection{Case 1: $s \in \val(D)=\bZ$.} 
First we claim that $D$ is an unramified extension of $F$. Indeed, if $D/F$ is
ramified, then $\val(F) = 2 {\mathbb{Z}}$ and $-r = \val(\Gamma)$ is odd because $\Gamma
\in D^{\tau,-1}$. This implies $s \not \in \bZ$, a contradiction.

Now $X \mapsto \varpiD^s X$ gives a $\sfG$-equivariant isomorphism $\fgg_{x,0:0^+} \simrightarrow \fgg_{x,s:s^+}$.
Therefore $\det_{\fffF}(g|_{\fgg_{x,s:s^+}}) = \det_{\fffF}(g|_{\fgg_{x,0:0^+}}) = 1$ since all the simple factors of $\sfG$ in \eqref{eq:sfG0} are of type A acting on its Lie algebra via adjoint action. 
Note that $\GL_{\fff_D}(\sfL_{t})$ should be viewed as a group defined
over $\fffF$ by restriction of scalars, but this does not affect the conclusion. 
Hence we have proved the claim in this case.

\def\HomfD{\Hom_{\fffD}}
\subsubsection{Case 2:} $s \not \in \val(D) = \bZ$. 
Then $s=\frac{r}{2} \in \frac{1}{2} \bZ \setminus \bZ$. We recall that 
\begin{equation}\label{eq:gls}
\fgl_{x,s:s^+} = \bigoplus_{t\in \bQ/\bZ} \HomfD(\sfL_t,\sfL_{t+s}).
\end{equation}
The adjoint action $*$ (cf. \Cref{sec:CG}) permutes the terms in \eqref{eq:gls}
and $\fgg_{x,s:s^+}$ is the $(-1)$-eigenspace of $*$ in $\fgl_{x,s:s^+}$. Let
$l_t := \dim_{\fff_D} \sfL_t$.

Now we consider the value of $\chigse$ on each factor of \eqref{eq:sfG0}.

\begin{enumerate}[(i)]
\item Suppose $t \nequiv -t \pmod{\bZ}$ and $t \nequiv -(t+s) \pmod{\bZ}$. We
  consider the action of $\GL(\sfL_t)$.  We have
\begin{align*}
  * & \colon \HomfD(\sfL_t,\sfL_{t+s})\simrightarrow \HomfD(\sfL_{-t-s},\sfL_{-t})
      \quad \text{and}\\
  * & \colon \HomfD(\sfL_{-t},\sfL_{-t+s})\simrightarrow
      \HomfD(\sfL_{t-s},\sfL_{t}).
\end{align*}
The two domains and codomains are distinct terms in \eqref{eq:gls}. 
Moreover $l_{t+s}  = l_{t-s}$ since $\sfL_{t-s} \cong 
\sfL_{t+s}$ via multiplication by $\varpi_D^{2s}$.
Therefore
\[
\det_{\fffF}(\Ad(g)|_{\fgg_{x,s:s^+}}) = \Norm_{\fff_D/\fffF} \left(\det_{\fff_D}(g|_{\sfL_t})^{-l_{t+s}}
\det_{\fff_D}(g|_{\sfL_t})^{l_{t-s}} \right) = 1 \quad  \forall g \in \GL(\sfL_t).
\] 
Hence $\chigse(g) = 1$ for $g \in \GL(\sfL_t)$. 

\item 
Suppose $t \equiv -t$, i.e. $t \equiv 0$ or $\half \pmod{\bZ}$. 
We consider the actions of $\rU(\sfL_0)$ and $\rU(\sfL_{\half})$.
Now $-s+\frac{1}{2}$ is an integer and multiplication by $\varpi_D^{-s+\frac{1}{2}}$ induces isomorphisms 
\[
\Hom(\sfL_0,\sfL_{s}) \simrightarrow \Hom(\sfL_0,\sfL_\half)
\text{ \ and \ } 
\Hom(\sfL_{-\half},\sfL_{-\half+s}) \simrightarrow
\Hom(\sfL_{-\half},\sfL_{0}). 
\]
Now $*\colon \Hom(\sfL_0,\sfL_{\half}) \cong
\Hom(\sfL_{-\half},\sfL_0)$. Combining with the above gives \linebreak
$\Hom(\sfL_0,\sfL_{s}) \simrightarrow \Hom(\sfL_{-\half},\sfL_{-\half+s})$.  As
a $\rU(\sfL_0)$-module, $\fgg_{x,s:s^+}$ is isomorphic to
$\Hom(\sfL_{0},\sfL_{\half})$ direct sum with certain copies of the trivial
representation.  Suppose $g_0\in \rU(\sfL_0)$. Then
$\det_{\fffF}(g_0|_{\fgg_{x,s:s^+}})= \det_{\fffF}(g_0|_{\sfL_0})^{-l_{\half}}$.
\begin{itemize}
\item If $D/F$ is unramified,  then $\rU(\sfL_0)$ is a unitary group. Now 
$\det_{\fffD}(g_0|_{\sfL_0})\in \fffD$ has norm $1$ so $\det_{\fffF}(g_0|_{\sfL_0}) = 1$. 
By a similar argument, $\det_{\fffF}(g_{\half}|_{\fgg_{x,s:s^+}}) = 1$ for
$g_{\half}\in \rU(\sfL_\half)$.

\item If $D/F$ is ramified, then $\sfL_0$ is an orthogonal space and
  $\sfL_{\half}$ is a symplectic space. Hence
  $\det_{\fffF}(g_0|_{\fgg_{x,s:s^+}})  = \det_{\fffF}(g_0|_{\sfL_0})^{-l_{\half}}
  = 1$ since $l_{\half}$ is even. Since $\sfL_{\half}$ is a symplectic space, $\det_{\fffF}(g_{\half}|_{\fgg_{x,s:s^+}}) = 1$ for
$g_{\half}\in \rU(\sfL_\half)$.
\end{itemize}
Hence, we have shown that $\chigse$ is trivial on ${\rU(\sfL_0)\times \rU(\sfL_{\half})}$. 

\item 
Suppose $t \equiv -t -s \pmod{\bZ}$. Then $t \equiv \pm \quarter \pmod{\bZ}$.
We consider the actions of $\GL(\sfL_{\frac{1}{4}})$ in \Cref{eq:sfG0}.
Composing with multiplication by $\varpi_D^{-s-\half}$ induces an isomorphism
\[
\Hom(\sfL_{\quarter},\sfL_{\quarter+s}) \simrightarrow
\Hom(\sfL_{\quarter}, \sfL_{-\quarter})
\]
and the $*$-action on the left hand side commutes with the
$(\varepsilon *)$-action on the right hand side.  Here $\varepsilon = 1$ if
$D/F$ is unramified and $\varepsilon = (-1)^{s+\half}$ if $D/F$ is
ramified.

Similarly composing with multiplication by $\varpi_D^{-s+\half}$ induces an isomorphism
\[
\Hom(\sfL_{-\quarter},\sfL_{-\quarter+s}) \simrightarrow \Hom(\sfL_{-\quarter},\sfL_{\quarter})
\]
and $*$ action on the left hand side commutes with the $(\varepsilon' *)$-action
on the right hand side with $\varepsilon' = 1$ if $D/F$ is unramified and
$\varepsilon' = (-1)^{s-\half}$ if $D/F$ is ramified.

Let $\fss:= \Hom(\sfL_{-\quarter},\sfL_{\quarter})$ and
$\fss' := \Hom(\sfL_{\quarter},\sfL_{-\quarter})$.  Clearly $\fss$ and $\fss'$
are dual to each other as $\GL(\sfL_{\quarter})$-modules over~$\fff_D$ via the
trace form $(X,Y) \mapsto \tr_{\fff_D}(YX)$.  Since the form is $*$-invariant,
$\fss^{*,e}$ and $\fss'^{*,e}$ are dual to each other for $e \in \set{\pm 1}$.
As a $\GL(\sfL_{\quarter})$-module, $\fgg_{x,s:s^+}$ is isomorphic to
$\fss^{*,-\varepsilon}\oplus \fss'^{*,-\varepsilon'}$ direct sum with
copies of the trivial representation.  Let
$g_{\quarter}\in \GL(\sfL_{\quarter})$.

\begin{itemize}
\item
If $D/F$ is unramified, then $\varepsilon = \varepsilon' = 1$. 
Since $\fss^{*,-1}$ and $\fss'^{*,-1}$ are dual to each other, we have
  $\det_{\fff_D}(g_{\quarter}|_{\fgg_{x,s:s^+}}) = 1$.

\item
If $D/F$ is ramified, then $\varepsilon = -\varepsilon'$. 
Since  $\det_{\fffF} (g_{\quarter}|_{\fss}) = \det_{\fffF}(g_{\quarter}|_{\sfL_{\quarter}})^{2
  l_{\quarter}}$ is a square 
and  $\fss = \fss^{*,\varepsilon}\oplus \fss^{*,-\varepsilon}$, we have $\chi_{\fffF}^{\fss^{*,\varepsilon}} = \chi_{\fffF}^{\fss^{*,-\varepsilon}}$.
Hence 
\[
\chi_{\fffF}^{\fgg_{x,s:s^+}}(g_{\quarter}) = \chi_{\fffF}^{\fss^{*,-\varepsilon}}(g_{\quarter})
\chi_{\fffF}^{\fss'^{*,-\varepsilon'}}(g_{\quarter}) = \chi_{\fffF}^{\fss^{*,\varepsilon}}(g_{\quarter})
\chi_{\fffF}^{\fss'^{*,\varepsilon}}(g_{\quarter}) = 1.
\]
\end{itemize}
We conclude that $\chigse = 1$ on $\GL(\sfL_{\quarter})$.
\end{enumerate}
Combining (i), (ii) and (iii), we have proved the Claim in view of \cref{eq:sfG0}.\qedhere
\end{proof}

This concludes the proof of \Cref{lem:chib}. \qed

\section{One positive depth block Case II: 
the constructions of refined minimal $K$-types}
\label{sec:OB2}
We retain the notation in \Cref{sec:OB1}.
Recall that $\Sigma = (x, \Gamma, \phi, \rho)$ is a single block datum with positive
depth $r=2s$ as in \Cref{sec:assumptions}. 
We have $\Gamma = M(w)$, $\Gamma' = M'(w)$ and a group isomorphism $\alpha\colon
G_x^0 \iso G'^0_{x'}$. In 
\Cref{def:LD.pos}, we had defined
$\Sigma' := \dthetap(\Sigma) = (x',-\Gamma', \phi',\rho')$ where $\phi':= \phi^*
\circ \alpha^{-1}$ and 
$\rho' = \rho^*\circ \alpha^{-1}$.

\subsection{A key proposition} \label{sec:oneblocksetting}

We always use `` $\breve{\ }$ '' to mark an object in $G\times G'$ having two
similar copies in components of $G$ and $G'$ as below.  We set
\begin{itemize}
\item  $\dbK :=K\times K'$, $\dbKzp = K_{0^+}\times K'_{0^+}$ and 
$\dbKp := K_+\times K'_+$;

\item $\dbsfGz := \sfG^0_x \times \sfG'^0_{x'} = G_x^0/G^0_{x,0^+} \times  G'^0_{x'}/G'^0_{x',0^+}$;

\item $\dbrho:= \rho \boxtimes \rho'$ be the 
$K\times K'$-module inflated from the $\dbsfGz$-module $\rho\boxtimes
\rho'$;

\item $\dbkappa := \kappa\boxtimes \kappa'$ be the Heisenberg-Weil representation
of $\dbK$ constructed by
$(\Gamma,-\Gamma')$ (cf. \Cref{sec:kappa});

\item  $\dbeta := \eta \boxtimes \eta' = \dbrho \otimes \dbkappa$ and

\item $\Omega := \dbK \cdot w + \sB_0 \subseteq W$. 
\end{itemize}

\begin{remark}
  The above notations also apply to multiple block $\Sigma$ and its lift
  $\Sigma'$ defined in \Cref{rmk:dt.genD} of \Cref{def:LD}.
\end{remark}

\subsubsection{}
If $J$ is a compact group, $U$ is a $J$-module and $\chi$ is an irreducible
$J$-module, then we let $U[\chi]$ denote the $\chi$ isotypic component of $U$.
Now we can state a key proposition.
\begin{prop} \label{prop:sSB}
Under the settings in \Cref{sec:oneblocksetting}, we have $\sSB_{\Omega}
[\eta\boxtimes \eta'] \cong \eta\boxtimes \eta'$.
\end{prop}
The proof will be given in \Cref{proof:sSB} based on \Cref{lem:IT1} below.

\subsubsection{} 
We now record an elementary fact which will be used freely in this paper. 
Suppose $H$ and $J$ are compact groups and $J$ is a subgroup of $H$.  For a
$J$-module $\tau$, we will identify the induced representation with a space of
functions:
\[
\Ind^H_J \tau = \set{f \colon H\rightarrow \tau | f(jh) = \tau(j)f(h) \; \forall j\in  J, h\in
  H}
\] 
where $H$ acts by right translation.

\begin{lemma}\label{lem:Ind1}
  Suppose $J$ is a compact normal subgroup of $H$. Let $J_1$ be a subgroup
  of $H$ such that $J<J_1<H$.  Let $\tau$ be a $J_1$-module and $\chi$ be an
  irreducible $J$-module.  Suppose that $H$ stabilizes $\chi$, i.e.
  $\chi\circ \Ad_h\cong \chi$ as $J$-modules for all $h \in H$. Then
  $\tau[\chi]$ is a $J_1$-module and
  $(\Ind_{J_1}^H \tau)[\chi] = \Ind_{J_1}^H(\tau[\chi])$. \qed
\end{lemma}

\def\btrianglez{\overline{\triangle^0}}

\def\psiS{\psi^{S}}

\subsubsection{} 
Let $\dbpsi := \psi_\Gamma\boxtimes \psi_{-\Gamma'}$ be the function on
$G_{x,0^+}\times G'_{x,0^+}$. Recall that $\dbpsi$ restricted on $\dbKp$ is a
character, $\dbK$ normalizes $\dbKp$ and stabilizes $\dbpsi|_{\dbKp}$. We
could extend $\dbpsi$ to a function on $\triangle^0 G_{x,0^+}G'_{x,0^+}$ by
letting $\dbpsi(xg) := \dbpsi(g)$ for all $x \in \triangle^0$ and
$g \in G_{x,0^+} G'_{x,0^+}$.

Combining  \Cref{lem:Saction},  \Cref{lem:SS}~\cref{it:lem.SS.5}
and \Cref{lem:psi} yields the following lemma.

\begin{lemma} \label{lem:SSaction}
As an $S$-module realized on $\bSb$, $\bomegaww$ is given by 
\[
\bomegaww(h) = \dbpsi(h)\, \bomegab(h)\, \bomegab(h^{-1}\cdot w-w) 
\quad \forall h \in S.
\qquad \qed
\]
\end{lemma}

We recall the definitions of $\bfbbpp$ and $\bfbbz := \bfbbpp^\perp/\bfbbpp$ in \Cref{sec:lem.S1}.   
Let $\bSb^{\bfbbpp}$ be the $\bomega_{\bfbb}(\bfbbpp)$ invariant subspace in $\bSb$.  Then
$\bSb^{\bfbbpp}\cong \chi^{\bfbbpp}\otimes \bomega_{\bfbbz}$ as 
$\bfP(\bfbbpp) \ltimes \bfH(\bfbbpp^\perp)$-modules where $\bfP(\bfbbpp)$ is the parabolic subgroup in $\bfSp(\bfbb)$ stabilizing $\bfbbpp$ (see \Cref{sec:HW} for notation).

From \Cref{lem:SS}~\cref{it:lem.SS.5},
$S = \left(\prod_{j\in \IGamma} \triangle^j \right) \dbG_s$ and
\[
S\dbK_+ = \left(\prod_{j\in \IGamma} \triangle^j \right) \dbK_+  \dbG_s = 
\begin{cases}
\triangle^0\triangle^1 \cdots \triangle^{d-1} \dbK_+ \dbG_s & \text{in Case I,}\\
 \triangle^0\triangle^1 \cdots \triangle^d \dbK_+ \dbG_s & \text{in Case II.}
\end{cases}
\]

\def\IndGz{\Ind_{\btrianglez}^{\dbsfGz}
    \bfone}
\begin{lemma} \label{lem:IT1}
\begin{enumerate}[(i)]
\item \label{it:IT1.1}
The evaluation map $\eva$ defined by $f\mapsto f(1)$ gives an isomorphism between the vector spaces
 \begin{equation} \label{eq:indSSKp}
 \xymatrix{
 \eva \colon \left(\Ind_{S}^{S\dbKp} \bomegaww \right) [\dbpsi|_\dbKp]
 \ar[r]^<>(.5){\sim} & \bSb^{\bfbbpp}.}
 \end{equation}
 Note that $\bSb^{\bfbbpp}\cong \bSbz$. The
 $S\dbK_+$-module structure on the left hand side of \eqref{eq:indSSKp}
 translates to an action $\psiS$ of $S\dbK_+$ on $\bSbz$. The action $\psiS$ is given as
 following:
\begin{enumerate}[(a)]
\item 
\label{it:IT1.a}
If $h \in \prod_{1\leq j\in \IGamma}\triangle^j \dbK_+$, then $\psiS(h)$ acts on
$\bSbz$ by the scalar $\dbpsi(h)$.  \footnote{By definition
  $\dbpsi|_{\triangle^0} \equiv 1$. Therefore by \cref{it:IT1.a} the function
  $\dbpsi$ is a character when restricted on
  $\prod_{j\in \IGamma}\triangle^j \dbK_+$.}

\item \label{it:IT1.b} If $h = (\exp(X),\exp(X')) \in G_{x,s}\times G'_{x,s}$,
  then $\psiS(g) = \dbpsi(h) \bomegabz(b)$ where $b$ is the image in $\bfbbz$ of
  $h^{-1}\cdot w - w \equiv wX - X'w \pmod{\sB_{0^+}}$.
\item \label{it:IT1.c} $\psiS|_{\triangle^0}$ is the inflation of $\bomegabz$ via $\triangle^0 \longrightarrow \bfSp(\bfbbz)$. 
\end{enumerate}
\item \label{it:IT1.2} We have $ \left( \Ind_S^{\triangle^0\dbKzp} \bomegaww \right)[\dbpsi|_{\dbKp}] \cong \dbkappa|_{\triangle^0\dbKzp}$ as $\triangle^0\dbKzp$-modules. 

\item \label{it:IT1.3} Let $\btrianglez$ denote the image of $\triangle^0$ in $\dbsfGz$. Then we
  have following $\dbK$-module isomorphisms
  \[
  \left( \Ind_S^{\dbK} \bomegaww \right)[\dbpsi|_{\dbKp}] \cong
  \Ind_{\triangle^0 \dbKzp}^{\dbK}
  \left(\left(\Ind_{S}^{\triangle^0\dbKzp}\bomegaww
    \right)[\dbpsi|_{\dbKp}]\right) \cong \left( \IndGz \right) \otimes \dbkappa.
  \] 
\end{enumerate}
\end{lemma}

Note that \cref{it:IT1.3} follows immediately from \cref{it:IT1.2}. 
Before we embark on the proofs of \cref{it:IT1.1} and \cref{it:IT1.2}, we will use the lemma to give a proof of \Cref{prop:sSB}.

\subsubsection{Proof of \mbox{\Cref{prop:sSB}}} \label{proof:sSB}
We have 
\[
\begin{split}
  \Hom_{\dbK}(\dbeta, \sSB_{\Omega}) &= \Hom_{\dbK}(\dbeta, \left(\Ind_S^{\dbK} \bomegaww \right)[\dbpsi|_{\dbKp}]) \\
  & \qquad \text{(by \Cref{lem:Saction} and the fact that $\dbeta|_{\dbK_+}$ is $\dbpsi|_{\dbK_+}$-isotypic)} \\
  & =	\Hom_{\dbK}(\dbrho\otimes \dbkappa, (\IndGz) \otimes \dbkappa) \qquad \text{(by \Cref{lem:IT1}~\cref{it:IT1.3})}\\
  & = \Hom_{\dbsfGz}(\dbrho, (\IndGz) \otimes \Hom_{\dbK_{0^+}}(\dbkappa, \dbkappa)) \\
  & \qquad \text{(since $\dbrho$ and $\IndGz$ are trivial when
    restricted on $\dbK_{0^+}$)}\\
  & = \Hom_{\dbsfGz}(\dbrho, \Ind_{\btrianglez}^{\dbsfGz}\bfone ) \qquad
  \text{(since $\dbkappa|_{\dbK_{0^+}}$ is irreducible)}\\
  & =\Hom_{\btrianglez}(\dbrho, \bfone) = \bC \qquad
  \text{(since $\rho' =\rho^*\circ \alpha^{-1}$ )}.
\end{split}
\]

This proves the proposition. \qed

\subsection{Proof of \Cref{lem:IT1}}\label{sec:IT1.proof}
The rest of this section is devoted to proving \Cref{lem:IT1}.

\subsubsection{Proof of \Cref{lem:IT1}~\cref{it:IT1.1}}
We recall that $S_+ := S \cap \dbKp$. Frobenius reciprocity gives the following
natural isomorphism of vector spaces:
\[  
\xymatrix{
\eva \colon \left(\Ind_{S}^{S\dbKp} \bomegaww \right) [\dbpsi|_\dbKp] \cong
\left(\Ind_{S_+}^{\dbKp} \bomegaww \right) [\dbpsi|_\dbKp] \ar[r]^<>(.5){\sim} & \bomegaww[\dbpsi|_{S_+}].
}
\]

Now the key is to prove the following claim.

\begin{claim*}
We have
\begin{equation}\label{eq:Sbbp}
\bomegaww[\dbpsi|_{S_+}] = \bSb^{\bfbbpp} \subseteq \bSb.
\end{equation}
\end{claim*}

\begin{proof} 
  We will only prove it for Case~I. The proof for Case~II is similar and easier,
  so we leave it to the reader.

  We recall \Cref{lem:SS}~\cref{it:lem.SS.5} that
  \begin{equation} \label{eq:Splus} S_+= \triangle^0_+ \triangle^1_+ \cdots
    \triangle^{d-1}_+ G_{x,s}^{d-1} G'^{d-1}_{x',s} G_{x,s^+}G'_{x',s^+}.
  \end{equation}
  Now we consider the $\bomegaww$-action (cf. \Cref{lem:SSaction}) of each
  factor on the right hand side of~\eqref{eq:Splus}. Note that $\bomegab|_{S_+}$
  is trivial since $S_+ \subseteq G_{x,0^+}G'_{x',0^+}$.
\begin{enumerate}[(1)]
\item \label{it:pf.S1} 
Suppose $h = \exp(X) \in G_{x,s^+}$ where
  $X\in \fgg_{x,s^+}$.  Then
  $h^{-1} \cdot w -w \in wX +\sB_{s^+} \subseteq \sB_{0^+}$.  Hence
  $\bomegaww(h) = \dbpsi(h)\, \bomegab(h^{-1} \cdot w-w) = \dbpsi(h)$ . By the same
  argument, we also have $\bomegaww(h) = \dbpsi(h)$ for $h\in G'_{x,s^+}$.
\item 
Suppose $h = \exp(X)\in G^{d-1}_{x,s}$ with $X \in \fgg^{d-1}_{x,s}$.
    Then $h^{-1} \cdot w-w \in wX + \sB_{0^+} \subseteq \sB_0$. Hence 
    $\bomegaww(h) = \dbpsi(h)\bomegab(wX)$. 
    The same argument gives $\bomegaww(h') = \dbpsi(h') \bomegab(-X'w)$ for $h' =
    \exp(X')\in G'^{d-1}_{x,s}$ where $X'\in \fgg'^{d-1}_{x,s}$. 
    Hence $\bomegaww[\dbpsi|_{S^+}] \subseteq \bSb^{\bfbb_+}$ since $\bfbb_+ =
    \biota(\fgg^{d-1}_{x,s}\oplus \fgg'^{d-1}_{x',s}) =
    w\fgg^{d-1}_{x,s}+\fgg'^{d-1}_{x',s}w  +\sB_{0^+} 
    \subseteq \sB_0/\sB_{0^+}$. 

 \item \label{it:pf.S3}
 Suppose $h = (g,g')\in \triangle^i_+$ for $0\leq i \leq d-1$. By \Cref{lem:SS}~\cref{it:lem.SS.1} and \cref{it:lem.SS.2}, $h^{-1}w -w \in \sB_{0^+}$.  Therefore
    $\bomegaww(h) = \dbpsi(h)$.
    \savemyenumi
  \end{enumerate}
  Combining \cref{it:pf.S1}--\cref{it:pf.S3}, we see that the $\dbpsi|_{S_+}$ isotypic component is exactly
  the $\bomega_{\bfbb}(\bfbbpp)$-invariant subspace in $\bS(\bfbb)$. This proves
  the claim.
\end{proof}

Now we calculate the translated $S \dbK_+ =
\triangle^0 \prod_{0<i\in \IGamma}\triangle^{i}\dbK_+$ action $\psiS$ on  $\bSbz \cong \bSb^{\bfbbpp}$.
\begin{enumerate}[(1)]
\resumemyenumi
\item \label{it:pf.S4} Clearly $\psiS|_\dbKp = \dbpsi|_{\dbKp}$. 

\item Suppose $h \in \triangle^i$ for $0 < i \in \IGamma$.  Then
    $h\in  G_{x,0^+}G'_{x',0^+}$ and $h^{-1} \cdot w - w \in \sB_{0^+}$ by
    \Cref{lem:SS}~\cref{it:lem.SS.2}.  So
    $\psiS(h) = \bomegaww(h) = \dbpsi(h)$.  Combining this with \cref{it:pf.S4}
    proves \Cref{it:IT1.a}.

\item Suppose $h\in G_{x,s}\times G'_{x',s}$. By \Cref{sec:HW}~\cref{it:fWeil.3}, 
\[
\psiS(h) = \bomegaww(h) = \dbpsi(h)\bomega_{\bfbb}(h^{-1}\cdot w -w) =
\dbpsi(h)\bomega_{\bfbbz}(b).
\] 
This proves \Cref{it:IT1.b}.

\item Suppose $h \in \triangle^0$. By \Cref{sec:HW}~\cref{it:fWeil.3} and \Cref{lem:chib}, 
  \[
  \psiS(h) = \bomegaww(h) = \bomega_{\bfbb}(h) = \chi^{\bfbbpp}(h)
  \bomega_{\bfbb_0}(h) = \bomega_{\bfbb_0}(h).
  \] 
  This proves \Cref{it:IT1.c}. 
\end{enumerate}
These complete the proof of \Cref{lem:IT1}~\cref{it:IT1.1}.

\smallskip

\subsubsection{Proof of \Cref{lem:IT1}~\cref{it:IT1.2}}
By Part \cref{it:IT1.1} and \Cref{lem:Ind1}, we have 
\begin{equation} \label{eq:II2}
  \left(\Ind_S^{\triangle^0\dbKzp} \bomegaww \right)
  [\dbpsi|_{\dbKp}]  \cong
  \Ind_{S\dbKp}^{\triangle^0\dbKzp} \left( \left(\Ind_{S}^{S\dbKp} \bomegaww \right) 
[\dbpsi|_{\dbKp}] \right) 
 \cong \Ind_{S\dbKp}^{\triangle^0\dbKzp} \psiS
\end{equation}
as $\triangle^0\dbKzp$-modules. 

\def\II{\Ind_{S\dbKp}^{\triangle^0\dbKzp} \psiS}

\begin{claim*} 
We have
\begin{equation} \label{eq:IIdim}
\dim \II = \dim \dbkappa.
\end{equation} 
\end{claim*}

\begin{proof}
Let $K_{+,s} := K_{+}\cap G_{x,s}$ and $K_s := K \cap G_{x,s} = G_{x,s}$. 
Let $Q := \# (S_{0^+}/S_{+})$, $N := \# (K_{0^+}/K_+)$, $N_s := \#(K_s/K_{+,s})$, 
 $N' := \# (K'_{0^+}/K'_+)$ and $N'_s := \#(K'_s/K'_{+,s})$. 

\smallskip

We note the following facts. 
\def\SSzp{S_{0^+}} 
\def\SSp{S_{+}}
\begin{enumerate}[1.]
\item 
By the definition of $\dbkappa$, we have $\dim\dbkappa = (\#
\dbKzp/\dbKp)^\half = \sqrt{NN'}$.

\item 
  By \Cref{lem:S1}~\cref{it:lem.S1.6}, we have $\dim \bSbz = \sqrt{\#\bfbbz} = 
\sqrt{\#(K_s/K_{+,s}) \#(K'_s/K'_{+,s})} = \sqrt{N_sN'_s}$. 

\item 
By \Cref{lem:SS}~\cref{it:lem.SS.5}, the projection to the first coordinate
$S_{0^+} \rightarrow K_{0^+}$ is surjective and its kernel is
$K'_{s}$. Hence $K_{0^+} \cong S_{0^+}/K'_{s}$. Similarly, we have $K_{+}\cong
S_{+}/K'_{+,s}$. Hence
\begin{equation}\label{eq:count1}
K_{0^+}/K_+  \cong   (S_{0^+}/K'_{s}) / (S_{+}/K'_{+,s}) \cong S_{0^+}/K'_{s} S_+ \cong (S_{0^+}/S_+)/(K'_{s} S_+/S_+).
\end{equation}
Note that $K'_s/K'_{+,s} = K'_sS_+/S_+$. Counting the elements of the both sides
of \eqref{eq:count1}, we get $N = Q/N'_s$. A similar argument yields $N'= Q/N_s$. 

\item 
Note that $\dbKzp\cap S\dbKp = \SSzp \dbKp$ and 
$\SSzp\dbKp/\dbKp = \SSzp/(\dbKp\cap \SSzp) = \SSzp/\SSp$. 
\end{enumerate}
Hence 
\[
\begin{split}
\dim \II = & \dim \bSbz \cdot \#(\dbKzp/\SSzp\dbKp) \\
 = & \dim \bSbz \cdot \#(\dbKzp/\dbKp)  /
\#(\SSzp\dbKp/\dbKp) \\
 &= \sqrt{N_sN'_s} N N'/ Q = \sqrt{NN'} = \dim \dbkappa.
\end{split}
\]
This proves the claim.
\end{proof}

In \Cref{sec:SKhom}, we will show that 
\begin{equation} \label{eq:SKhom}
\dim \Hom_{\triangle^0\dbKzp}(\dbkappa,\II)=\dim 
\Hom_{S\dbKp}(\dbkappa, \psiS) =1. 
\end{equation}
Combining \eqref{eq:II2}, \eqref{eq:IIdim} and \eqref{eq:SKhom}  gives \Cref{lem:IT1}~\cref{it:IT1.2}. 
This also completes the proof of the whole \Cref{lem:IT1}. \qed

\subsubsection{Proof of \mbox{\eqref{eq:SKhom}}}\label{sec:SKhom}
The first equality in \eqref{eq:SKhom} is just Frobenius reciprocity. It remains
to prove the second equality.  We only give the proof for Case~I. The
  proof for Case~II where $\bfbb_0 = 0$ is essentially contained in the proof of
  \Cref{lem:dimPj}~\cref{it:dimPj.2} below.

\def\Qz{\triangle^0}

We assume the notation and the construction of $\kappa$ in \Cref{sec:kappa}.  We
also retain the notion in \Cref{sec:iso}. Recall
$J^j := (G^{j-1},G^j)_{x,(r_{j-1},s_{j-1})}$,
$J^j_+ := (G^{j-1},G^j)_{x,(r_{j-1},s_{j-1}^+)}$, $\dbJ^j := J^j\times J'^j$ and
$\dbJ^j_+ = J^j_+\times J'^j_+$.

\begin{itemize}
\item  For $1\leq j\leq d$, let $Q^{j} := \triangle^1\cdots \triangle^j \dbK^j_+$ and  $Q_j := (\triangle^j\cap
\dbJ^j)\dbJ^j_+$. 
\item For $0 \leq j \leq d-1$, let $\dbkappa^j := \kappa^j\boxtimes \kappa'^j$
  and
  $P^{j} := S\dbK_+\cap G^{j}= \triangle^0 \triangle^1 \cdots \triangle^j
  \dbK^j_+$. Moreover, let $P :=P^{d-1}$.

\item For $1\leq j \leq d$, let
  $\dbbomega^j_{\dbGamma_{j-1}}:= \bomega^j_{\Gamma_{j-1}}\boxtimes
  \bomega^j_{-\Gamma'_{j-1}}$
  and
  $\tv\times \dbpsi_{\dbGamma^j}: = (\tv\times \psi_{\Gamma^j})\boxtimes
  (\tv\times \psi_{-\Gamma^j})$
  be
  $\dbK^{j-1}\ltimes \dbJ^j = (K^{j-1}\ltimes J^j)\times (K'^{j-1}\ltimes J'^j)
  $-modules. Here $\Gamma^j := \sum_{i=j}^d \Gamma_i$ and
    $\Gamma'^j := \sum_{i=j}^d \Gamma'_i$.
\end{itemize}

\def\dbbomegaG#1{\dbbomega^{#1}_{\dbGamma_{#1-1}}}
By definition we have
\begin{enumerate}[(a)]
\item a surjection $P \ltimes \dbJ^d \twoheadrightarrow S\dbK_+$;
\item 
$Q^j$ and $Q_j$ are groups such that $Q^{j} = Q^{j-1} Q_{j}$;
\item $P^0 = \Qz K^{0}_+$ and  $P^{j} =  \Qz Q^{j} = P^{j-1}Q_j$;

\item Pulling back via $P^{j-1}\ltimes \dbJ^{j} \rightarrow \dbK^j$, $\dbkappa^j|_{P^{j-1}\ltimes \dbJ^{j}}  = \dbkappa^{j-1} \otimes
(\dbbomegaG{j}\otimes (\tv\times \dbpsi_{\dbGamma^j}))$
where the $\dbK^{j-1}$-module $\dbkappa^{j-1}$ is inflated to $P^{j-1}\ltimes \dbJ^j$
via $P^{j-1}\ltimes \dbJ^j \twoheadrightarrow P^{j-1} \hookrightarrow \dbK^{j-1}$
(cf. \Cref{sec:kappa}).
In particular, as $P\ltimes J^d$-module, 
\[
\dbkappa = \dbkappa^{d-1}\otimes \dbbomegaG{d}.
\]
\end{enumerate}


Now
\begin{equation}\label{eq:SK}
\begin{split}
\Hom_{S\dbK_+}(\dbkappa, \psiS) &= \Hom_{P\ltimes \dbJ^d}(\dbkappa^{d-1},
\Hom_{\bC}(\dbbomegaG{d}, \psiS))  \\
& = \Hom_{P}(\dbkappa^{d-1},\Hom_{\dbJ^d}(\dbbomegaG{d}, \psiS)). 
\end{split}
\end{equation}

\def\dbbomegaGa{\bomega^d_{\Gamma_{d-1}}\otimes \bomega^d_{-\Gamma'_{d-1}}}
\def\dbbomegaGa{\bomega^d_{\Gamma_{d-1}}\otimes
\bomega^d_{-\Gamma'_{d-1}}}

The map $\biota$ defined in \Cref{sec:lem.S1} induces a
$P$-equivariant isomorphism of symplectic spaces: 
\begin{equation}\label{eq:bfbbz.iso}
\xymatrix@R=0em@C=3em{
\biotabz\colon 
\dbbfW^{d} := \fgg^{d-1\perp}_{x,s:s^+}\oplus
\fgg'^{d-1\perp}_{x',s:s^+}  \ar[r]^<>(.5){\biota}& \bfbbz.
}
\end{equation}
In fact, this is just a rephrase of \Cref{lem:S1}~\cref{it:lem.S1.6} and \Cref{lem:SS}~\cref{it:lem.SS.1}
since the $P$ actions on the both
sides of $\biotabz$ factors through $\triangle^0$.

\def\SH{\zeta}
\def\SHomega{\SH^{\bfbbz}}
Consider the map\footnote{Note that there is a negative sign before $(X,X')\cdot w$.}
\begin{equation}\label{eq:SH.b0}
  \vcenter{
    \xymatrix@C=0em@R=0em{
      \SHomega \colon& \dbJ^d \ar[rr]&\hspace{2em}&
      (\bfbbpp^\perp/\bfbbpp)\times \fff  = \bfH(\bfbb_0)\\
      & h=(\exp(X),\exp(X')) \ar@{|->}[rr]&& 
      (- (X,X')\cdot w,\overline{\half\innw{w}{h^{-1}\cdot w -w}}).
    }
  }
\end{equation}
By the explicit description of
special isomorphism in \Cref{sec:Special}, \cref{eq:Sym.s}, \cref{eq:bfbbz.iso} and 
\Cref{lem:IT1}~\cref{it:IT1.1}~\Cref{it:IT1.b},
the following diagram commutes:
\[
\xymatrix@R=-.3em{
  &\bfSH(\dbbfW^{d})\ar[dd]^{\cong}\\
\triangle^0 \ltimes \dbJ^d\ar[ru]^{
}\ar[rd]_{\SHomega}&\\
& \bfSH(\bfbbz)\\
}
\]

Hence, as $\triangle^0\ltimes \dbJ^d$-module,
\begin{equation} \label{eq:bomega.d}
\psiS\cong \dbbomegaG{d}.
\end{equation}

On the other hand, $Q^{d-1}\subseteq \dbK^{d-1}_{0^+}$ acts trivially on $\dbbomegaG{d}$
and acts as $\dbpsi$ on $\psiS$ (cf. \Cref{lem:IT1}~\cref{it:IT1.1}~\Cref{it:IT1.a}). 
Therefore, as $P = \triangle^0Q^{d-1}$-modules, 
$\Hom_{\dbJ^d}(\dbbomegaG{d}, \psiS)\cong \dbpsi|_P$.
Putting this into \eqref{eq:SK}, \eqref{eq:SKhom} becomes
\begin{equation} \label{eqkdminus}
\dim \Hom_{P}(\dbkappa^{d-1}, \dbpsi|_P) =1. 
\end{equation}
This follows from \Cref{it:dimPj.2} of the next lemma.


\begin{lemma} \label{lem:dimPj}
\begin{enumerate}[(i)]
\item\label{it:dimPj.1} For  $1 \leq j \leq d-1$, we have
\begin{equation} \label{eq:PJ}
 \Hom_{Q_j}(\dbbomegaG{j} \otimes (\tv \times \dbpsi_{\dbGamma^j}),\dbpsi) \cong \dbpsi
\end{equation}
as $P^{j-1}$-modules.

\item \label{it:dimPj.2} For $0 \leq j \leq d-1$, we have $\dim \Hom_{P^{j}}(\dbkappa^{j}, \dbpsi)=1$ .
\end{enumerate}
\end{lemma}

\begin{proof}
\begin{asparaenum}[(i)]
\item We first check that the right hand side of \eqref{eq:PJ} has dimension
  one.  The image of $Q_j$ under
  $\dbJ^j \twoheadrightarrow \dbJ^j/\dbJ^j_+ =: \dbbfW^j$ is $\bfD^j$ which is a
  maximal isotropic subspace in $\dbbfW^j$ according to \Cref{lem:Di}.  Note
  that the $Q_j$-character\footnote{One can see that $\dbpsi_{\dbGamma_{j-1}}$
    is a character of $Q_j$ directly from the discussion in \Cref{sec:Special}.}
  $\dbpsi_{\dbGamma_{j-1}} := \psi_{\Gamma_{j-1}}\boxtimes \psi_{-\Gamma'_{j-1}}
  \cong \dbpsi|_{J^j} \otimes (\dbpsi_{\dbGamma^j}|_{J^j})^{-1}$
  factors to a $\bfD^j$-character.  Therefore
\[
\dim \Hom_{Q_j}(\dbbomega_{\dbGamma_{j-1}} \otimes (\tv \times
\dbpsi_{\dbGamma^j}),\dbpsi) = \dim \Hom_{Q_j}(\dbbomega_{\dbGamma_{j-1}},
\dbpsi_{\dbGamma_{j-1}}) =1.
\]

We now check that the actions of $P^{j-1} = \triangle^0 Q^{j-1}$ on both sides
of~\eqref{eq:PJ} agree.
\begin{itemize}
\item The character $\dbpsi$ is trivial on $\triangle^0$. By the \Cref{sec:HW},
  the left hand side of \eqref{eq:PJ} is isomorphism to $ \chi^{\bfD^j}$ as
  $\Qz$-module.  We claim that $\chi^{\bfD^j}|_{\triangle^0}$ is trivial. Indeed
  $\bfD^j \cong \bfW^j$ as $\triangle^0 \cong G^0_x$-module. Since the right
  hand side has a symplectic form preserved by $G^0_x$-action,
  $\det_\fffF(h|_{\bfD^j})= 1$ for all $h\in \triangle^0$. This proves the claim.

\item The group $Q^{j-1}\subseteq P^{j-1}\cap \dbK^{j-1}_{0^+}$ has trivial
  action on
  $\dbbomega_{\dbGamma_{j-1}} \otimes (\tv \times \dbpsi_{\dbGamma^j})$.
  Therefore the left hand side of \eqref{eq:PJ} is $\dbpsi$-isotypic as
  $Q^{j-1}$-module.
\end{itemize}
This proves (i).

\item  We prove by induction on $j$.
\begin{enumerate}[1.]
\item 
By definition, $\dim \Hom_{P^0}(\dbkappa^0, \dbpsi)=1$.
 
\item Now assume $\dim \Hom_{P^{j-1}}(\dbkappa^{j-1}, \dbpsi)=1$.
By (i)
\[
\Hom_{P^j}(\dbkappa^j, \dbpsi) = 
\Hom_{P^{j-1} }(\dbkappa^{j-1}, 
\Hom_{Q_j}(\dbbomega_{\dbGamma_{j-1}} \otimes (\tv \times
\dbpsi_{\dbGamma^j}),\dbpsi)) = \Hom_{P^{j-1}}(\dbkappa^{j-1}, \dbpsi).
\]
Hence $\dim \Hom_{P^j}(\dbkappa^j, \dbpsi) = 1$. This completes the induction
process and proves (ii).
\end{enumerate}
\end{asparaenum}
\end{proof}

Now \eqref{eqkdminus} holds and the proof of \eqref{eq:SKhom} is complete. \qed

\section{Proof of the main theorem I: 
  construction of $K$-types in the general case}
\label{sec:CKtype.gen}


In this section, we will prove the part~\cref{it:Main.1} of the \Cref{thm:main}
by reducing the statement into one block cases.  The idea is old, already
appeared in \cite[\Sec{2.4}]{HoweMoy} and \cite[Section~3.3]{S01i} for
example. Hence we will omit the proofs of some simple facts.

We retain the notation in the \Cref{thm:main}. The \Cref{it:Main.1} of the
\Cref{thm:main} is a consequence of the following proposition.
\begin{prop} \label{prop:CKtype} Suppose
  $\Sigmap := \dthetaVT(\Sigma) = (x',-\Gamma',\phi',\rho')$.  Let
  $w \in V\otimes_DV'$ be the element defined by \eqref{eq:w.gen} via the
  construction of $\Sigmap$. We retain the notation in
  \Cref{sec:oneblocksetting} with respect to $w$, $\Sigma$ and $\Sigma'$ so that
  $\Omega := \dbK \cdot w + \sB_0$.  Then
  \[
  \dim \Hom_{\dbK}(\dbeta, \sSB_{\Omega})=1. 
  \]
\end{prop}

\begin{remark}
  More generally, if $\Sigma'$ a theta lift of $\Sigma$ as in
    \Cref{def:LDC}, then the same proof in this section would show that
    \[
      \Hom_{\dbK}(\dbeta, \sSB_{\Omega})\neq 0.
    \]
\end{remark}
The rest of this section is devoted to proving the proposition by induction on
the number of blocks.  We first state the induction hypothesis in
\Cref{sec:indhyp}. Then we prove some lemmas on the block decomposition in
\Cref{sec:decSigma,sec:HW.blocks}. The proof is completed in
\Cref{sec:CKtype.fini}.

\def\sm#1{\dot{#1}}
\def\smG{\sm{G}}
\def\smdbK{\sm{\dbK}}
\def\smV{\sm{V}}
\def\smVp{\sm{V}'}
\def\smTp{\sm{\cT}'}
\def\smD{\sm{D}}
\def\smtD{\sm{\tD}}
\def\smbb{\sm{b}}
\def\smdbeta{\sm{\dbeta}}
\def\smOmega{\sm{\Omega}}
\def\smSigma{\sm{\Sigma}}
\def\smSigmap{\sm{\Sigma}'}
\def\smsB{\sm{\sB}}
\def\smdthetaVT{\dtheta_{\smV,\smTp}}

\subsection{Induction hypothesis} \label{sec:indhyp}
Let 
\begin{enumerate}[(a)]
\item $\smV$ be an $\epsilon$-Hermitian space such that
  $\dim_D \smV \leq \dim_D V$;

\item $\smTp$ be a Witt tower of $\epsilon'$-Hermitian spaces;

\item $\smSigma$ be a supercuspidal data for $\smG:=\rU(\smV)$ such that 
$\smSigma$ has $\smbb$ blocks and $\smbb<b$;

\item $\smSigma' := \smdthetaVT(\smSigma)$ be a supercuspidal data of $\smVp$
  where $[\smVp] \in \smTp$.
\end{enumerate}
We extend all the notations to this dual pair
by adding ``$\sm{\hspace{1em}}$''. 
We assume that \Cref{prop:CKtype} holds for $(\smSigma,\smSigmap)$, i.e.
\[
\dim \Hom_{\smdbK}(\smdbeta, \sS(\smsB_0)_{\smOmega})=1. 
\]

\medskip 

Note that the Hypothesis holds when $\smbb = 0$, i.e. the
depth-zero case (cf. \cite{Pan02J} and \Cref{sec:TL.zero}). 

\subsection{Block decomposition of vector spaces}\label{sec:decSigma}
We have already treated the depth zero case, so we assume that $b\geq 1$. 
\subsubsection{}
Let $\Sigma = \bigoplus_{l=0}^b \llSigma$ and
$\Gamma = \bigoplus_{l=0}^b \llGamma$ be the decomposition of datum $\Sigma$
according to \Cref{prop:decSigma} where $\llGamma$ has depth $-\llr$.  We denote
$\aaSigma := \bigoplus_{i=0}^{b-1}\llSigma$ so that
$\Sigma = \bbSigma \bigoplus \aaSigma$.  In the rest of the section, the index
$i$ is reserved specially for $i = b, a$.


\begin{definition}\label{def:sskappa}
We collect the following definitions and facts.
\begin{enumerate}[(i)]
\item Let $r = {}^br$ be the depth of $\Sigma$ and $s = r/2$ as usual. 
\item We have $V = \bbV \oplus \aaV$ where $\aaV := \bigoplus_{l=0}^{b-1} \llV$.

\item Let $\iiG := \rU(\iiV)$. Then
  $\bbG\times \aaG \subseteq \End_D(\bbV)\oplus \End_D(\aaV)$ sitting in
  $G\subseteq \End_D(V)$ block diagonally.

\item We have $\Gamma = \bbGamma \oplus \aaGamma$ where
  $\aaGamma := \bigoplus_{l=0}^{b-1} \llGamma$.

\item We have $x = (\bbxx,\aaxx) \in \BTB{\bbG} \times \BTB{\aaG}$ and
  $\sL = \bbsL \oplus \aasL$ gives the decomposition in terms of the
  corresponding lattice functions.

\item We have $G_x^0 = \bbG_{\bbxx}^0\times \aaG_{\aaxx}^0$,
  $\rho= \bbrho\boxtimes \aarho$ and $\phi = \bbphi\boxtimes \aaphi$.

\item Let $\iiK := K\cap \iiG$, $\iiK_{0^+} := K_{0^+} \cap \iiG$,
  $\iiK_+ := K_+ \cap \iiG$ and $\dgK := \bbK \times \aaK$.


\item Let $\ssfgg := \fgg \cap \ssEnd$ and $\ssfgg_{x,s} := \ssfgg \cap
  \fgg_{x,s}$
  where 
  \[
    \ssEnd := \Hom_D(\bbV,\aaV) \oplus \Hom_D(\aaV,\bbV) \subseteq \End_D(V).
  \]

\item Let $\ssJ := G_{x,r}\exp(\ssfgg_{x,s})$ and
  $\ssJ_+ := \ssJ\cap K_+ = G_{x,r} \exp(\ssfgg_{x,s^+})$.  Obviously
  \[
  \exp \colon  \ssfgg_{x,s:s^+} \iso \ssJ/\ssJ_+ =: \ssbfW
  \]
  is a $\bbG_{\bbxx}\times \aaG_{\aaxx}$-equivariant isomorphism between abelian
  groups.
\item We have
  \begin{equation}\label{eq:KK}
    K = \dgK\ssJ = \bbK \; \aaK\;  \ssJ \quad 
    \text{and}\quad K_+  =\dgK_+\ssJ_+ =  \bbK_+ \aaK_+ \ssJ_+. 
  \end{equation}
  \trivial[h]{ Clearly the left hand sides contain the right hand sides. We will
    prove the opposite containment.  Let $\ckGamma$ be a good element congruent
    to $\Gamma$. The $\ckGamma$ acts on $\aaV$ by zero and invertible on $\bbV$.
    By \Cref{sec:Ktype} $K = K^{d-1} G^d_{x,s}$ where
    $K^{d-1} = G^0_{x} G^1_{x,s_0} \cdots G^{d-1}_{x,s_{d-2}}$.  Now $K^{d-1}$
    commutes with $\ckGamma$ so it stabilizes the direct sum
    $V = \bbV \oplus \aaV$. Hence $K^{d-1}$ is contained in $\bbK \times
    \aaK$.
    Since $G_{x,s} = \bbG_{x,s} \aaG_{x,s} \ssJ$, we have
    $K \subseteq (\bbK \aaK) G_{x,s} = \bbK \aaK \ssJ$.  }

\item \label{it:sskappa.def} 
Let $\sskappa$ be the pull back of $\bomega_{\ssbfW}$ via
  \[
  \dgK\ltimes \ssJ \longrightarrow \bfSp(\ssbfW)\ltimes \bfH(\ssbfW)
  \]
  where $\ssJ\rightarrow \bfH(\ssbfW)$ is the restriction of special morphism
  given by \eqref{eq:Special}.

\item Let $\dgkappa$ be the Heisenberg-Weil representation of
  $\dgK= \bbK\times \aaK$ determined by $(\bbGamma, \aaGamma)$ and let
  $\dgeta := \bbeta \boxtimes \aaeta$ where $\iieta$ is the $\iiK$-module
  defined by the datum $\iiSigma$.  
\item We identify $\dgkappa$ and $\dgeta$ with their inflations to
  $\dgK\ltimes \ssJ$. Directly from the construction of $\kappa$ and $\eta$, we
  get
  \begin{equation}\label{eq:kappa.2}
    \kappa = \dgkappa \otimes \sskappa \quad \text{and} \quad \eta = \dgeta
    \otimes \sskappa.
  \end{equation}
  as $\dgK\ltimes \ssJ$-modules.
\end{enumerate}
For the data $\Sigma'$, we define similar notations by adding ``prime'' and get
corresponding conclusions.
\end{definition}

\subsubsection{} \label{sec:Vnotations} 
We now consider the block decomposition in the context of theta correspondence.  

\begin{definition}\label{def:block.theta}
  We collect the following notations and facts. 
\begin{enumerate}[(i)]
\item \label{it:blk.1} Let $\iiW := \iiV\otimesD \iiVp$,
  $\baW := \bbV\otimesD\aaVp$, $\abW := \aaV\otimesD \bbVp$,
  $\dgW := \bbW \oplus \aaW$ and $\ssW := \baW \oplus \abW$. Then we have
  following orthogonal decompositions of the symplectic spaces:
\begin{equation}\label{eq:decW}
W = V\otimesD V'  = \dgW \oplus \ssW = (\bbW \oplus \aaW) \oplus (\baW \oplus \abW).
\end{equation}

\item \label{it:blk.2} We have irreducible reductive dual pairs $(\iiG,\iiG')$
  in $\Sp(\iiW)$ for $i = b$ and $a$.  They form a reducible reductive dual pair
  $(\dgG,\dgG') := (\bbG\times \aaG,\bbG'\times \aaG')$ in $\Sp(\dgW)$.

\item \label{it:blk.3} By the construction of lifting of datum (see
  \Cref{def:LD}), $\bbSigmap = \dthetap(\bbSigma)$ and
  $\aaSigmap = \dthetaaaVT(\aaSigma)$ where $\aacTp := \cT' - [\bbV']$.  In
  addition, we have $w = \bbww \oplus \aaww \in \bbW \oplus \aaW$ so that
  $\aaww = \bigoplus_{l=0}^{b-1}\llww$, $M(\iiww) = \iiGamma$ and
  $M'(\iiww) = \iiGammap$.

\item  \label{it:blk.4}
Define lattice functions 
\[
\anysB_t := \anyW \cap \sB_t \text{ and } \anybfbb =
  \anysB_0/\anysB_{0^+} \text{ for }\any = b,a,ba,ab,\boxbackslash, \boxslash.
\]
We have $\iisB = \iisL \otimes \iisL'$, 
  $\basB = \bbsL \otimes \aasL'$,  
  $\absB = \aasL \otimes \bbsL'$,
  \[
  \begin{split}
    \sB & = \dgsB \oplus \sssB =  (\bbsB \oplus \aasB) \oplus (\basB \oplus \absB)
    \quad \text{and} \\
    \bfbb &  = \dgbfbb \oplus \ssbfbb = (\bbbfbb \oplus \aabfbb) \oplus (\babfbb
    \oplus \abbfbb).
  \end{split}
  \]  

\item 
  \label{it:blk.5}
  Define doubled objects $\dbK := K \times K'$, $\iidbK := \iiK \times \iiK'$,
  $\ssdbJ := \ssJ\times \ssJ'$, $\dgdbeta := \dgeta\boxtimes \dgeta'$,
  $\ssdbkappa := \sskappa\boxtimes \sskappa'$ etc. as usual.

\item 
   \label{it:blk.6}
   Let $\iiOmega := \iidbK\iiww + \iisB_{0}$ and
   $\iiS := \Stab_{\iidbK}(\iiww+\iisB_{0})$ for $i = a, b$. We have
   \[
   \dgS := \Stab_{\bbdbK\times \aadbK}(w+\sB_{0})=\bbS \times \aaS \text{ and }
   S := \Stab_{\dbK}(w+\sB_{0}) = \dgS \ssdbJ.
   \]

\end{enumerate}

\end{definition}

\subsubsection{}
Similar to \eqref{eq:bfbbz.iso}, we have following lemma. 
\begin{lemma}\label{lem:TB1}
  Consider the map
  $\ssbiota\colon \ssfgg_{x,s:s^+} \oplus \ssfggp_{x',s:s^+} \rightarrow
  \sssB_{0:0^+} = \ssbfbb$ induced by
  \[
  (X,X') \mapsto (X,X') \cdot w = -wX + X'w \qquad \forall X\in \ssfgg_{x,s},
  X'\in \ssfggp_{x',s}.
  \]
  Both the domain and codomin of $\ssbiota$ have natural $\dgdbK$-module
  structures and the actions factor through
  $\dgdbK/\dgdbK_{0^+}$. Moreover $\ssbiota$ is an $\dgS$-equivariant
  isomorphism between $\fffF$-vector spaces.
\end{lemma}

\begin{proof}
  Let $(X,X') \in \ssfgg_{x,s} \oplus \ssfggp_{x',s}$.  In
  terms of matrices with respect to the decomposition
  $V = \bbV \oplus \aaV$ and $V' = \bbV' \oplus \aaV'$, we write
  \[
  X = \begin{pmatrix}0 & A\\ -A^* & 0\end{pmatrix},\quad X' = \begin{pmatrix}0 &
    A'\\ -A'^* & 0\end{pmatrix} \quad \text{and} \quad w = \begin{pmatrix} \bbww
    & 0 \\ 0 & \aaww\end{pmatrix}.
  \]
  Here $A\in \Hom(\aasL,\bbsL)_s\subseteq \Hom_D(\aaV,\bbV)$ and
  $* \colon \Hom_D(\aaV, \bbV) \iso \Hom_D(\bbV,\aaV)$ is defined by
  $\inn{A v_1}{v_2}_{\bbV} = \inn{v_1}{A^* v_2}_{\aaV}$ for all
  $v_1\in \aaV, v_2 \in \bbV$. The notation for $V'$ is defined similarly. Then
  \[
  (X,X')\cdot w = -w X + X'w = \begin{pmatrix}
    0 & -\bbww A+ A'\aaww\\
    \aaww A^* - A'^* \bbww & 0
  \end{pmatrix}.
  \]

\setcounter{claim}{0}
\begin{claim}\label{claim:ssinj}
Then map  $\ssbiota$ is injective. 
\end{claim}

\begin{proof}
Suppose $\ssiota(X,X') \in \sssB_{0^+}$, i.e. 
\begin{eqnarray}
-\bbww A+ A'\aaww & \equiv & 0 \pmod{\absB_{0^+}} \label{eq:ss1} 
 \quad \text{and}  \\
\aaww A^* - A'^* \bbww & \equiv & 0 \pmod{\basB_{0^+}} \label{eq:ss2}.
\end{eqnarray}
Applying $\mstar$ to \eqref{eq:ss2}  gives
\begin{equation}\label{eq:ss3}
A \aaww^\mstar  -\bbww^\mstar A'  \equiv 0 \pmod{\absB_{0^+}}.
\end{equation}
Note that $\iiww\in \Hom(\iisL,\iisL)_{-s}$ for $i=a,b$. Hence
$\bbww^\mstar \eqref{eq:ss1} + \eqref{eq:ss3}\aaww$ yields
\[
-\bbGamma A + A\aaGamma \equiv 0 \pmod{\Hom(\aasL,\bbsL)_{-s^+}}.
\]
By the definition of block decomposition,  $\aaGamma \in
\aafgg_{\aaxx,-(^{b-1}r)^+}\subseteq \aafgg_{\aaxx,-r^+}$
and so 
$-\bbGamma A \in \Hom(\aasL,\bbsL)_{-s^+}$. 
On the other hand, the datum $\bbSigma$ is a single positive depth block so
multiplying by $\bbGamma$ induces an
isomorphism  $\bbsL_{x,t} \xrightarrow{\bbGamma \cdot \_\ } \bbsL_{x,t-r}$.
Hence $A \in \Hom(\aasL,\bbsL)_{s^+}$, i.e. $X\in \ssfgg_{x,s^+}$.
A similar argument yeilds $X'\in \ssfgg'_{x',s^+}$. 
This proves \Cref{claim:ssinj}.
\end{proof}

\begin{claim}\label{claim:dimss}
We have $\dim_\fffF \ssfgg_{x,s} = \dim_\fffF \abbfbb$ and
$\dim_\fffF \ssfggp_{x',s} = \dim_\fffF \babfbb$.
\end{claim}

\begin{proof}
We recall that $\bbww \bbsL_{t} = \bbsL'_{t-s}$ by the construction of
  lift of datum.
Hence $A \mapsto \bbww A$ induces an isomorphism
\[\ssfgg_{x,s:s^+} \cong \Hom(\aasL,\bbsL)_{s:s^+}
\simrightarrow \Hom(\aasL,\bbsL')_{0:0^+} = \abbfbb.
\] 
Similarly, $A' \mapsto A'^* \bbww$ induces an isomorphism
\[
\ssfgg'_{x',s:s^+}\cong \Hom(\aasL',\bbsL')_{s:s^+}\iso
\Hom(\bbsL,\aasL')_{0:0^+} = \babfbb. \qedhere
\]
\end{proof}

\Cref{claim:ssinj} and \Cref{claim:dimss} prove that $\ssbiota$ is an
isomorphism of $\fffF$-vector spaces.

The group $\dgS$ stabilizes the coset $w+ \dgsB_0 \in \dgW/\dgsB_0$.  Using this
fact and a direct computation show that $\ssbiota$ is $\dgS$-equivariant.
\end{proof}

\begin{remark} We only use the fact that $\aaww \in \aasB_{\aaxx,-s}$ when
    we prove the injectivity of $\ssbiota|_{\ssfgg_{x,s:s^+}}$.
		Therefore we could and will reuse this proof in \Cref{sec:Exhaust}.
\end{remark}


\subsection{Block decomposition of representations}\label{sec:HW.blocks}

As in the one block case, we consider the space $\sSB_{\Omega}$.  By
\Cref{lem:Saction},
$\sSB_{\Omega} = \Ind_S^\dbK \sS(\sB_0)_w = \Ind_S^\dbK \bomegaww$.
We now decompose $\bomegaww$ according to the decomposition of data. 

\subsubsection{}\label{sec:ssbomega}
\def\SH{\zeta}
\def\SHssbb{\SH^{\ssbfbb}} 
Using the formula in \eqref{eq:SH.b0}, we get
a morphism $\SHssbb \colon \ssdbJ \longrightarrow \bfH(\ssbfbb)$. Its natural
extension to $\dgS\ltimes \ssdbJ\longrightarrow \bfSH(\ssbfbb)$ is again denoted
by $\SHssbb$.

Let $\ssbomega$ denote the pull back of $\bomega_{\ssbfbb}$ via $\SHssbb$ which
is an $\dgS\ltimes \ssdbJ$-module realized on $\bS(\ssbfbb)$. More precisely,
for $h = (u,(\exp(X),\exp(X')) \in \dgS\ltimes \ssJ$, 
\begin{equation}\label{eq:ssbomega.def}
\ssbomega(h) := \bomega_{\ssbfbb}(u) \bomega_{\ssbfbb}(- (X,X')\cdot w)
\psi(\half \innw{w}{ (\exp(X),\exp(X'))^{-1}\cdot w - w}). 
\end{equation}

\begin{lemma} 
\label{lem:iso.KO}
We have  $\ssdbkappa \cong \ssbomega$ as $\dgS\ltimes \ssdbJ$-modules. 
\end{lemma}

\begin{proof}
  By \eqref{eq:Sym.s}, we have
  $\half \innw{w}{ h^{-1}\cdot w - w} \equiv \bB(X,\Gamma)+\bB(X',-\Gamma') \ 
  \pmod{\fpp}$
  for all $h =(\exp(X), \exp(X')) \in \ssdbJ\subseteq G_{x,s}\times G'_{x',s}$.
  Now the lemma follows immediately
  from \Cref{lem:TB1} with the same proof of \eqref{eq:bomega.d}.
\end{proof}

Since $\bfbb = \dgbfbb \oplus \ssbfbb$, we have
$\bomega_{\bfbb} = \bomega_{\dgbfbb} \boxtimes \bomega_{\ssbfbb}$ as
$\bfSH(\dgbfbb) \times \bfSH(\ssbfbb)$-module, realized on
\begin{equation}\label{eq:bS.dec}
\bS(\bfbb) = \bS(\dgbfbb)\boxtimes \bS(\ssbfbb)
\end{equation}

Note that $w = (\bbww,\aaww)\in \dgW$. Evaluation at $w$ gives an isomorphism 
of $\bC$-vector spaces 
 \begin{equation} \label{eq:bbsSB}
\xymatrix@R=0em{
 \bbsSB_{\bbww} \boxtimes \aasSB_{\aaww}  = \dgsSB_{w} \ar[r]^<>(.5){\sim} &
 \bS(\dgbfbb) = \bS(\bbbfbb)\boxtimes \bS(\aabfbb).
}
 \end{equation}
 Translating the $\dgS$-module (resp. $\iiS$-module for $i=a,b$) structure via
 \eqref{eq:bbsSB}, we let $\bomegadgS$ (resp. $\bomegaiiS$) be the resulting
 module acting on $\bS(\dgbfbb)$ (resp. $\bS(\iibfbb)$).  Clearly, by
 \Cref{lem:Saction},
\begin{equation}\label{eq:bomegadgS}
\begin{split}
\bomegadgS(h) 
 & =   \bomega_{\dgbfbb}(h)\bomega_{\dgbfbb}(h^{-1}\cdot w-w) \psi (
\half\inn{w}{h^{-1}\cdot w -w}_{\dgW} ) \\
& = \bomegabbS(\bbhh)\boxtimes
\bomegaaaS(\aahh)
\qquad \forall h = (\bbhh,\aahh)  \in \dgS.
\end{split}
\end{equation}

We state a key lemma for the induction process.

\begin{lemma} 
\label{lem:de.S}
By an abuse of notation, let $\bomega_{\dgS}$ also denote its inflation to
$\dgS\ltimes \ssdbJ$. 
Then  
\[
\bomega_w = \bomegadgS \otimes \ssbomega
\]
as $\dgS\ltimes \ssdbJ$-module under the  factorization \eqref{eq:bS.dec}.  
\end{lemma}

\begin{proof}
  Suppose $h \in \dgS$. Then $h^{-1}\cdot w -w \in \dgsB_0$, i.e. its component
  in $\sssB_0$ is zero. Therefore by \Cref{lem:Saction}, \eqref{eq:ssbomega.def}
  and \eqref{eq:bomegadgS},
\[
\begin{split}
  \bomegaww(h) & = \bomegab(h)\bomega_{\bfbb}(h^{-1}\cdot w-w)
  \psi\left(\half\innw{w}{h^{-1}\cdot w-w} \right) \\
  & = (\bomega_{\dgbfbb}\boxtimes \bomega_{\ssbfbb})(h)
  \bomega_{\dgbfbb}(h^{-1}\cdot w-w)
  \psi\left(\half\inn{w}{h^{-1}\cdot w-w}_{\dgW} \right)\\
  & = \bomegadgS(h) \otimes \ssbomega(h). \quad
\end{split}
\]

\def\dbxx{\breve{x}}
Suppose  $h = (\exp(X),\exp(X')) \in \ssdbJ\subseteq
G_{x,s}\times G_{x,s}$. Then $\bomega_{\bfbb}(h) = \id$. Since 
\[h^{-1}\cdot w - w \in  -(X,X')\cdot w + \sB_{0^+} \subseteq \dgsB_{0^+}\oplus \sssB_{0}
\]
we see that $\bomega_\bfbb(h^{-1}\cdot w - w) = \id_{\bS(\dgbfbb)}\boxtimes
\bomega_{\ssbfbb}(-(X,X')\cdot w)$. 
Putting the above into \eqref{eq:OS} gives 
\[\bomegaww(h) =
\bomega_{\ssbfbb}(-(X,X')\cdot w) \psi(\half\inn{w}{h^{-1}\cdot w-w}_{W})
 = \ssbomega(h).
\]
This completes the proof of the lemma. 
\end{proof}

\subsection{Proof of \mbox{\Cref{prop:CKtype}}}
\label{sec:CKtype.fini}
Note that $\aaSigma$ has $b-1$ blocks and the data $(\aaSigma, \aaSigmap)$
satisfies the induction hypothesis in \Cref{sec:indhyp}. Also note that
$\ssbomega$ is an irreducible $\ssdbJ$-module since it is a Heisenberg
representation.  Now
\[
\begin{split}
  & \Hom_{\dbK}(\dbeta, \sSB_{\Omega})  \\
  & = \Hom_{\dbK}(\dbeta, \Ind_{S}^{\dbK} \bomegaww) = \Hom_{S}(\dbeta,
  \bomegaww)   \\
  & = \Hom_{\dgS \ltimes \ssJ }(\dgdbeta\otimes \ssdbkappa, \bomegadgS \otimes
  \ssbomega) \hspace{4em} \text{(by \eqref{eq:kappa.2} and \Cref{lem:de.S})}\\
  & = \Hom_{\dgS} \left(\dgdbeta,\bomegadgS\otimes
    \Hom_{\ssJ}(\ssdbkappa,\ssbomega) \right) \\
  & = \Hom_{\dgS} ( \dgdbeta, \bomegadgS)  \hspace{8em} \text{(by \Cref{lem:iso.KO})}\\
  & = \Hom_{\bbS}(\bbdbeta, \bomega_{\bbS})\otimes
  \Hom_{\aaS}(\aadbeta,\bomega_{\aaS}).
\end{split}
\]
It has dimension $1$ by \Cref{prop:sSB} and the Induction Hypothesis.
This completes the induction process and proves the proposition. \qed

\section{Proof of the main theorem II: Exhaustion}\label{sec:exhaust}
In this section, we prove \Cref{it:Main.2} of the \Cref{thm:main}.

\subsection{Occurrence of  refined $K$-types}\label{sec:pf.main.2}
Part~\cref{it:Main.2} of the \Cref{thm:main} is a easy consequence of following
\Cref{prop:exhaust}. Its proof consists of the whole \Cref{sec:I.exhaust} which
uses the key identity \cref{eq:Pan}.

Recall the notion of $K$-type data in the remark of \Cref{def:SC.D} and its
extension to covering groups in \Cref{rmk:dt.genD} of \Cref{def:LD}.

\begin{prop}\label{prop:exhaust}
  Let $(G,G') = (\rU(V),\rU(V'))$ be a type~I reductive dual pair.  Let
  $\tD = (x,\Gamma,\phi,\rho,\xi) $ be a supercuspidal datum of $\wtG$.  Suppose
  that $\theta_{V,V'}(\tpiD)\neq 0$ (or equivalently $\omega[\tetaD]\neq 0$) with
  respect to $(G,G')$.  Then there exist a supercuspidal datum $\Sigma$ for $G$
  and a $K$-type datum $\Sigma'$ for $G'$ such that
\begin{enumerate}[(i)]
\item $\Sigmap$ is a theta lift of $\Sigma$
  (cf. \Cref{def:LDC}), and the pair $(\Sigma,\Sigmap)$ defines a splitting
  $\spt_{x,x'}$;

\item $\tD$ is equivalent to $\tSigma := (\Sigma,\spt_{x,x'})$, i.e. $\pi_{\tD}
  \simeq \pi_{\tSigma}$ and
   
\item $\omega[\etaSigma \boxtimes \etaSigmap] \neq 0$ under the splitting
  $\spt_{x,x'}$ where $\etaSigma$ and $\etaSigmap$ are the refined $K$ and
  $K'$-types defined by $\Sigma$ and $\Sigma'$ respectively.
\end{enumerate}
\end{prop}

\begin{proof}[Proof of \Cref{thm:main}~\Cref{it:Main.2}]\label{pf:main.2}
Let $\tD$ be a supercuspidal datum such that $\tpi  = \tpi_{\tD}$. By
\Cref{prop:exhaust}, we have $\Sigma$ and $\Sigmap$ such that $\tpi =
\tpiSigma:= \cInd_{\wtK}^{\wtG}\tetaSigma$ and 
\[
\begin{split}
0 &\neq \Hom_{\wtK\times\wtK'}(\tetaSigma\boxtimes \tetaSigmap, \omega)
= \Hom_{\wtG\times\wtK'}(\tpi \boxtimes \tetaSigmap, \omega)\\ 
&= \Hom_{\wtG\times \wtK'}(\tpi\boxtimes \tetaSigmap, \tpi\boxtimes \thetaVVp(\tpi)).
\end{split}
\] 
Therefore $\tetaSigmap$ occurs in $\thetaVVp(\tpi)$.  By
\cite[Proposition~17.2~(2)]{Kim}, we conclude that $\rho'$ in $\tSigmap$
is cuspidal since
$\tpi'= \thetaVVp(\tpi)$ is supercuspidal by assumption.  Hence
$\Sigmap = \dthetaVT(\Sigma)$ by definition, $\tSigma'$ is a supercuspidal datum
and $\tpi' = \tpiSigmap$.
\end{proof}

\begin{remark}
  In the proof of \cite[Proposition~17.2]{Kim}, \cite[Lemma~15.4]{Kim} is used
  to treat the depth-zero case (see \cite[p~315]{Kim}). The covering group
  version of this lemma also holds since ``The proof of Proposition 6.7 in
  \cite{MP2} goes through without changes'' as stated in \cite[proof of
  Theorem~3.10]{HW}.  Meanwhile the other parts of the proof of
  \cite[Proposition~17.2~(2)]{Kim} only involves subgroups of $G$ which split
  canonically (because they are either unipotent or pro-$p$). Thus the proof
  \cite[Proposition~17.2]{Kim} also adapts ``mutatis mutandis''.
\end{remark}

\def\pretaD{{\pr_{\tetaD}}}
\def\prr{{\pr_{G_{x,r^+}}}}
\def\prGamma{{\pr_{[\Gamma]}}}


\subsection{Proof of \Cref{prop:exhaust}} \label{sec:I.exhaust} When $\tD$ has
no positive depth block, i.e. $\tD$ is a depth zero data, this is proved by Pan
in \cite{Pan02J}. See \Cref{sec:TL.zero} and in particular \Cref{thm:Pan0}.

We prove by induction on the number of blocks similar to \Cref{sec:CKtype.gen}.  We
now assume $\tD$ has $b$ positive depth blocks with $b > 0$.

\begin{IndH}
  Assume \Cref{prop:exhaust} holds for $(\smtD,\smV,\smVp)$ where
  $\dim \smV\leq \dim V$, $\dim \smVp \leq \dim V'$ and $\smtD$ has $\smbb$
  positive depth blocks with $b>\smbb$.
\end{IndH}

\subsubsection{} 
Suppose $\tD = (x,\Gamma,\phi,\rho, \spt)$. Let $\lD = (x,\Gamma,\phi,\rho)$ and
$\lD = \bblD\oplus \aalD$ be the decomposition of $\lD$ as that for $\Sigma$ in
\Cref{sec:decSigma} so that $\depth(\bblD)=r$ and $\depth(\aalD)<r$.  We adopt
the notation defined in \Cref{def:sskappa} with respect to $\lD$.

Let $\prr$ (resp. $\prGamma$ and $\pretaD$) be the projection operator to the
$G_{x,r+}$-invariant spaces (resp. the $\psi_{\Gamma}|_{G_{x,s^+}}$-isotypic
component and the $\tetaD$-isotypic component). Clearly,
$\prGamma = \prGamma \circ \prr$ and $\pretaD\circ \prGamma = \pretaD$.


\begin{lemma} \label{lem:ww}
There exist an $x''\in \BTB{G'}$ and a $w \in \sB_{x,x'',-s}$ such that 
\begin{enumerate}[(i)]
\item \label{it:ww.1} $\pretaD(\sS(\sB_{x,x'',0})_w) \neq 0$ and
	
\item \label{it:ww.2} $M(w) \in \Gamma + \bbfgg_{\bbxx,-s} \oplus
  \aafgg_{\aaxx,-s} \subseteq \bbfgg\oplus \aafgg$.
\end{enumerate}
\end{lemma}

\begin{proof}
  \def\sBxys{\sS_{\sB_{x,y',-s}}} 
  Since $\Jump(\sL_x)\subseteq \bQ$, we have
  \begin{equation}\label{eq:Pan}
    \sS^{G_{x,r^+}} = \sum_{y'\in \BTB{G'}} \sBxys.  
  \end{equation}
  See \Cref{sec:IPD} for the notation and a quick proof of \cref{eq:Pan}.

  Note that $K\subseteq G_x$ preserves $\sBxys$.  Therefore, we can find an
  $x'' \in \BTB{G'}$ such that $\pretaD(\sS_{\sB_{x,x'',-s}}) \neq 0$ since
  $\pretaD(\sS)\neq 0$.

We denote $\sB_{x,x'',t}$ by $\sB_{t}$. 
Since $\sB_{-s} = \bigcup_{w \in \sB_{-s}} (w + \sB_0)$, there is a $w \in
\sB_{-s}$ such that
$\pretaD(\sSB_{w}) \neq 0$. Clearly $\prGamma(\sSB_w)\neq 0$.  

\setcounter{claim}{0}
\begin{claim*}
  If $\prGamma(\sSB_w)\neq 0$, then $M(w) \in \Gamma + \fgg_{x,-s}$.
\end{claim*}

\begin{proof}
  Let $g = \exp(X) \in G_{x,s^+}$ with $X \in \fgg_{x,s^+}$.  By
  \eqref{eq:bomegaS.sp},  $\omega(g) f = \psi_{M(w)}(g) f$ for each
  $f\in \sSB_{w+\sB_{0}}$.  Hence $\prGamma(\sSB_{w+\sB_{0}}) \neq 0$ is
  equivalent to $\psi_{M(w)}(g) = \psi_\Gamma(g)$ for all $g\in G_{x,s^+}$. This
  is equivalent to
\[
\psi(\bB(M(w)-\Gamma,X)) = 1 \quad \forall X \in \fgg_{x,s^+}.
\]
Since $\psi$ is a non-trivial character with conductor $\fppF$, the above
condition is equivalent to $\bB(M(w)-\Gamma,\fgg_{x,s^+}) \subseteq \fppF$,
i.e. $M(w) \in \Gamma + \fgg_{x,-s}$.  
\end{proof}

\smallskip

By \Cref{prop:GD}, let $\ckGamma$ be a $\GL(\bbV)$-good element in
$\bbfgg_{\bbxx,-r}$ representing $\bbGamma+ \bbfgg_{\bbxx,-r^+}$. Then $\ckGamma$
is also a good element in $\fgg_{x,r}$ representing
$\Gamma+\fgg_{x,-r^+}$. Clearly
$\ckG := \Cent{G}{\ckGamma} \subseteq \bbG\times \aaG$.

We now recall a result of Kim-Murnaghan.
\begin{lemma*}[{\cite[Lemma~5.1.3~(3)]{KM2}}]
Let $x\in \BTB{\ckG}$ and $X\in \ckfgg_{x,-r}\cap \ckfgg_{-r^+}$. 
Then for $t>-r$ we have
\[
(\Ad G_{x,r+t})(\ckGamma+X + \ckfgg_{x,t}) = \ckGamma + X +
\fgg_{x,t}.\qedhere
\]
\end{lemma*}

Setting $X = \Gamma- \ckGamma$, and $t = -s$, the above lemma gives
$(\Ad G_{x,s})(\Gamma + \ckfgg_{x,-s}) = \Gamma + \fgg_{x,-s}$.
In other words, there is an $h\in G_{x,s}$ such that
\[
M(h\cdot w) = h\cdot M(w) \in \Gamma + \ckfgg_{x,-s} \subseteq (\bbGamma +
\bbfgg_{\bbxx,-s}) \oplus \aafgg_{\aaxx,-s}.
\]
Since $G_{x,s}$ normalizes $\tetaD$,
$\pretaD(\sSB_{h\cdot w}) = h\cdot \pretaD(\sSB_w) \neq 0$. Therefore by
replacing~$w$ with $h\cdot w$, we may assume that $w$ satisfies
\Cref{lem:ww}~\cref{it:ww.2}.

This completes the proof of \Cref{lem:ww}.
\end{proof}

\subsubsection{} Let $x''$ and $w$ satisfy \Cref{lem:ww}. Let $\sB :=
\sB_{x,x''}$ and 
\begin{equation} \label{eq:bbXaaX}
M(w) = \bbX + \aaX
\end{equation}
where $\bbX \in \bbGamma + \bbfgg_{\bbxx,-s}$ and $\aaX \in \aaGamma + \aafgg_{\aaxx,-s}$.
 
We define $\bbV' = w (\bbV)$ and $\aaV' = \bbV'^\perp$. Let $\iiww := w|_{\iiV}$
for $i = a$ and $b$.

\begin{lemma} 
  \label{lem:wbbwaa} 
The following statements hold:
  \begin{enumerate}[(i)]
  \item \label{it:wba.1} Restricting on $\bbV$, the map
    $w|_{\bbV} \colon \bbV \iso \bbV'$ is an isomorphism.
    
  \item \label{it:wba.2} The restriction of $\inn{}{}_{V'}$ to $\bbV'$ is
    non-degenerate. In particular, $V' = \bbV' \oplus \aaV'$. 
    
  \item \label{it:wba.3} The image $w(\aaV) \subseteq \aaV'$, i.e.
    $\aaww\in \Hom(\aaV,\aaV')$.
  \end{enumerate}
\end{lemma}

\begin{proof}
\begin{asparaenum}[(i)]
\item All elements in $\bbGamma + \bbfgg_{\bbxx,-s}$ are invertible elements in
  $\End(\bbV)$. In particular $\bbX$ is invertible. Since $w^\mstar w = \bbX$
  restricted on $\bbV$, $w \colon \bbV \rightarrow \bbV'$ is an injection and
  hence an isomorphism.

\item Let $v_1, v_2 \in \bbV$. Then
  $\inn{wv_1}{wv_2}_{V'} = \inn{v_1}{M(w) v_2}_{V} = \inn{v_1}{\bbX
    v_2}_{\bbV}$. Since $\bbX$ is invertible, the claim follows.

\item Suppose $v \in \aaV$. Then $\inn{wu}{wv}_{V'} = \inn{-\bbX u}{v}_{V}
  = 0$ for all $u \in \bbV$. Hence $wv \in \bbVp^\perp = \aaV'$. \qedhere

\end{asparaenum}
\end{proof}

By \Cref{lem:wbbwaa}, we see that $\iiww\in \Hom(\iiV,\iiV')$ and
$w = \bbww \oplus \aaww$ is a block diagonal decomposition.  Moreover
$\iiM(\iiww) = \iiX$ for $\iiX$ in \eqref{eq:bbXaaX} where $\iiM$ is the moment
map defined with respect to the dual pair $(\rU(\iiV),\rU(\iiVp))$.

Let $\iisL = \sL \cap \iiV$ which is the lattice function corresponding to
$\iixx\in \BTB{\iiG}$ and $\sL = \bbsL \oplus \aasL$.  The following lemma says
that the lattice function $\sL'$ corresponding to $x''$ is split under
$V' = \bbVp\oplus \aaVp$.

\begin{lemma} \label{lem:decompLp}
Let  $\iisLp_t = \sL'_t \cap \iiVp$ for $i = b,a$. Then 
\begin{enumerate}[(i)] 
\item \label{it:dec.Lp.1} $\bbsL'_{t} = \bbww(\bbsL_{t+s})$ and it is self-dual
  in $\bbV'$,

\item \label{it:dec.Lp.2} $\sL' = \bbsL' \oplus \aasL'$ and

\item \label{it:dec.Lp.3} $\aasL'$ is self-dual\footnote{We warn that
    $\aasL_t' \neq w_a \aasL_{t+s}$.}.

\end{enumerate}
\end{lemma}

\begin{proof}
\begin{asparaenum}[(i)]
\item Let $\bbsL''_{t} := \bbww(\bbsL_{t+s}) = w(\bbsL_{t+s})$. By
  \Cref{lem:L'}, it is a self-dual lattice function in $\bbV'$. Since
  $w\in \sB_{-s}$, $\bbsL''_t \subseteq \sL'_{t}\cap \bbV' = \bbsL'_{t}$ for all
  $t\in \bR$.  Taking dual lattice in $\bbV$, we have
  $\bbsL''_{-t^+} = (\bbsL''_t)^* \supseteq (\bbsL'_{t})^* \supseteq
  (\sL'_{t})^* \cap \bbV' = \bbsL'_{-t^+}$ for all $t\in \bR$. Hence \cref{it:dec.Lp.1} holds.
\item Obviously $\sL'_t \supseteq \bbsL'_t \oplus \aasL'_t$.  Conversely
    let $v = \bbvv + \aavv\in \sL'_{-t^+} = (\sL'_{t})^*$ with
    $\iivv \in \iiVp$.  Then $\bbvv\in \bbsLp_{-t^+} = (\bbsLp_{t})^*$ since
    $\inn{\bbvv}{\bbsLp_{t}}_{\bbV} = \inn{v}{\bbsLp_{t}}_{V'} \subseteq
    \inn{v}{\sLp_{t}}_{V'} \subseteq \fppD$. Now
    $\aavv = v - \bbvv \in (\sL'_{-t^+} + \bbsL'_{-t^+})\cap \aaV' =
    \aasL'_{-t^+}$.
    Therefore, $\sL'_{-t^+} \subseteq \bbsL'_{-t^+}  \oplus \aasL'_{-t^+}$ for
    any $t\in \bR$. This proves \Cref{it:dec.Lp.2}.
  \item By \cref{it:dec.Lp.1,it:dec.Lp.2},
    $\bbsL'_{-t^+}\oplus \aasL'_{-t^+} =\sL'_{-t^+} = (\sL'_{t})^* =
    (\bbsL'_{t})^* \oplus (\aasL'_{t})^* = \bbsL'_{-t^+} \oplus (\aasL'_{t})^*$.
    Hence $\aasL'_{-t^+} = (\aasL'_{t})^*$, i.e. $\aasL'$ is self-dual. \qedhere
\end{asparaenum}
\end{proof}

Note that $\bbM(\bbww) \in \bbGamma + \bbfgg_{x,-s}$.  By \Cref{lem:surj}, there
is a $\bbww_0 \in \bbww+\bbsB_0$ such that $\bbM(\bbww_0) = \bbGamma$. Replacing
$w$ with $\bbww_0 \oplus \aaww\in w+\bbsB_0\subseteq w+ \sB_0$, we assume that
$\bbM(\bbww) = \bbGamma$ from now on.

\subsubsection{} \label{sec:Exhaust}

We retain all the notations in
\Cref{def:block.theta}~\cref{it:blk.1,it:blk.2,it:blk.4}.  On the other hand, we
do not have enough information about $\Sigma'$ to define $K'$ at the moment.
Instead, we replace $\dbK$ by $K$ in \Cref{def:block.theta}~\cref{it:blk.6} and
define the following notations.
\begin{enumerate}
\item[(\ref{it:blk.6}')] Let $\OmegaK := K w + \sB_0$ and
  $\SK := \Stab_{K}(w+\sB_{0})$.  \\
  Let $\iiOmegaK := \iiK \iiww + \iisB_{0}$,
  $\iiSK := \Stab_{\iiK}(\iiww+\iisB_{0})$ for $i = a, b$. \\
  Let $\dgSK := \bbSK\times \aaSK = \Stab_{\bbK\times \aaK}(w+\sB_{0})$ so that
  $\SK = \dgSK \ssJ$.
\end{enumerate}
By the remark of \Cref{lem:TB1}, we see that the map


\def\ssbfbbKp{\ssbfbb_K^\perp}
\[
\xymatrix@R=0em{
\ssbiota_K\colon \ssfgg_{x,s:s^+} \ar[r]&
\sssB_{0:0^+} = \ssbfbb 
}
\]
induced by $ X\mapsto X\cdot w$ is an $\dgS_K$-equivariant injection between
$\fffF$-modules. Moreover $\ssbiota_K$ is an isometry with respect to natural
symplectic forms of the domain and codomain. Let $\ssbfbbKp$ be the orthogonal
complement of the image of $\ssbiota_K$ in $\ssbfbb$.  

The next lemma is a variation of \Cref{lem:iso.KO} and \Cref{lem:de.S} which
follows by the same arguments. 
\begin{lemma} \label{lem:iso.e} Let $\sskappa$ be the Heisenberg-Weil
  representation of $\dgK \ltimes \ssJ$ defined in \linebreak[4]
  \Cref{def:sskappa}~\cref{it:sskappa.def}.  Let $\bomegadgSK$ and $\ssbomega$
  be the $\dgSK\ltimes \ssJ$-modules realized on $\bS(\dgbfbb)$ and
  $\bS(\ssbfbb)$ respectively as in \Cref{sec:ssbomega} (see
  \eqref{eq:ssbomega.def} and \eqref{eq:bomegadgS}).  Then, as
  $\dgS_K\ltimes\ssJ$-modules,
\begin{enumerate}[(i)]
\item
$
\ssbomega \cong \sskappa\otimes \bS(\ssbfbbKp)
\cong (\sskappa)^{\oplus c}
$
where $c = \dim \bS(\ssbfbbKp) =
\left(\#\ssbfbbKp\right)^\half$,
\item $\bomegadgSK = \bomegabbSK\boxtimes \bomegaaaSK$, and

\item 
$
\bomegaSK \cong \bomegadgSK \otimes \ssbomega. 
$\qed
\end{enumerate}
\end{lemma}

\subsubsection{}

Let $\nD = (x,\Gamma, \phi, \nrho)$ where
$\nrho := \rho\otimes \MU{\xi}{\xi_{x,x''}}$ and let $\nD = \bbnD\oplus \aanD$
be the corresponding block decomposition.  Note that
$(\nD, \xi_{x,x''})\sim \tD$ by definition and $\pretaD(\sSB_{w})$ is nonzero
(see \Cref{def:cSC.eq} and \Cref{lem:ww}~\cref{it:ww.1}).  By
\Cref{lem:iso.e}, we have
\[
\begin{split}
  0 & \neq \Hom_{\wtK}(\tetaD,\sSB_{\Omega_K}) = \Hom_{K}(\etanD,
  \sSB_{\Omega_K})\\
  & = \Hom_{K}(\etanD, \Ind_{\SK}^K \bomegaSK) = \Hom_{\SK}(\etanD, \bomegaSK)\\
  & = \Hom_{\dgS\ltimes \ssJ}(\dgetanD\otimes \sskappa,\bomegadgSK \otimes
  \ssbomega)
  \\
  & = \left(\Hom_{\dgSK}(\dgeta, \bomega_{\dgSK})\right)^{\oplus c} \\
  & = \left(\Hom_{\bbSK}(\etabbnD, \bomegabbSK) \boxtimes \Hom_{\aaSK}(\etaaanD,
    \bomegaaaSK)\right)^{\oplus c}.
\end{split}
\]

In particular,
$ 0\neq \Hom_{\aaSK}(\etaaanD, \bomega_{\aaSK}) \subseteq \Hom_{\aaK}(\etaaanD,
\sS(\aasB_0))$.

Let $\aatnD := (\aaxx, \aaGamma, \aaphi, \aanrho, \spt_{\aaxx,\aaxx''})$. Then
$\aatnD$ has $b-1$ blocks. Applying the induction hypothesis to
$(\aatnD,\aaV,\aaV')$, we get $\aaSigma$ and
$\aaSigmap = \dtheta_{\aaV,\aaV'}(\aaSigma)$.

Now we define 
\begin{enumerate}[(a)]
\item
  $x' := (\bbxx'',\aaxx')\in \BTB[F]{\bbG'}\times \BTB[F]{\aaG'}\subseteq
  \BTB[F]{G'}$;

\item $\bbSigma := (\bbxx, \bbGamma, \bbphi, \bbnnrho)$ where
  $\bbnnrho := \bbrho\otimes (\MU{\xi}{\xi_{x,x'}}|_{\bbK}) = \bbnrho\otimes
  (\MU{\xi_{x,x''}}{\xi_{x,x'}}|_{\bbK})$;
  
\item $\Sigma := \bbSigma \oplus \aaSigma$, $\bbSigmap := \dthetap(\bbSigma)$
  with respect to $\bbww$ and $\Sigmap := \bbSigmap \oplus \aaSigmap$.
\end{enumerate}
Obviously, $\Sigma'$ is a lift of $\Sigma$ and $x'$ occur as a part of the datum
$\Sigma'$.  By the functoriality of the construction of lattice model, one can
see that
\[
\MU{\xi_{\aaxx,\aaxx''}}{\xi_{\aaxx,\aaxx'}} =
\MU{\xi_{(\bbxx,\aaxx),(\bbxx'',\aaxx'')}}{\xi_{(\bbxx,\aaxx),(\bbxx'',\aaxx')}}|_{\aaG_{\aaxx}}
= \MU{\xi_{x,x''}}{\xi_{x,x'}}|_{\aaG_{\aaxx}}.
\]
Hence, we conclude that $(\Sigma, \xi_{x,x'})\sim \tD$. 
Applying the argument in \Cref{sec:CKtype.gen}, we conclude that
\[
0\neq \Hom_{K\times K'}(\etaSigma \boxtimes
\etaSigmap,\sS(\sB_{x,x';0})_{KK\cdot w}). 
\]
This finishes the proof of \Cref{prop:exhaust} and hence also completes the proof
of \Cref{it:Main.2} of the
\Cref{thm:main}. \qed

\appendix
\section{Heisenberg-Weil representations} \label{sec:AppendixA}

\def\bfSH{\mathbf{SH}}

\def\psiGa#1{\overline{\bB(#1,\Gamma)}}
\def\SH{\zeta}

In this appendix, we collect some facts about Heisenberg-Weil representations. 

\subsection{Heisenberg-Weil representation after G\'erardin}\label{sec:HW}
Let $\fff$ be a finite field with $q$ elements and let $\bpsi$ be a nontrivial
character of $\fff$.  Let $\bfW$ be a non-degenerate symplectic space over
$\fff$. Let $\bfH(\bfW) = \bfW \times \fff$ and $\bfSp(\bfW)$ denote the
corresponding Heisenberg group and symplectic group as usual.  We let
$(\bomega_\bfW,\bS(\bfW))$ denote the space of the Heisenberg-Weil representation of
$\bfSH(\bfW):= \bfSp(\bfW)\ltimes \bfH(\bfW)$ with central character $\bpsi$
realizing on $\bS(\bfW)$.  In \cite{Ger}, G\'erardin carefully studied the
isomorphism class of $\bomega_{\bfW}$.  We now recall the mixed model of this
representation. See \cite[\Sec{2}]{Ger} for details.

For any subspace $\bfV\subseteq \bfW$, let $\bfH(\bfV)$ be its inverse image in
$\bfH(\bfW)$ under the projection $\bfH(\bfW)\rightarrow \bfH(\bfW)/\fff=\bfW$.
Let $\bfW_+$ be a non-trivial totally isotropic subspace of $\bfW$.  Then
$\bfW_0 := \bfW_+^\perp/\bfW_+$ is naturally a non-degenerate symplectic space.
Let $\bfP(\bfW_+)$ be the parabolic subgroup stabilizing $\bfW_+$.  By an abuse
of notation, we let $\bomega_{\bfW_0}$ denote the pull back of $\bomega_{\bfW_0}$ to
$\bfP(\bfW_+)\ltimes \bfH(\bfW_+^\perp)$ via the natural quotient
\[
\xymatrix{ \bfP(\bfW_+)\ltimes \bfH(\bfW_+^\perp) \ar@{->>}[r]&
  \bfSp(\bfW_0) \ltimes \bfH(\bfW_0). }
\]
Let $\chi^{\bfW_+}$ be the (unique real) character of $\bfP(\bfW_+)$ given by
$g\mapsto (\det g|_{\bfW_+})^{(q-1)/2} \in \{ \pm 1 \}$ for all
$g \in \bfP(\bfW_+)$.  Then
\begin{enumerate}[(i)]
\item \label{it:fWeil.1}
  $\bomega_\bfW$ is the unique $\bfSH(\bfW)$-module extending
$\Ind_{\bfP(\bfW_+)\ltimes\bfH(\bfW_+^\perp)}^{\bfP(\bfW_+)\ltimes \bfH(\bfW)}
(\chi^{\bfW_+} \otimes \bomega_{\bfW_0})$.
\item 
  \label{it:fWeil.2}
  Fix a totally isotropic subspace $\bfW_-$ such that
  $\bfW = \bfW_- \oplus \bfW_+^\perp$, then the induced module in
  \cref{it:fWeil.1} could be identified with the set of functions on $\bfW_-$
  with values in $\bS(\bfW_0)$. The group actions could be easily work out.
  
\item \label{it:fWeil.3} The space $(\bomega_{\bfW})^{\bfW_+}$ of
  $\bfW_+$-invariants in $\bomega_{\bfW}$ is isomorphic to $\bomega_{\bfW_0}$ as
  an $\bfH(\bfW_+^{\perp})$-module. Moreover $\bfP(\bfW_+)$ acts by
  $\chi^{\bfW_+}\otimes \bomega_{\bfW_0}$.

\item \label{it:fWeil.4} The module $\bomega_{\bfW}$ has dimension
  $\sqrt{\# \bfW} = q^{\half \dim_\fff \bfW}$.
\end{enumerate}

Note that when $\bfW_+$ is a maximal isotropic subspace in $\bfW$, we have
$\bfH(\bfW_0) = \fff$ and $\bomega_{\bfW_0} = \bpsi$ so that we get the
Schr\"odinger model of $\bomega_{\bfW}$.

\subsection{Construction of $\kappa$.}
Following \cite{Yu}, we discuss the construction of the $K^i$-module $\kappa^i$
which extends $\psi_\Gamma|_{K^i_+}$.

\subsubsection{Special isomorphism}\label{sec:Special}
As Yu \cite{Yu} has pointed out, the extension of a Heisenberg
representation to a ``Weil representation'' of $K$ is subtle.
The problem is that, $\bfH(\bfW)$ has a large subgroup of the automorphism group
(isomorphic to $\bfW$) whose
action on the center $\fff$ and on $\bfH(\bfW)/\fff$ are identity. 
Therefore, $J\rightarrow \bfH(\bfW)$ in \cite[\Sec{11}]{Yu} is far from
unique and, Yu gives a canonical construction from root datum. 

We retain the notation and situation in \cite[\Sec{11}]{Yu} and \cite{Kim}:
\begin{enumerate}[(i)]
\item $\Gamma\in \fgg$ is a good element of depth $-r$;
\item $(\ckG,G)$ is a tamely ramified twisted Levi sequence with
  $\ckG = \Cent{G}{\Gamma}$;
\item $\BTB{\ckG} \hookrightarrow \BTB{G}$ is a fixed embedding of buildings;
\item  $x\in
  \sB(\ckG)$;
\item $\fgg = \ckfgg\oplus \ckfgg^\perp$ is an orthogonal decomposition with
  respect to the form $\bB$ in \Cref{sec:CG};
\item $J = (\ckG,G)(F)_{x,(r,s)}$ and $J_+ = (\ckG,G)(F)_{x,(r,s^+)}$. See
    \cite[p. 586]{Yu}\footnote{In our cases,
      $(\ckG,G)(F)_{x,(r_1,r_2)} = \exp(\ckfgg_{x,r_1} \oplus
      \ckfgg_{x,r_2}^\perp)$ for $0 < \half r_1 \leq r_2$.}.
\end{enumerate}

Taking a clue from \Cref{lem:Saction}, we could defined a ``canonical'' morphism
for $J$ below.  The symplectic space $\bfW = J/J_+$ is identified with
$\ckfgg^{\perp}_{x,s:s^+}$ via the exponential map. Suppose
$\barX_1,\barX_2\in \ckfgg^{\perp}_{x,s:s^+}$ with lifts
$X_1,X_2\in \ckfgg^{\perp}_{x,s}$ respectively. Then we have a non-degenerate
symplectic form on $\bfW$ given by
$\inn{\barX_1}{\barX_2} =\psiGa{[X_1,X_2]} \in \fff$ (cf. \cite[Lemma
11.1]{Yu}).  By the Baker-Campbell-Hausdorff formula, we have a group
homomorphism\footnote{ We check that $\SH|_J$ is a group homomorphism.  Indeed
  by the Baker-Campbell-Hausdorff formula
  $\log(\exp(X)\exp(Y)) \equiv X +Y + \half[X,Y] \pmod{\fgl_{x,r^+}}$ and 
\[
\SH(e^X) \SH(e^Y) = (\bar{X},\psiGa{X}) \cdot (\bar{Y}, \psiGa{Y}) = (\bar{X} +
\bar{Y},\psiGa{X+Y+\half[X,Y]}) = \SH(e^Xe^Y).
\] 
}
\begin{equation}\label{eq:Special}
  \SH\colon J \longrightarrow \bfH(\bfW) = \bfW \times \fff 
  \quad \text{given by} \quad \exp(X) \mapsto
  (\barX,\psiGa{X}). 
\end{equation}

Note that $\SH$ agrees with the special morphism defined in
\cite[Section~11]{Yu} since they agree on root subgroups.  By an abuse of
notation, we also let $\SH$ denote its natural extension
\[\SH \colon \ckG_{x}\ltimes J \longrightarrow \bfSp(\bfW)\ltimes \bfH(\bfW) =
\bfSH(\bfW).
\] 

\subsubsection{}\label{sec:kappa}
We fix a good factorization $\Gamma = \sum_{i=0}^d \Gamma_i$ and therefore get a
sequence of subgroups~$G^i$ as in \Cref{def:GD}.  We follow the notation in
\cite[p.591]{Yu}:
$K^i := K \cap G^i = G_x^0 G_{x,s_0}^1\cdots G_{x,s_{i-1}}^{i}$,
$J^i := (G^{i-1}, G^i)_{r_{i-1},s_{i-1}}$,
$J^i_+ := (G^{i-1},G^i)_{r_{i-1}, s_{i-1}^+}$ and $\bfW^i := J^i/J^i_+$.

Now we define a sequence of representations $\kappa^i$ of $K^i$ inductively such
that $K^i_+$ acts by the character $\psi_{\Gamma}$.
This is essentially Yu's construction in
\cite[\Sec{4}]{Yu}.

\begin{enumerate}[1.]
\setcounter{enumi}{-1}
\item 
First we set $\kappa^0 = \phi$ (cf. \Cref{def:SC.N}~\cref{it:SC.N.6}). 

\item [] Suppose we have constructed $\kappa^{i-1}$. We now construct the
  $K^i$-module $\kappa^i$:
\item Note that $K^i = K^{i-1} J^{i}$.  Let
  $\zeta^i \colon K^{i-1} \ltimes J^i \longrightarrow \bfSH(\bfW^i)$ denote the
  special homomorphism with respect to the Levi sequence $(G^{i-1},G^i)$ and the
  good element $\Gamma_{i-1}$ (cf. \Cref{sec:Special}).
  Let $\bomega^i_{\Gamma_{i-1}}$ denote the $K^{i-1} \ltimes J^i$-module
  obtained by pulling back of $\bomega_{\bfW^i}$ (cf. \Cref{sec:HW}) via
  $\SH^i$.

\item We set $\Gamma^i := \sum_{l=i}^d \Gamma_l$ which is in the center of $\fgg^i$. We
  see that $\psi_{\Gamma^i}$ is a character of $G^i_{x,0^+}\supseteq J^i$.  Let
  $\tv\times \psi_{\Gamma^i}$ be its extension to $K^{i-1}\ltimes J^i$ such that
  $K^{i-1}$ acts trivially.  As a subgroup of $J^i$, $J^{i-1}_+$ acts by the
  character $\psi_{\Gamma}$ on
  $\bomega^i_{\Gamma_{i-1}} \otimes (\tv \times \psi_{\Gamma^i})$.

\item We inflate $\kappa^{i-1}$ to a $K^{i-1}\ltimes J^i$-module. Since
  $K^{i-1}\cap J^i = G^{i-1}_{x,r_{i-1}} \subseteq J^i_+\cap K^{i-1}_+$, the
  $K^{i-1}\ltimes J^i$-module
  $\kappa^{i-1} \otimes \bomega^i_{\Gamma_{i-1}} \otimes (\tv \times
  \psi_{\Gamma^i})$
  factors through $K^{i-1} \ltimes J^i \longrightarrow K^{i-1}J^i = K^i$. Let 
  $\kappa^i$ be the corresponding $K^i$-module. It is clear that $K^i_+ =
  K^{i-1}_+J^i_+$ acts by $\psi_{\Gamma}$. 
\end{enumerate}

\section{A quick proof of a result of Pan}\label{sec:IPD}
\def\fracmm{\frac{1}{m}} 
\def\fracdmm{\frac{1}{2m}} 

As the reader may notice, \cref{eq:Pan} is a generalization of \cite[Proposition
6.3]{Pan02D}. Our proof follows Pan's idea.  Although we use the exponential map
to identify $\fgg_{x,r}$ with $G_{x,r}$, the statements and proofs in this
appendix also hold if we replace the exponential map by the ``Cayley transform''
and in which case we only assume the residual characteristic $p\neq 2$. To ease
notation, we normalize the valuation map such that $\val(\varpiD) = 1$.

\subsection{Invariant vectors under the action of lattices}
\label{sec:IPD.supp}
Let $\sS$ be any realization of the oscillator representation of
$\wtSp(W)\ltimes \rH(W)$. Here $\rH(W):=W\times F$ denotes the Heisenberg group
of $W$ and we identify $W$ as a subset of $\rH(W)$.

Suppose $L$ is a lattice in $W$ such that $L \supseteq A^*$ for a certain good
lattice $A$. Let $\sS^{L^*}$ denote the space of ${L^*}$-fixed vectors in $\sS$
under the Heisenberg group $\rH(W)$ action.  Let $\sS(A^*)$ be the generalized
lattice model of the oscillator representation with respect to $A^*$ and let
$\sS(A^*)_L$ be the subspace of functions in $\sS(A^*)$ supported on $L$.  We
identify $\sS$ with $\sS(A^*)$ via a fixed intertwining map.  It is easy to
see that $\sS^{L^*}$ is exactly the image of $\sS(A^*)_L$
(cf. \cite[Lemma~8.2]{Pan02D}).  Since $\sS^{L^*}$ neither depends on the choice
of $A$ nor the choice of intertwining map, it makes sense to let $\sS_L$ denote
$\sS^{L^*}$ to emphasis that it is the space of functions with support on $L$
under the generalized lattice module with respect to any $A^*\subseteq L$.  In
particular, for a self-dual lattice function $\sB$ in $W$, we identify
  $\sS^{\sB_{s^+}}$ with $\sS(\sB_0)_{\sB_{-s}}$ and denote it by
  $\sS_{\sB_{-s}}$.

\subsection{Proof of \cref{eq:Pan} and depth preservation}
We only need to consider rational points in the building for our study of
minimal $K$-types. These points correspond to lattice functions with rational
jumps.

In the rest of this section, we will prove the following theorem which is a
slightly stronger version of \cref{eq:Pan}.

\begin{thm}\label{thm:IPD}
Suppose $\Jump(\sL_x) \subseteq \fracmm\bZ$ so that $\Jump(\fgg_{x}) \subseteq
\fracmm\bZ$ for certain positive integer $m$. 
Let
\[
\BTB{G'}_{2m} := \set{y\in \BTB{G'}|\Jump(\sL'_y)\subseteq
  \fracdmm\bZ}.
\] 
Then for all $0\leq r \in \fracmm \bZ$,
\[
\sS^{G_{x,r^+}} = \sum_{y\in \BTB{G'}_{2m}} \sS_{\sB_{x,y,-r/2}}.
\]
\end{thm}

\subsubsection{}
We will call $\sL$ an $\fooD$-module function in $V$ if $\sL_s$ is only an
$\fooD$-submodule in $V$ in \Cref{def:latticefn}.  In this case,
$\sL_s \otimes_{\fooD} D$ may not equal to $V$ and
\[
\sL_s^* := \set{v\in V|\innv{v}{\sL_s}\subseteq \fppD}
\]
 may not be a lattice.

The following are the key lemmas:
\begin{lemma}[c.f. {\cite[Lemma~10.1]{Pan02D}}] \label{lem:SLattice} 
  Suppose
  $\sN$ is an $\fooD$-module function in $V$ such that
  $\Jump(\sN) \subseteq \fracdmm \bZ$ and
  $\innv{\sN_{t_1}}{\sN_{t_2}}\subseteq \fppD^{\ceil{t_1+t_2+\fracmm}}$.  Then
  there is a self-dual lattice function $\sL$ such that
  $\sN_t \subseteq \sL_{t+\fracdmm}$ and $\Jump(\sL)\subseteq \fracdmm \bZ$.
\end{lemma}
\begin{proof}
  Since
  $\innv{\sN_{\frac{i}{2m}}}{\sN_{\frac{-i-1}{2m}}} \subseteq
  \fppD^{\ceil{\frac{i-i-1}{2m}+\frac{1}{m}}} = \fppD$,
  we have $\sN_{\frac{i}{2m}} \subseteq \sN_{\frac{-i-1}{2m}}^*$.  In
  particular, we have $\sN_0 \subseteq \sN_{-\fracdmm}\subseteq \sN_{0}^*$.  Fix
  a good lattice $R$ which contains $\sN_0$.
  Define
  \begin{equation}\label{eq:NLattice.def}
    \sL_{\frac{i}{2m}} :=
    \begin{cases}
      \sN_{\frac{i-1}{2m}}+R^* & \text{when } 
      - \half < \frac{i}{2m}\leq 0;\\
      \sL_{\frac{-i+1}{2m}}^* = \sN_{\frac{-i}{2m}}^* \cap R &  \text{when }
      0 < \frac{i}{2m}\leq \half.  
    \end{cases}
  \end{equation}
  Observe that
  $\innv{\sN_{-\half}}{R^*} = \varpiD^{-1} \innv{\sN_{\half}}{R^*} \subseteq
  \varpiD^{-1} \innv{\sN_0}{R^*}\subseteq \varpiD^{-1} \innv{R}{R^*} = \fooD$.
  Therefore we have
  $\innv{\sL_{\frac{-m+1}{2m}}}{\sL_{\frac{-m+1}{2m}}}
  =\innv{\sN_{-\half}+R^*}{\sN_{-\half}+R^*}\subseteq \fooD$
  which is equivalent to
  \[
\sL_{\frac{-m+1}{2m}} \subseteq (\sL_{\frac{-m+1}{2m}})^*\varpiD^{-1} =
  \sL_{\frac{m}{2m}}\varpiD^{-1}.
  \]
  Hence \cref{eq:NLattice.def} determines a lattice function $\sL$ such
  that $\Jump(\sL)\subseteq \fracdmm \bZ$. Moreover, $\sL$ is self-dual since
  $\sL_{\frac{i}{2m}}^* = \sL_{\frac{-i+1}{2m}}$ by definition.
  
  Note that
 \begin{align*}
   \sL_{\frac{i+1}{2m}} &= \sN_{\frac{i}{2m}} + R^* \supseteq \sN_{\frac{i}{2m}}
   & &\text{when} &-\half \leq &\frac{i}{2m}< 0;\\
   \sL_{\frac{i+1}{2m}} &= \sL_{\frac{-i}{2m}}^* = \sN_{\frac{-i-1}{2m}}^*\cap
                          R\supseteq \sN_{\frac{i}{2m}}\cap \sN_0 =
                          \sN_{\frac{i}{2m}} & &\text{when} &0\leq &\frac{i}{2m} <\half.
 \end{align*}
 Therefore $\sL_{t+\fracdmm} \supseteq \sN_t$ for all $t\in \bR$ by the definition of
 $\sL$.
\end{proof}

\begin{lemma}[cf. {\cite[Proposition~10.5]{Pan02D}}]\label{lem:PD.B}
Suppose $x'$ is a point in $\BTB{G'}$ and $j$ is a positive integer. Then 
\[
\sS_{\sB_{x,x',-j/2m}}^{G_{x,j/m}} \subseteq \sum_{y\in
  \BTB{G'}_{2m}}\sS_{\sB_{x,y,\frac{-j+1}{2m}}}.
\]
\end{lemma}
\begin{proof}
Let $r := j/m$ and $s := r/2 = j/2m$.
Note that $G_{x,r}\subseteq G_{x,s^+}$ and 
\[
\sS_{\sB_{x,x',-j/2m}} = \bigoplus_{w} \sS(\sB_{x,x',0})_w 
\]
as $G_{x,r}$-module where $w$ is running over representatives of
$\sB_{x,x',{-s}:0}$.  By \Cref{rmk:Sact.2} of \Cref{lem:Saction}, the summand
$\sS(\sB_{x,x',0})_w$ is $G_{x,r^+}$-isotypic and $\exp(X)\in G_{x,r}$ acts by the
scalar $\psi(\bB(X,w^\mstar w))$.  Now fix a $w$ such that
$\sS(\sB_{x,x',0})_w^{G_{x,r}} \neq 0$. Since $\psi|_{\fooF}$ is non-degenerate,
$\bB(\fgg_{x,r},w^\mstar w)=0$, i.e.  $M(w) = w^\mstar w \in
\fgg_{x,-r^+}$. Clearly, $\sS(\sB_{x,x',0})_w \subseteq \sS^{G_{x,r}}$.

Define $\sN_{t} := (w+\sB_{x,x',0})\sL_{x,t+s}$. It is clear that $\sN_t$ is an
$\fooD$-module function in $V'$ and $\Jump(\sN_t) \subseteq \fracdmm\bZ$.  On the
other hand, $w_1^\mstar w_2 \equiv M(w) \equiv 0 \pmod{\fgl(V)_{x,-r^+}}$ for
any $w_1,w_2\in w+\sB_{x,x',0}$ and
$\fgl(V)_{x,-r^+} = \fgl(V)_{x,-r+\fracmm}$. Therefore,
\[
\begin{split}
\innvp{\sN_{t_1}}{\sN_{t_2}} &\subseteq
\innv{\sL_{x,t_1+s}}{\fgl(V)_{x,-r+\fracmm}\cdot 
  \sL_{x,t_2+s}}\\
&\subseteq \innv{\sL_{x,t_1+s}}{\sL_{x,t_2-s+\fracmm}}  \subseteq
\fppD^{\ceil{t_1+t_2+\fracmm}}.
\end{split}
\]
By \Cref{lem:SLattice}, there is a self-dual lattice function $\sL'_{y}$ such
that $\sN_{t}\subseteq \sL'_{y,t+\fracdmm}$.  Hence we have
\[
w+\sB_{x,x',0} \subseteq \bigcap_{t\in \fracdmm\bZ}
\Hom_{\fooD}(\sL_{x,t+s},\sL'_{y,t+\fracdmm}) = \sB_{x,y,-s+\fracdmm}
\]
and
$\sB_{x,x',0} = \set{w_1-w_2| w_1,w_2\in w+\sB_{x,x',0}} \subseteq
\sB_{x,y,-s+\fracdmm}$.
This means $\sS(\sB_{x,x',0})_w \subseteq \sS_{\sB_{x,y,\frac{-j+1}{2m}}}$ and
proves the lemma.
\end{proof}

\subsubsection{Proof of {\Cref{thm:IPD}}} 
The ``$\supseteq$'' direction is obvious. We now prove the ``$\subseteq$''
direction.  Let $r = k/m$ and fix any $x'\in \BTB{G'}_{2m}$.  For each integer
$j>k$, we have
\[
\begin{split}
\sS_{\sB_{x,x',\frac{-j}{2m}}}^{G_{x,r^+}} & =
(\sS_{\sB_{x,x',\frac{-j}{2m}}})^{G_{x,\frac{k+1}{m}}}
= ((\sS_{\sB_{x,x',\frac{-j}{2m}}})^{G_{x,\frac{j}{m}}})^{G_{x,\frac{k+1}{m}}}\\
&\subseteq \sum_{y\in
  \BTB{G'}_{2m}}(\sS_{\sB_{x,y,\frac{-(j-1)}{2m}}})^{G_{x,\frac{k+1}{m}}}
\subseteq \cdots 
 \subseteq \sum_{y\in \BTB{G'}_{2m}} \sS_{\sB_{x,y,\frac{-k}{2m}}} 
\end{split}
\]
by \Cref{lem:PD.B}.
Now the theorem follows from the fact that
$\sS = \bigcup_{j\geq k} \sS_{\sB_{x,x',\frac{-j}{2m}}}$. \qed

\subsubsection{}
By the argument in the proof of \cite[Theorem~6.6]{Pan02D}, ``depth
preservation'' of theta correspondence is an immediate consequence of
\Cref{thm:IPD}. Indeed let $\Phi \colon \sS \surj \pi\boxtimes \theta(\pi)$ be the
$\wtG\times \wtG'$-intertwining map. Suppose $\pi$ has depth $r$. Then there is
a certain $x\in \BTB{G}$ such that $\Jump(\sL_x)\in \bQ$ and
$\pi^{G_{x,r^+}}\neq 0$.  By \Cref{thm:IPD}, $\Phi(\sS_{\sB_{x,y,-r/2}}) \neq 0$
for some $y\in \BTB{G'}$ so $\left(\theta(\pi)\right)^{G'_{y,r^+}} \neq 0$.
Hence $\depth(\theta(\pi)) \leq \depth(\pi)$. Since the roles of $\pi$ and
$\theta(\pi)$ are symmetric, we have $\depth(\pi)\leq \depth(\theta(\pi))$ as
well which proves the ``depth preservation''.  


We remark that in proving his result \cite[Theorem~5.5]{Pan03}, Pan uses the
fact that an irreducible representation of a classical group of positive depth
has an unrefined minimal $K$-types of the form $(G_{\cL,r},\zeta)$ where $\zeta$
is a character of $G_{\cL,r:r^+}$ and $\cL$ is some regular small admissible
lattice chain.  See \cite[Proposition~3.4]{Pan03}.  By the result in this
appendix, this could be circumvented and we could replace ``a regular small
admissible lattice chain $\cL$ in $\cV$'' by ``a rational point in the building
of $\rU(\cV)$'' or simply ``a point in the building of $\rU(\cV)$'' (since
  unrefined minimal $K$-type always could be achieved at a rational point) in
the statement of \cite[Theorem~5.5]{Pan03}.

\end{document}